\documentclass[12pt]{amsart}

\usepackage{hyperref}
\hypersetup{
  colorlinks=true,
  linkcolor=blue,
  urlcolor=black,
  citecolor=black
}
\usepackage{amsmath,amsfonts,amsthm,amssymb,amsrefs,enumitem}
\usepackage{graphicx}
\usepackage{tikz}

\usepackage[margin=1in]{geometry}

\numberwithin{equation}{section}

\newtheorem{theorem}{Theorem}[section]

\newtheorem{Proposition}[theorem]{Proposition}
\newtheorem{proposition}[theorem]{Proposition}
\newtheorem{corollary}[theorem]{Corollary}
\newtheorem{Corollary}[theorem]{Corollary}
\newtheorem{lemma}[theorem]{Lemma}

\theoremstyle{definition}
\newtheorem*{definition*}{Definition}

\newtheorem{property}[theorem]{Property}

\newtheorem{Remark}[theorem]{Remark}

\newtheorem{definition}[theorem]{Definition}

\newtheorem{remark}[theorem]{Remark}

\def \R {\mathbb{R}}

\def \Z {\mathbb{Z}}

\def \ep {\varepsilon}

\def \O {\Omega}

\def \a {\alpha}

\def \grad {\nabla}
\def \J {\mathcal{J}}
\def \HolderConst {\kappa}

\def\Xint#1{\mathchoice
{\XXint\displaystyle\textstyle{#1}}
{\XXint\textstyle\scriptstyle{#1}}
{\XXint\scriptstyle\scriptscriptstyle{#1}}
{\XXint\scriptscriptstyle\scriptscriptstyle{#1}}
\!\int}
\def\XXint#1#2#3{{\setbox0=\hbox{$#1{#2#3}{\int}$ }
\vcenter{\hbox{$#2#3$ }}\kern-.6\wd0}}

\def\dashint{\Xint-}

\newcommand{\osc}{\mathop{\textup{osc}}}

\newcommand{\eref}[1]{(\ref{e.#1})}
\newcommand{\tref}[1]{Theorem \ref{t.#1}}
\newcommand{\lref}[1]{Lemma \ref{l.#1}}
\newcommand{\pref}[1]{Proposition \ref{p.#1}}
\newcommand{\cref}[1]{Corollary \ref{c.#1}}
\newcommand{\fref}[1]{Figure \ref{f.#1}}
\newcommand{\sref}[1]{Section \ref{s.#1}}
\newcommand{\partref}[1]{\ref{part.#1}}
\newcommand{\dref}[1]{Definition \ref{d.#1}}
\newcommand{\rref}[1]{Remark \ref{r.#1}}
\renewcommand{\Box}{\square}

\title[Two-phase free boundary minimizers in periodic media]{Regularity of two-phase free boundary minimizers in periodic media}
\author[F. Abedin and W. M. Feldman]{Farhan Abedin and William M Feldman} 
\address{Department of Mathematics, Lafayette College, Easton, PA 18042}
\email{abedinf@lafayette.edu}
\address{Department of Mathematics, University of Utah, Salt Lake City, UT 84112}
\email{feldman@math.utah.edu}

\date{\today}
\begin{document}

\begin{abstract}
We study the regularity of minimizers of a two-phase energy functional in periodic media. Our main result is a large scale Lipschitz estimate.  We also establish improvement-of-flatness for non-degenerate minimizers, which is a key ingredient in the proof of the Lipschitz estimate. As a consequence, we obtain a Liouville property for entire non-degenerate minimizers.
\end{abstract}

\maketitle
{\hypersetup{linkcolor=black}
\tableofcontents
}

\section{Introduction}

We consider minimizers of the two-phase  energy functional
\begin{equation}\label{e.energy-intro}
    \mathcal{J}(u, U) := \int_{U} a(x) \grad u \cdot \grad u + Q_+(x)^2 {\bf 1}_{\{u > 0 \}} + Q_-(x)^2 {\bf 1}_{\{u \leq 0 \}} \ dx.
\end{equation}
The matrix field $a(x)$ and scalar coefficients $Q_\pm(x)$ are $\Z^d$-periodic, and satisfy the standard ellipticity hypotheses
\[ \Lambda^{-1} \text{ Id} \leq a(x) \leq \Lambda \text{ Id} \ \hbox{ and } \ \Lambda^{-1} \leq Q_\pm(x) \leq \Lambda \ \hbox{ for all } \ x\in \R^d.\]

Functionals like \eref{energy-intro} arise in several applications, most notably in equilibrium models of two-phase fluid flow and capillary surfaces in inhomogeneous media with microstructure. Critical points satisfy the Euler-Lagrange equation
\begin{equation}\label{e.periodic-PDE-intro}
   \begin{cases}
      - \grad \cdot (a(x)\grad u) = 0 &\hbox{in } \ \{u \neq 0\} \cap U\\
      |\grad_{a} u_+|^2 - |\grad_{a} u_-|^2 = Q_+(x)^2 - Q_-(x)^2 &\hbox{on } \partial \{u>0\} \cap U.
     \end{cases} 
\end{equation}
The sets $\{u > 0\}$ and $\{u < 0\}$ represent the equilibrium configurations of two immiscible fluids, while the free boundary $\partial \{u > 0\}$ describes the interface between the fluids. The jump condition for the gradient at $\partial \{u > 0\}$ arises from Bernoulli's law for perfect (incompressible and irrotational) fluids.

In the limit as the size of the domain $U$ becomes large compared to the unit scale associated with the periodicity of the coefficients, minimizers of \eref{energy-intro} converge to minimizers of the homogenized energy functional
\[\mathcal{J}_0(u, U) := \int_{U} \bar{a} \grad u \cdot \grad u + \langle Q_+^2\rangle {\bf 1}_{\{u > 0 \}} + \langle Q_-^2\rangle {\bf 1}_{\{u \leq 0 \}} \ dx. \]
We will assume that $Q_\pm$ satisfy 
\[\langle Q_+^2 \rangle  > \langle Q_-^2 \rangle\]
where $\langle \cdot \rangle$ denotes the period cell averages. The constant matrix $\bar{a}$ is the standard homogenized field associated with the periodic medium $a(x)$. With no loss of generality, we may assume $\bar{a} = \text{id}$ and $\langle Q_+^2\rangle  - \langle Q_+^2\rangle  = 1$ (see \rref{normalizationremark} below) and so $\J_0$ is essentially the same functional studied by Alt, Caffarelli, and Friedman in their influential work \cite{Alt-Caffarelli-Friedman-1984}. The Euler-Lagrange equation for $\J_0$, under the above normalizations, is the classical two-phase Bernoulli free boundary problem 
\begin{equation}\label{e.homogenized-two-phase-PDE}
     \begin{cases}
      - \Delta u = 0 &\hbox{in } \ \{u\neq0\} \cap U\\
      |\grad u_+|^2 - |\grad u_-|^2 = 1 &\hbox{on } \partial \{u>0\} \cap U.
     \end{cases}
\end{equation}

The regularity for $\J_0$ minimizers is one of the hallmark achievements in the theory of elliptic free boundary problems. A key result is the Lipschitz estimate for $\J_0$ minimizers, which is a consequence of the so-called Alt-Caffarelli-Friedman (ACF) monotonicity formula \cite{Alt-Caffarelli-Friedman-1984}*{Lemma 5.1}, see also \cite{Caffarelli-Kenig-Jerison-2002}. Another important set of results are that flat free boundaries are Lipschitz \cites{Alt-Caffarelli-Friedman-1984, Caffarelli-Harnack-II} and Lipschitz free boundaries are $C^{1,\alpha}$ \cite{Caffarelli-Harnack-I}, where flatness can be measured with respect to the two plane solutions of \eref{homogenized-two-phase-PDE}
     \[\Phi_\alpha(t) := \sqrt{1+\alpha^2}\max\{0,t\} + \alpha \min\{0,t\} \qquad \a> 0.\]
The ACF monotonicity formula can also be used to categorize entire minimizers of $\J_0$; indeed, the ones with non-trivial negative phase are exactly of the form $u(x) = \Phi_{\alpha}(x\cdot \nu)$ for some $\a > 0$ and unit vector $\nu$, while the remaining entire minimizers have only one-phase and are only fully classified in low dimensions.

The Lipschitz estimate, flat-implies-$C^{1,\alpha}$ regularity result, and the classification of entire minimizers are all essential ingredients in the blow-up analysis that yields $C^{1,\alpha}$ regularity of the free boundary for $\J_0$ minimizers. The monograph \cite{CaffarelliSalsa} contains a detailed exposition of the original ideas involved; see \cites{DeSilva, DeSilvaFerrariSalsa-APDE2014, DeSilvaSavinLipschitz, DeSilva-Savin-Global} for a more recent versatile approach based on improvement-of-flatness. Several far-reaching extensions and applications of these regularity results have been obtained over the past four decades; without claiming completeness, we point out some of the major developments for nonlinear and inhomogeneous problems \cites{Alt-Caffarelli-Friedman-Quasilinear, Caffarelli-DeSilva-Savin-2018, DeSilva, DeSilva-Savin-Global, Dipierro-Karakhanyan-2018}, functionals whose minimizers have a non-trivial zero-phase \cites{DePhilippis-Spolaor-Velichkov-2021}, and almost minimizers \cites{David-Toro-2015, David-Engelstein-Toro-2019, David-Engelstein-Garcia-Toro-2021, DeSilvaSavinAM}. We note specifically that, in the case of almost minimizers, there is no Euler-Lagrange equation, so novel approaches were necessary to circumvent this issue.

The primary motivation of the present work is to establish results analogous to those described above, but in the periodic setting. Our central result is a large scale Lipschitz estimate for minimizers of the inhomogeneous functional \eref{energy-intro}. We note that in elliptic homogenization, large scale estimates are the natural proxies for pointwise estimates; indeed, the periodicity of the medium is, by itself, insufficient to obtain Lipschitz estimates below the microscopic scale. A key ingredient in the proof of the large scale Lipschitz estimate is a large scale improvement-of-flatness result for minimizers of \eref{energy-intro} that are flat with respect to a non-degenerate two-plane solution. 

At a high level, the aforementioned results are obtained by viewing minimizers of \eref{energy-intro} as almost minimizers of the homogenized energy $\J_0$, though with a slight twist, as the typical algebraic decay at small scales of the error in almost minimality is replaced by algebraic decay at large scales. As such, the regularity theory for non-degenerate almost minimizers of $\J_0$, which we also develop in this work, plays a crucial role. We emphasize that our results are valid \emph{only} for minimizers. The large scale behavior of general solutions of the Euler-Lagrange PDE \eref{periodic-PDE-intro} is far more complicated due to slope pinning; we refer to \cite{Feldman-2021} for a more thorough discussion on this matter. 

As is common in homogenization theory, we will measure large scale regularity using the averaged $L^2$-norm $\| \cdot \|_{\underline{L}^2(B_r)}$ (see Subsection \ref{subsec:notations} below for the definition). Moreover, since we are only interested in regularity at large scales, the regularity of the coefficients $a$ and $Q$ below the unit scale will be irrelevant.

We proceed to state our main results, referring to the body of the paper for more precise statements.
\begin{theorem}[Large scale Lipschitz estimate for $\J$ minimizers]\label{t.u-lipschitz}
Suppose that $u$ minimizes $\mathcal{J}$ over $u + H^1_0(B_{R})$. Then
\[\|\grad u\|_{\underline{L}^2(B_r)} \leq C(1+\| \grad u\|_{\underline{L}^2(B_{R})}) \ \hbox{for all } \ 1 \leq r \leq R.\]
The constant $C$ depends only on the universal parameters $\Lambda$ and $d$.
\end{theorem}

The proof of \tref{u-lipschitz} relies crucially on the following improvement of flatness result down to unit scale for the free boundary of non-degenerate minimizers. The notion of flatness is with respect to the two-plane solutions $\Phi_{\a}(t)$, see Section \ref{sec:large-scale-regularity-of-J-minimizers} for the precise definition.

\begin{theorem}[Improvement of flatness for non-degenerate $\J$ minimizers]
Suppose $u$ minimizes $\mathcal{J}$ in $B_R$ and $u(0) = 0$. For all $\beta \in (0,1)$ and $\underline{\alpha}>0$ there exists $\bar\delta>0$ so that if
    \[\frac{1}{R}\textup{flat}(u,B_R) \leq \delta  \leq \overline{\delta} \ \hbox{ with } \ \alpha(u,B_R) \geq \underline{\alpha}\]
    then
    \[\frac{1}{r}\textup{flat}(u,B_r)   \leq C(\delta (\tfrac{r}{R})^\beta+(\tfrac{1}{r})^\omega) \quad \hbox{ for all } \ 1\leq r \leq R/2.\]
    Moreover,
    \[\osc_{1 \leq s \leq r} \nu(u,B_s) +  \osc_{1 \leq s \leq r} \log \alpha(u,B_{s})  \leq C(\delta (\tfrac{r}{R})^\beta+(\tfrac{1}{r})^\omega) \quad \hbox{ for all } \ 1 \leq r \leq R/2.\]
\end{theorem}
In the above, $\omega \in (0,1)$ is universal, depending only on $d,\Lambda$ and $C \geq 1$ depends on $d$, $\Lambda$, $\beta \in (0,1)$, and $\underline{\alpha}>0$.

\begin{remark} As a consequence of the improvement of flatness result, one can also show that flat $\mathcal{J}$ minimizers in $B_r$ have (large scale) Lipschitz free boundary in $B_{r/2}$, see \cite{Feldman}*{Theorem 1.2 and Corollary 6.5}.
\end{remark}

An interesting byproduct of the Lipschitz estimate is the following Liouville property of minimizers on $\R^d$ with non-trivial negative phase. See \sref{liouville} below for the fully detailed statement and proof.
\begin{theorem}[Liouville property of minimizers on $\R^d$]\label{t.liouville}
    Suppose that $u \in H^1_{loc}(\R^d)$ minimizes $\mathcal{J}$ with respect to compact perturbations with $u(0) = 0$, 
    \[  \liminf_{r \to \infty} \| \grad u\|_{\underline{L}^2(B_r)} < + \infty, \ \hbox{ and } \  \inf_{x \in \R^d}\liminf_{r \to \infty} \frac{u(rx)}{r} <0. \]
    Then there is $e \in \mathbb{S}^{d-1}$ and $\alpha > 0$ so that $u$ grows sublinearly away from $\Phi_\alpha(e \cdot x)$.
\end{theorem}

The proofs of our results are based on several foundational ideas from elliptic homogenization, free boundary theory, and regularity for almost-minimizers in the calculus of variations. What follows is a more detailed description of the arguments involved. For the sake of readability, we have presented our results in two, mostly independent, parts. 

In Part 1 of the paper, the main goal is to establish a one-step improvement of flatness result for non-degenerate $\J_0$ approximate minimizers, \tref{onestep-improve-flatness} below. The philosophy of the proof can be traced back to Savin's Harnack inequality approach to regularity of flat viscosity solutions for fully nonlinear elliptic equations \cite{Savin-2007}, as well as De Silva and Savin's approach to regularity for functions that satisfy a viscosity solution property only at certain length scales \cite{DeSilva-Savin-QuasiHarnack}. This strategy has recently been employed in the context of free boundary problems in \cites{DeSilva, DeSilvaFerrariSalsa-APDE2014}, which prove improvement-of-flatness for Lipschitz minimizers (both one and two phase), and in \cites{DeSilvaSavinAM,Feldman-2021}, which establish Lipschitz regularity and improvement of flatness for one-phase almost minimizers. Our approach is inspired by these works. 

The main steps in Part 1 are as follows:
\begin{itemize}
    \item Employ energy difference formulae to establish an approximate viscosity solution property for approximate minimizers of $\J_0$; this is the content of Section \ref{sec:approximate-viscosity}.  
    \item Use the approximate viscosity solution property to prove a Harnack-type inequality and H\"older estimates for approximate minimizers near the free boundary, see \tref{holder-est-w}.
    \item Carry out a compactness-contradiction argument to establish one-step improvement of flatness for non-degenerate approximate minimizers, see \tref{onestep-improve-flatness}.
\end{itemize}
The approximate viscosity solution property is a distinguishing attribute of the almost minimizer setting. We provide a seemingly new proof of this property based on an interesting energy difference formula, see \lref{intbyparts2} and \lref{approximate-supersoln-cond} below. Another noteworthy difference between existing works and ours is that we neither assume nor establish Lipschitz regularity of the approximate minimizer. However, given the applications we have in mind in Part 2 to large scale regularity of minimizers for the inhomogeneous functional $\J$, which are known to be H\"older continuous (see \sref{regularityofJminimizer}), we impose an a priori H\"older regularity assumption on the $\J_0$ approximate minimizers studied in this part of the paper.

In Part 2 of the paper, the main goal is to prove quantitative homogenization and large scale regularity results for $\J$ minimizers. The philosophy here is based on the approach to large scale regularity for standard elliptic homogenization problems developed by Armstrong and Smart \cite{Armstrong-Smart}, which yields quantitative versions of seminal results due to Avellaneda and Lin \cite{Avellaneda-Lin}. The proof of the large scale Lipschitz estimate takes inspiration from the strategy introduced by De Silva and Savin for two-phase $\J_0$-minimizers \cite{DeSilvaSavinLipschitz}, which avoids the use of ACF-type monotonicity formulas. 

The main steps in Part 2 are as follows:
\begin{itemize}
    \item[$\bullet$] Establish a (sub-optimal) quantitative homogenization result for the energy of $\J$ minimizers, see \tref{energy-hom-error}.
    \item[$\bullet$] Apply \tref{onestep-improve-flatness} to flat non-degenerate $\J$ minimizers (which, by \tref{energy-hom-error}, are large scale $\J_0$ almost minimizers) and iterate down to the microscopic (unit) scale to get $C^{1,\alpha}$-type improvement-of-flatness, see \tref{improvement-of-flatness-to-microscale}.
    \item[$\bullet$] Carry out a dichotomy argument for the size of the slope near the free boundary to obtain a large scale Lipschitz regularity result for $\J$ minimizers, see \tref{u-large-scale-lipschitz}.
\end{itemize}

It is worth elaborating on the ideas involved in the proof of the large scale Lipschitz estimate, as the result holds for $\J$ minimizers that are not necessarily flat or non-degenerate. The first observation, which has also been crucial in several related works in the almost minimizer literature, is that the $a$-harmonic replacement $v$ of a $\J$ minimizer $u$ yields a very good approximation, in the sense that 
    \[\|\grad u - \grad v\|_{\underline{L}^2(B_r)} \leq C,\]
for some constant $C$ independent of $r$ (see \lref{a-harmonic-replacement}). Such an estimate allows us to control $\grad u$ using $\grad v$ up to a universal error that is independent of $r$ and does not grow even when the slope of $u$ is large. Next, we capitalize on the fact that $a$-harmonic functions satisfy large scale $C^{1,\beta}$-estimates (with corrector) proved by Avellaneda and Lin \cite{Avellaneda-Lin}. This leads to the following dichotomy, based on the slope of the $a$-harmonic replacement:
    \begin{itemize}
        \item Case (i): If $\|\grad v\|_{\underline{L}^2(B_R)}$ is small compared to $\|\grad u\|_{\underline{L}^2(B_R)}$, then there exists $\eta < 1$ such that \[\|\grad u\|_{\underline{L}^2(B_{\eta R})} \leq \frac{1}{2}\|\grad u\|_{\underline{L}^2(B_R)}.\]
        \item Case (ii):  If $\|\grad v\|_{\underline{L}^2(B_R)}$ is comparable to $\|\grad u\|_{\underline{L}^2(B_R)}$, then the Avellaneda-Lin $C^{1,\beta}$ estimate for $\grad v$ implies $u$ is flat with respect to a two-plane solution at large scales. This is because two-plane solutions with large slope are similar to planes relative to the size of their slopes. In this case, an earlier consequence of improvement of flatness, \cref{flat-implies-Lipschitz-fcn}, implies 
         \[\|\grad u\|_{\underline{L}^2(B_r)} \leq C_0(1+\| \grad u\|_{\underline{L}^2(B_{R})}) \ \hbox{for all } \ 1 \leq r \leq \eta R.\]
    \end{itemize}
We refer to \pref{harmonic-replacement-dichotomy} for details. The two cases outlined above are then used in an iteration argument that yields Lipschitz regularity for $\J$ minimizers down to unit scale.

As mentioned earlier, an effective and time-honored approach to proving Lipschitz regularity for minimizers of two-phase problems makes use of subtle monotonicity formulas \cites{Alt-Caffarelli-Friedman-1984, Caffarelli-Kenig-Jerison-2002, Matevosyan-Petrosyan-2011}. Such formulas can be adapted to the setting of almost minimizers as well \cites{David-Toro-2015, David-Engelstein-Toro-2019, David-Engelstein-Garcia-Toro-2021}. However, it was not evident to us how analogous monotonicity formulas can be used effectively in the periodic setting. As such, we have opted to employ techniques such as those introduced in \cites{Dipierro-Karakhanyan-2018, DeSilvaSavinAM, DeSilvaSavinLipschitz}, which are able to bypass the use of monotonicity formulas. Other interesting approaches to proving Lipschitz regularity also exist, but unfortunately are restricted to 2 dimensions \cites{Caffarelli-DeSilva-Savin-2018, Spolaor-Velichkov}. 

The idea that elliptic and parabolic problems in heterogeneous media could inherit regularity, at large scales, from a homogenized problem originates in the work of Avellaneda and Lin \cite{Avellaneda-Lin}. More recently, Armstrong and Smart \cite{Armstrong-Smart} revisited the ideas in \cite{Avellaneda-Lin} and found a new, more quantitative, proof inspired by the classical Morrey-Campanato approach to Schauder theory. The Armstrong-Smart approach is amenable to homogenization in random environments and has since been applied successfully in a variety of other settings as well; we refer to the monographs \cites{Armstrong-Kuusi-Mourrat, Shen} for a review of recent developments. 

The literature is somewhat limited when it comes to large scale regularity for free boundary problems in periodic media. Several works have focused on identifying and understanding the nature of the limiting homogenization problem \cites{AttouchPicard1,DalMasoLongo,CioranescuMurat,CodegoneRodrigues,Kim,CaffarelliLee,CaffarelliMellet09,FeldmanSmart,Feldman-2021,CaffarelliLeeMellet}. There are some large scale regularity results on the related, but qualitatively different, obstacle problem in oscillatory settings \cites{AleksanyanKuusi2024,abedin2023quantitative}; in particular, \cite{AleksanyanKuusi2024} is one of the first papers where an Avellaneda-Lin-type large scale regularity theory is proved in the context of a free boundary problem. The work that is closest in spirit to ours is that of the second author \cite{Feldman}, who proved large scale regularity results for the one-phase problem in periodic media. However, the approach in the present work necessarily deviates from that in \cite{Feldman}, where large scale Lipschitz regularity can be proved at the outset and used to prove non-degeneracy of the solution; improvement of flatness does not play any role in the Lipschitz estimate. It should be noted that while there are several proofs of Lipschitz regularity in the one-phase setting (see \cite{VelichkovBook}*{Chapter 3}), the approach due to De Silva and Savin \cite{DeSilvaSavinAM} for one-phase almost minimizers was most adaptable to the proof of the Lipschitz estimate for the periodic one-phase problem. Compared to \cite{Feldman}, we prove large scale Lipschitz regularity of non-degenerate two-phase minimizers as a \emph{consequence} of improvement of flatness, also following an approach due to De Silva and Savin \cite{DeSilvaSavinLipschitz}.

Finally, we comment that the homogenization theory of free boundaries, as we study here, is closely connected with the theory of minimals of variational problems on the torus as studied by Moser \cite{Moser} and Bangert \cite{Bangert1}. In that theory, entire minimizers without self-intersections are known to be plane-like (i.e. staying within a constant distance from a fixed plane). Moser and Struwe \cite{MoserStruwe} applied Avellaneda-Lin-type theory to obtain a Liouville theorem when the medium is laminar, i.e. with no dependence on the $u$-variable, and posed the general question of whether all linearly growing entire minimizers are plane-like (see \cite{MoserStruwe}*{Question 1) on Page 4}). This question remains unanswered at present. Our Liouville result in \tref{liouville}, along with a similar result in \cite{Feldman}*{Theorem 6.7}, makes progress in this direction for a related interface model without any laminarity assumption.

\subsection{Notation and Conventions}\label{subsec:notations}

\begin{enumerate}[label = $\bullet$]
\item Balls are denoted $B_r$ (if the center does not need denotation) or $B_r(x)$. Boxes are denoted
\[ \Box_r = [-r/2,r/2)^d \ \hbox{ and } \ \Box_r(x) = x + \Box_r.\]
\item The Lebesgue measure of a measurable set $U$ is denoted $|U|$.  This notation may also be used for lower dimensional measures if the Hausdorff dimension of the set is clear; for instance, $|\partial B_r|$ denotes the surface area of the sphere of radius $r$.
\item Averaged integrals and $L^p$ norms are denoted 
\[ \dashint_U f\  dx = \frac{1}{|U|} \int_U f \ dx \quad \hbox{and} \quad \|f\|_{{\underline{L}}^p(U)} = \left(\frac{1}{|U|} \int_U |f|^p \ dx\right)^{1/p}.\]
\item Period cell averages are denoted $\left\langle \cdot \right\rangle$.
\end{enumerate}

\begin{remark}\label{r.normalizationremark}
    We can achieve the normalization
    \[\bar{a} = \textup{id} \ \hbox{ and } \ q_+^2 - q_-^2 = 1\]
    for the functional $\J_0$ via the transformation
    \[ u \mapsto  \frac{1}{\sqrt{q_+^2 - q_-^2}}u(\bar{a}^{1/2}x)\]
    with 
    \[\hat{\mathcal{J}}_0(v, U) := \int_{U} |\grad v|^2 + \frac{q_+^2}{q_+^2 - q_-^2} {\bf 1}_{\{v > 0 \}} + \frac{q_-^2}{q_+^2 - q_-^2} {\bf 1}_{\{v \leq 0 \}} \ dx\]
    since, with $v(x) :=\frac{1}{\sqrt{q_+^2 - q_-^2}}u(\bar{a}^{1/2}x)$
    \[\hat{\mathcal{J}}_0(v, U) = \frac{\textup{det}(\bar{a})^{1/2}}{\sqrt{q_+^2 - q_-^2}}\mathcal{J}_0(u,\bar{a}^{1/2}U).\]
    This transformation, when applied to $\mathcal{J}$, does change the periodicity lattice of the coefficient fields by a linear transformation. The only substantive effect of this in the proofs is to modify the diameter of period cells that appears as a constant in some places, so we continue to write $\Z^d$ for the transformed lattice.
\end{remark}

\subsection{Acknowledgments} W. Feldman was partially supported by the NSF Grant DMS-2407235, and also benefitted from the environment provided by the NSF RTG Award DMS-2136198. F. Abedin acknowledges support from NSF Grant DMS-2246611.

\part{Improvement of flatness for $\mathcal{J}_0$ approximate minimizers.}

Our goal in this part of the paper is to develop an (approximate) viscosity solution theory and establish one-step improvement-of-flatness for approximate minimizers of the homogenized functional $\J_0$. These results will be used in Part 2 to prove large scale regularity of minimizers of the inhomogeneous functional $\J$. 

\section{Approximate energy minimality and approximate viscosity solutions}\label{sec:approximate-viscosity}

In this section we will describe approximate viscosity solution properties satisfied by states with close to minimal energy for the constant coefficient two-phase functional. Specifically, recall the functional
\begin{equation}\label{e.constant-coeff-energy}
    \mathcal{J}_0(u, U) := \int_{U} |\grad u|^2 + q_+^2 {\bf 1}_{\{u > 0 \}} + q_-^2 {\bf 1}_{\{u \leq 0 \}} \ dx 
\end{equation}
where $q_\pm >0$ and satisfy the normalization $q_+^2 - q_-^2 = 1$.  The Euler-Lagrange equation associated with the minimization of $\mathcal{J}_0$ is the two-phase free boundary problem
     \begin{equation}\label{e.two-phase-PDE}
     \begin{cases}
      \Delta u = 0 &\hbox{in } \ \{u>0\} \cup \{u<0\}\\
      |\grad u_+|^2 - |\grad u_-|^2 = 1 &\hbox{on } \partial \{u>0\}.
     \end{cases}
     \end{equation}
     We begin by explaining how to interpret the equation \eref{two-phase-PDE} in the viscosity sense in two equivalent ways.  

     First we introduce some useful notation. For $\alpha_\pm\geq0$ we define
     \[\Phi_{\alpha_\pm}(t) := \alpha_+\max\{0,t\} + \alpha_- \min\{0,t\} \ \hbox{ for } \ t \in \R.\]
     When $\alpha_\pm > 0$ we can also define the inverse map
     \[\Psi_{\alpha_\pm}(t) := \frac{1}{\alpha_+}\max\{0,t\} + \frac{1}{\alpha_-} \min\{0,t\} \ \hbox{ for } \ t \in \R.\]
     Note that $\Phi_{\alpha_\pm}(x_d)$ is a minimizer of \eref{constant-coeff-energy} and a solution of \eref{two-phase-PDE} if and only if $\alpha_\pm$ are related by
\begin{equation}\label{e.alpha-pm-relation}
    \alpha_+^2 - \alpha_-^2= 1.
\end{equation}     
     In the case when $\alpha_\pm$ satisfy the relation \eref{alpha-pm-relation} we will frequently attempt to reduce clutter and parametrize by $\alpha = \alpha_- \geq 0$, abusing notation to call
     \[\Phi_\alpha(t) := \sqrt{1+\alpha^2}\max\{0,t\} + \alpha \min\{0,t\} \ \hbox{ for } \ t \in \R.\]
     When $\alpha>0$, we similarly denote $\Psi_\alpha$ to be the inverse of $\Phi_\alpha$.

     \begin{definition}
     Suppose that $u$ is harmonic in $\{u\neq 0\}$.
     \begin{itemize}
    \item We say that $u$ is a viscosity subsolution of \eref{two-phase-PDE} at a point $x_0 \in \partial \{u>0\}$ if for all $C^2$ functions $P$ satisfying $\grad P(x_0) = n$ a unit vector, $\Delta P(x_0) < 0$, and $\Phi_{\a_{\pm}} \circ P$ touches $u$ from above at $x_0$, then
     \[\alpha_+^2-\alpha_-^2 \geq 1. \]
   
     \item We say that $u$ is a viscosity supersolution of \eref{two-phase-PDE} at a point $x_0 \in \partial \{u>0\}$ if for all $C^2$ functions $P$ satisfying $\grad P(x_0) = n$ a unit vector,  $\Delta P(x_0) > 0$, and $\Phi_{\a_{\pm}} \circ P$  touches $u$ from below at $x_0$, then
     \[\alpha_+^2-\alpha_-^2 \leq 1. \]
     \item We say that $u$ is a viscosity solution of \eref{two-phase-PDE} if it is both a sub and supersolution.
     \end{itemize}
     \end{definition}

\subsection{Energy difference formulae}
We begin the section with an essential formula which quantifies the first variation of the energy. This formula will be used to show that approximate minimizers of the energy satisfy an approximate viscosity solution property. We believe this formula is new in application to two-phase free boundary problems and quite useful; it is inspired by formulas previously derived in \cite{DeSilvaSavinAM}*{Lemma 4.1} and in \cite{FeldmanKimPozar}*{Appendix A}.

For any function $v$, we write
$(v)_+ = v {\bf 1}_{\{v > 0 \}}$ and $(v)_- = v {\bf 1}_{\{v \leq 0 \}}$ so that $v=(v)_+ + (v)_-$.

\begin{lemma}\label{l.intbyparts2}
    Suppose that $v_0,v_1 \in H^1(U) \cap C(U)$ with $v_0 \leq v_1$ and $v_0 = v_1 \geq 0$ on $\partial U$.  Call $\Omega_j = \{v_j>0\} \cap U$.
    \begin{itemize}
    \item[(i)] If $v_1$ satisfies $\Delta v_1 \geq \mu>0$ in the $H^1$ weak sense in $\{v_1 > (v_0)_+\}$ then
    \begin{equation}\label{e.intbyparts2-subsoln} \int_{\overline{\Omega}_0} |\grad v_0|^2 \ dx- \int_{\Omega_1} |\grad v_1|^2 \ dx \geq \int_{\Omega_1 \setminus \overline{\Omega}_0}|\nabla v_1|^2 \ dx + 2\mu \int_{\Omega_1} (v_1 -(v_0)_+) \ dx.\end{equation}
    \item[(ii)] If $v_0$  satisfies $\Delta v_0 \leq -\mu<0$ in the $H^1$ weak sense in $\{v_1 > (v_0)_+\}$ then
    \begin{equation}\label{e.intbyparts2-supersoln} \int_{\overline{\Omega}_0} |\grad v_0|^2 \ dx - \int_{\Omega_1} |\grad v_1|^2 \ dx \leq  \int_{\Omega_1 \setminus \overline{\Omega}_0}|\nabla {v_0}|^2 \ dx -  2\mu \int_{\Omega_1} (v_1 -(v_0)_+) \ dx
    \end{equation}
    \end{itemize}
    \end{lemma}

\begin{proof} 
 
For \eref{intbyparts2-subsoln}, we start with
\begin{align*}
\int_{\overline{\Omega}_0} |\grad v_0|^2 \ dx - \int_{\Omega_1} |\grad v_1|^2 \ dx &= \int_{\overline{\Omega}_0 } |\grad v_0|^2 - |\grad v_1|^2 \ dx + \int_{\Omega_1 \setminus \overline{\Omega}_0} - |\grad v_1|^2 \ dx
\end{align*}
Then we compute the first term on the right
\begin{align*}
 \int_{\overline{\Omega}_0} |\grad v_0|^2 - |\grad v_1|^2 \ dx &=  \int_{\overline{\Omega}_0 } |\grad v_1+\grad (v_0 - v_1)|^2-|\grad v_1|^2  \ dx \\
 &=\int_{\overline{\Omega}_0 } 2\grad v_1\cdot \grad(v_0 - v_1) + |\grad(v_0-v_1)|^2 \ dx \\
 &\geq \int_{\overline{\Omega}_0 } 2\grad v_1\cdot \grad(v_0 - v_1) \ dx\\
 &= \int_{\Omega_1} 2\grad v_1\cdot \grad ((v_0)_+ - v_1) \ dx +\int_{\Omega_1 \setminus \overline{\Omega}_0} 2 |\grad v_1|^2 \ dx\\
 &=\int_{\{v_1 > (v_0)_+\}} 2\grad v_1\cdot \grad ((v_0)_+ - v_1) \ dx +\int_{\Omega_1 \setminus \overline{\Omega}_0} 2 |\grad v_1|^2 \ dx\\
 &\geq 2\mu\int_{\Omega_1}  (v_1 -(v_0)_+) \ dx+\int_{\Omega_1 \setminus \overline{\Omega}_0} 2 |\grad v_1|^2 \ dx.
\end{align*}
The last inequality in the sequence above is due to $v_1$ being an $H^1$ subsolution in $\{v_1 - (v_0)_+>0\}$. Note that $v_1 - (v_0)_+$ is a valid test function for the subsolution condition since it is non-negative in $\{v_1 - (v_0)_+>0\}$ and belongs to $H^1_0(\{v_1 - (v_0)_+>0\})$.

For \eref{intbyparts2-supersoln}, we start with
\begin{align*}
\int_{\overline{\Omega}_0} |\grad v_0|^2 \ dx - \int_{\Omega_1} |\grad v_1|^2 \ dx &= \int_{\Omega_1 } |\grad v_0|^2 - |\grad v_1|^2 \ dx + \int_{\Omega_1 \setminus \overline{\Omega}_0} - |\grad v_0|^2 \ dx
\end{align*}
Then we compute the first term on the right
\begin{align*}
 \int_{\Omega_1} |\grad v_0|^2 - |\grad v_1|^2 \ dx &=  \int_{\Omega_1 } |\grad v_0|^2-|\grad v_0+\grad (v_1-v_0)|^2  \ dx \\
 &=\int_{\Omega_1 } -2\grad v_0\cdot \grad(v_1 - v_0) - |\grad(v_1-v_0)|^2 \ dx \\
  &\leq\int_{\Omega_1 } -2\grad v_0\cdot \grad(v_1 - v_0) \ dx \\
 &= \int_{\Omega_1} -2\grad v_0\cdot \grad(v_1 - ((v_0)_++(v_0)_-)) \ dx \\
 &=\int_{\Omega_1} -2\grad v_0\cdot \grad(v_1 - (v_0)_+) \ dx +\int_{\Omega_1 \setminus \overline{\Omega}_0} 2 |\grad v_0|^2 \ dx \\
 &=\int_{\{v_1 > (v_0)_+\}} -2\grad v_0\cdot \grad(v_1 - (v_0)_+) \ dx +\int_{\Omega_1 \setminus \overline{\Omega}_0} 2 |\grad v_0|^2 \ dx \\
 &\leq  - 2\mu\int_{\Omega_1} (v_1 - (v_0)_+) \ dx+ \int_{\Omega_1 \setminus \overline{\Omega}_0} 2 |\grad v_0|^2 \ dx.
\end{align*}
The last inequality in the sequence above is due to $v_0$ being an $H^1$ supersolution in $\{v_1 - (v_0)_+>0\}$. Note that $v_1 - (v_0)_+$ is a valid test function for the superharmonicity condition since it is  non-negative in $\{v_1 - (v_0)_+>0\}$ and belongs to $H^1_0(\{v_1 - (v_0)_+>0\})$.

\end{proof}
\subsection{Viscosity solution property of minimizers}  Minimizers of $\mathcal{J}_0$ are viscosity solutions of \eref{two-phase-PDE}.  This is a well known fact and proofs can be found in various places in the literature, for example see \cite{CaffarelliSalsa}.  We will present a (seemingly) new proof based on the energy difference formulas \lref{intbyparts2}.  In order to ease the presentation, we first show how the proof works in the case of minimizers, and then we will generalize to approximate minimizers below in \sref{viscosity-property-approximate}.  Since the sub and supersolution arguments are quite similar, but not exactly the same, we will show the subsolution argument for minimizers here, and then the supersolution argument below for approximate minimizers.

\begin{lemma}
If $u$ is a minimizer of $\mathcal{J}_0$ over $u + H^1_0(U)$ then $u$ is a viscosity solution of \eref{two-phase-PDE} in $U$.
\end{lemma}

\begin{proof} Since minimizers are continuous, for example see \lref{holder-FB} below, both $\{u>0\}$ and $\{u<0\}$ are open sets, and so $u$ is a minimizer for the Dirichlet energy in a positive radius ball around any $x_0 \in \{u \neq 0\}$ and so $u$ is harmonic in $\{u \neq 0\}$.

We prove only the viscosity subsolution condition at the free boundary. The supersolution condition can be proved by symmetry, since $-u$ minimizes a similar functional with the roles of $q_\pm$ swapped, and the argument below does not rely on $q_+^2 - q_-^2 >0$. 

 Suppose $P$ is a $C^2$ function with $\grad P(x_0) = n$ a unit vector, $\Delta P (x_0) < 0$, and such that $\Phi_{\alpha_\pm}\circ P$ touches $u$ from above strictly at $x_0 \in \partial \{u>0\}$. Let $\mu>0$ and consider the comparison function
\[ u_\mu(x)  := u(x) \wedge \Phi_{\alpha_\pm}(P(x) - \mu).\] 
Since the touching is strict $ \{u_\mu < u\}$ is contained in a ball of radius $o_\mu(1)$ around $x_0$ as $\mu \to 0$.  In particular we can assume that $\mu$ is small enough so that $P$ is superharmonic and has non-zero gradient in $\{u_\mu< u\}$.

Notice that, by set algebra, 
\[\Gamma_\mu := \{ u>0\} \setminus \{u_\mu \geq 0\} = \{ u_\mu < 0\} \setminus \{u \leq 0\}.\]
Since $\grad P \neq 0$ in $\{u_\mu < u\}$ then $\overline{\{u_\mu > 0\}} \cap \{u_\mu < u\} = \{u_\mu \geq 0\} \cap \{u_\mu < u\}$ so that we also have the formula
\begin{equation}\label{e.udelta-closure-technical}
    \{ u>0\} \setminus \overline{\{u_\mu > 0\}} = \Gamma_\mu.
\end{equation}

\begin{figure}
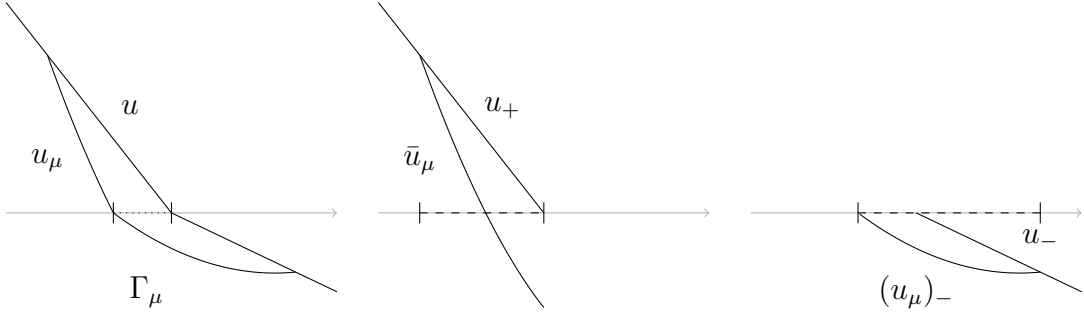

    \centering
    \begin{tikzpicture}[xscale = 1.1, yscale = 1.4]

    \draw[black!30!white,->] (-2,0) -- (2,0);
    \draw (-2,2) -- (0,0);
    \draw (0,0) -- (2,-.75);
    \input{figures/points1.tikz}
    \draw (\z,-.1) -- (\z,.1);
    \draw (0,-.1) -- (0,.1);
    \draw[dotted] (\z,0) -- (0,0);
    \node at (-.5,1) {$u$};
    \node at (-1.5,.5) {$u_\mu$};
    \node at (-.3,-.75) {$\Gamma_\mu$};

    \begin{scope}[shift={(4.5,0)}]
      \draw[black!30!white,->] (-2,0) -- (2,0);
    \draw (-2,2) -- (0,0);
    \input{figures/points2.tikz}

     \draw[dashed] (-1.5,0) -- (0,0);
    \draw (-1.5,-.1) -- (-1.5,.1);
    \draw (0,-.1) -- (0,.1);
    \node at (-.5,1) {$u_+$};
    \node at (-1.5,.5) {$\bar{u}_\mu$};
    \end{scope}

    \begin{scope}[shift={(9,0)}]
          \draw[black!30!white,->] (-2,0) -- (2,0);
    \draw (0,0) -- (2,-.75);
    \input{figures/points3.tikz}

    \draw[dashed] (\z,0) -- (1.5,0);
    \draw (\z,-.1) -- (\z,.1);
    \draw (1.5,-.1) -- (1.5,.1);
    \node at (1.5,-.25) {$u_-$};
    \node at (0,-.75) {$(u_\mu)_-$};
    \end{scope}
    
\end{tikzpicture}
    \caption{Set-up to apply \lref{intbyparts2}. Left figure: $u$, $u_\mu$ and $\Gamma_\mu$ are displayed.  Middle figure: set up to apply \eref{intbyparts2-supersoln}, the dashed set is $\{u > (u_\mu)_+\}$ -- note that $\bar{u}_\mu(x)$ is equal to $\alpha_+(P(x) - \mu)$ and hence is strictly superharmonic in this region. Right figure: set up to apply \eref{intbyparts2-subsoln}, the dashed set is $\{-u_\mu > (-u)_+\}$ -- note that $-u_\mu$ is equal to $-\alpha_-(P(x) - \mu)$ and hence is strictly subharmonic in this region.}
    \label{f.min-test}
\end{figure}
We are going to apply \lref{intbyparts2} several times.  First of all notice that $(u_\mu)_+$ extends naturally to $\{u>0\}$ as an $H^1$ superharmonic function in $\{u>0\} \setminus \{u_\mu >0\}$ by 
\[\bar{u}_\mu(x) : = \begin{cases} u_\mu(x) &\hbox{in } \ \{u_\mu >0\} \\
\alpha_+ (P(x) - \mu) &\hbox{in } \ \{ u>0\} \setminus \{u_\mu >0\}.
\end{cases}\]
 Now we are able to apply \eref{intbyparts2-supersoln} from \lref{intbyparts2} to $v_0 = \bar{u}_\mu(x)$ and $v_1 = u_+$. See \fref{min-test}. Note that 
 \[v_0 = \bar{u}_\mu = \alpha_+(P(x)-\mu) \ \hbox{ in  } \ \{v_1 > (v_0)_+\} = \{u>0\} \setminus \{u_\mu >0\},\]
 so indeed $v_0$ is superharmonic in $\{v_1 > (v_0)_+\}$ and therefore satisfies the hypothesis of \lref{intbyparts2} part (ii) with $\mu = 0$.  Then applying \eref{intbyparts2-supersoln}
\begin{align*} 
\int_{U} |\grad (u_\mu)_+|^2 - |\grad u_+|^2 \ dx &= \int_{U} |\grad (\bar{u}_\mu)_+|^2 - |\grad u_+|^2 \ dx \\
&\leq \int_{\{ u>0\} \setminus \overline{\{u_\mu > 0\}}} \alpha_+^2 |\grad P|^2\ dx\\
&=\int_{\Gamma_\mu} \alpha_+^2 |\grad P|^2\ dx
\end{align*}
using \eref{udelta-closure-technical} at the last step. 

Applying \eref{intbyparts2-subsoln} to $v_0 = -u_-$ and $v_1 = -(u_\mu)_-$, noting that $v_1$ is subharmonic in $v_1 > (v_0)_+$, we find
\begin{align*} 
\int_{U} |\grad u_-|^2 - |\grad (u_\mu)_-|^2 \ dx \geq \int_{\{ u_\mu < 0\} \setminus \overline{\{u < 0\}}} \alpha_-^2 |\grad P|^2\ dx \geq \int_{\Gamma_\mu} \alpha_-^2 |\grad P|^2\ dx,
\end{align*}
since $\overline{\{u < 0\}} \subset \{u \leq 0\}$ so that $\{ u_\mu < 0\} \setminus \overline{\{u < 0\}} \supset \Gamma_\mu$.

Now by the minimality of $u$
\[\mathcal{J}(u) \leq \mathcal{J}(u_\mu).\]
But applying the previous inequalities
\begin{align*} 
\mathcal{J}(u) - \mathcal{J}(u_\mu) 
& \geq \int_{\Gamma_\mu} \alpha_-^2 |\grad P|^2-\alpha_+^2 |\grad P|^2+(q_+^2 - q_-^2) \ dx \\
& \geq [(\alpha_-^2-\alpha_+^2) + (q_+^2 - q_-^2)-o_{\mu}(1)]|\Gamma_\mu|.
\end{align*}
We used here that $|\grad P(x_0)| = 1$ and that $\Gamma_\mu \to \{x_0\}$ as $\mu \to 0$ to get the $o_\mu(1)$ error term. Combining the previous
\[[(\alpha_-^2-\alpha_+^2) + (q_+^2 - q_-^2)-o_{\mu}(1)]|\Gamma_\mu| \leq 0\]
 so dividing through by $|\Gamma_\mu|>0$ and taking $\mu \to 0$ we find
\[ \alpha_+^2-\alpha_-^2 \geq q_+^2 - q_-^2 = 1.\]

\end{proof}

\subsection{Viscosity solution property of approximate minimizers}\label{s.viscosity-property-approximate}

This section considers the approximate viscosity solution property of approximate minimizers of $\mathcal{J}_0$.  More specifically, assume that $u \in H^1(B_1)$ satisfies for some $\sigma, K > 0$ and $\gamma \in (0,1)$
\begin{equation}\label{e.sigma-almost-min}
\mathcal{J}_0(u,B_1) \leq \mathcal{J}_0(v,B_1) + \sigma  \ \hbox{ for all } \ v \in u + H^1_0(B_1),
\end{equation}
and also
\begin{equation}\label{e.holder-assumption}
    ||u||_{C^\gamma(B_1)} \leq K.
\end{equation}
Later, we will consider approximate minimizers that are close to a two-plane solution $\Phi_\alpha(x_d)$ and we will have $K \lesssim \a_+$.

Under these hypotheses, we will show that $u$ satisfies an approximate viscosity solution property at the unit scale, with an error tied to the minimality error $\sigma$. The proof will again use the energy difference formulas in \lref{intbyparts2}, now taking advantage of the strict inequality afforded by the additional $L^1$ terms on the right-hand sides of \eref{intbyparts2-subsoln} and \eref{intbyparts2-supersoln}.  The integration by parts formulas in \lref{intbyparts2} replace similar ideas appearing in \cite{DeSilvaSavin}, proven by an auxiliary minimization problem (we are  referring, in particular, to the proofs of Lemmas 4.1 and 4.3 in \cite{DeSilvaSavin}).

The approximate minimality property \eref{sigma-almost-min} is, importantly, different from the almost-minimality and quasi-minimality properties that appear often in the literature.  It is an estimate on the proximity to energy minimality only at a single scale. To emphasize this, and for later application, we note how the approximate minimality property \eref{sigma-almost-min} transforms under rescaling. 

\begin{lemma}\label{lem:scaleinvarianceunitscale}
Suppose $u$ satisfies \eref{sigma-almost-min}. Then, for any $\rho \in (0,1]$, the function $u_{\rho}(x):=\rho^{-1}u(\rho x) \in H^1(B_1)$, satisfies 
\[\J_0(u_{\rho},B_1) \leq \J_0(w,B_1) + \rho^{-d} \sigma \qquad \hbox{for all }   w \in u_{\rho} + H^1_0(B_1).\]
\end{lemma}

We also note that \eref{sigma-almost-min} is, by itself, insufficient for regularity at scales below unit scale, so H\"older regularity is included as an additional assumption in \eref{holder-assumption}.  In practice, we will attain the requisite H\"older regularity \eref{holder-assumption} for a minimizer or almost-minimizer $u$ of a related, but different functional (e.g. the heterogeneous media functional $\mathcal{J}$) via interior H\"older estimates; see, for instance, the discussion in \sref{regularityofJminimizer}.

We proceed to prove the approximate viscosity supersolution property of approximate minimizers of $\mathcal{J}_0$.

\begin{lemma}[Supersolution Condition at Free Boundary]\label{l.approximate-supersoln-cond}
Assume $u$ satisfies \eref{sigma-almost-min} and \eref{holder-assumption}. Suppose that $\alpha > 0$, $0 \leq \mu < \frac{1}{2}$ and $P$ is a smooth test function with
\[ \Delta P \geq  \mu, \ 2 \geq |\grad P(x)|^2 \geq 1+ \mu, \ \hbox{ and } \ \partial_{x_d}P(x) \geq \frac{1}{2} \ \hbox{ in } \ B_1.\]
Then $\Phi_\alpha\circ P$ cannot be below $u$ in $B_1$ and touch $u$ from below at a point in $B_{1/2}$ if 
\[
\sigma \leq c_d K^{-\frac{d}{\gamma}}(\a\mu)^{2+\frac{d}{\gamma}}
\]

for a dimensional constant $c_d$.
\end{lemma}

\begin{proof}

We will prove the contrapositive of the claim. Denote $\a_- = \a$, $\a_+ = \sqrt{1+\a^2}$. Let $P$ be as in the statement of the lemma, and suppose $\Phi_\alpha \circ P$ is below $u$ in $B_1$ and touches $u$ from below at $x_0 \in B_{1/2}$.  Define 
\[\bar{P}(x) = P(x)  + \frac{\mu}{4d} ( 1- |x|^2).\]
Note that
\[\Delta \bar{P} \geq  \frac{\mu}{2} \ \hbox{ and } \ |\grad \bar{P}(x)|^2 \geq 1+\mu - \frac{1}{d}\mu \geq 1 \ \hbox{ in } \ B_1.\]  
and
\[\partial_{x_d} \bar{P}(x) \geq \frac{1}{2}-\frac{\mu}{2d} \geq \frac{1}{4} \ \hbox{ in } \ B_1. \]
Thus $\bar{P}$ is monotone increasing along all directions $|e-e_d| \leq \frac{1}{4}$. Using this cone monotonicity property, we obtain 
\[ \Phi_{\alpha}(\bar{P}(x_0)) \leq \Phi_{\alpha}(\bar{P}(y))  \ \hbox{ for } (y-x_0) \cdot e_d \geq \frac{3}{4}|x_0-y|.\]
Thus, in the local cone where $|y-x_0| \leq r$ and $(y-x_0) \cdot e_d \geq \frac{3}{4}|y-x_0|$, we have
\begin{align*}
    \Phi_{\alpha}(\bar{P}(y))- u(y)
    &\geq \Phi_{\alpha}(\bar{P}(x_0))- u(x_0)- Kr^\gamma\\
    &=\Phi_{\alpha}(P(x_0)+\tfrac{\mu}{4d} ( 1- |x_0|^2)) - \Phi_{\alpha}(P(x_0))-Kr^\gamma\\
    &\geq \Phi_{\alpha}(P(x_0)+\tfrac{3\mu}{16d}) - \Phi_{\alpha}(P(x_0) ) - Kr^\gamma\\
    &\geq \tfrac{3\alpha}{16d}\mu - K r^\gamma\\
    & \geq \tfrac{\alpha}{8d}\mu
\end{align*} 
 as long as $r \leq r_*$ with
\[r_* := \left(\frac{\alpha\mu}{16dK} \right)^{\frac{1}{\gamma}}.\]

Now define
\begin{equation}
    u_\mu(x) := u(x) \vee \Phi_{\alpha}(\bar{P}(x)).
\end{equation}
Notice that, by set algebra,
\[\Gamma := \{ u<0\} \setminus \{u_\mu \leq 0\} = \{ u_\mu > 0\} \setminus \{u \geq 0\}.\]
Since $\grad P \neq 0$ in $\{u_\mu>u\}$ then $\overline{\{u_\mu<0\}} \cap \{u_\mu>u\}= \{u_\mu\leq 0\} \cap \{u_\mu>u\}$ so that we also have the formula
\begin{equation}\label{e.gamma-set-id-approx}
    \{u<0\} \setminus \overline{ \{v <0 \}} = \Gamma.
\end{equation}
We are going to apply \lref{intbyparts2} several times.  First of all notice that $-(u_\mu)_-$ extends naturally to an $H^1$ superharmonic function in $\{u<0\} \setminus \{u_\mu <0\}$ by 
\[\bar{u}_\mu(x) : = \begin{cases} -u_\mu(x) &\hbox{in } \ \{u_\mu < 0\} \\
-\alpha_- \bar{P}(x) &\hbox{in } \ \{ u < 0\} \setminus \{u_\mu < 0\}.
\end{cases}\]
 Note that we are only claiming superharmonicity in $\{ u < 0\} \setminus \{u_\mu < 0\}$ where $\bar{u}_\mu$ is equal to $-\alpha_- \bar{P}$ which is strictly superharmonic $\Delta(-\alpha_- \bar{P}) \leq - \alpha_-\mu $. Applying \eref{intbyparts2-supersoln} to $v_0 = \bar{u}_\mu$ and $v_1 = -(u)_-$
\begin{align*} 
\int_{B_1} |\grad (u_\mu)_-|^2 - |\grad u_-|^2 \ dx &= \int_{B_1} |\grad (\bar{u}_\mu)_-|^2 - |\grad u_-|^2 \ dx \\
&\leq \int_{\{ u < 0\} \setminus \{u_\mu < 0\}} \alpha_-^2 |\grad \bar{P}(x)|^2\ dx - \alpha_-\mu  \int_{B_1} (u_- - (u_\mu)_-) \ dx\\
&=\int_{\Gamma} \alpha_-^2 |\grad \bar{P}(x)|^2\ dx - \alpha_-\mu  \int_{B_1} (u_- - (u_\mu)_-) \ dx
\end{align*}
using \eref{gamma-set-id-approx} at the last step.

 Next we apply \eref{intbyparts2-subsoln} to $v_0 = u_+$ and $v_1 = (u_\mu)_+$. Note that $v_1$ is strictly subharmonic in $\{v_1>(v_0)_+\} = \{(u_\mu)_+ > u_+\}$, since $u_\mu$ is equal to $\alpha_+P$ on that set and $\Delta (\alpha_+ P) \geq \alpha_+ \mu$. So we can indeed apply \eref{intbyparts2-subsoln} and find
\begin{align*} 
\int_{B_1} |\grad u_+|^2 - |\grad (u_\mu)_+|^2 \ dx &\geq \int_{\{ u_\mu >0\} \setminus \overline{\{u>0\}}} \alpha_+^2 |\grad \bar{P}(x)|^2\ dx +  \alpha_+\mu \int_{B_1} ((u_\mu)_+ - u_+) \ dx\\
&\geq \int_{\Gamma} \alpha_+^2 |\grad \bar{P}(x)|^2\ dx +  \alpha_+\mu \int_{B_1} ((u_\mu)_+ - u_+) \ dx
\end{align*}
since $\overline{\{u>0\}} \subset \{u \geq 0\}$ so that $\{u_\mu >0\} \setminus \overline{\{u>0\}} \supset \Gamma$.

Applying the previous inequalities, 
\begin{align*} 
\mathcal{J}(u) - \mathcal{J}(v) 
    & \geq \int_{\Gamma} \alpha_+^2 |\grad \bar{P}(x)|^2-\alpha_-^2 |\grad \bar{P}(x)|^2-[q_+^2 - q_-^2] \ dx \\
    &  \quad \cdots +  \mu\int_{B_1} \alpha_+(v_+ - u_+)+\alpha_-(u_- - v_-) \ dx\\
& \geq \int_{\Gamma} [(\alpha_+^2-\alpha_-^2)|\grad \bar{P}(x)|^2 - (q_+^2 - q_-^2)] \ dx+ \mu\alpha \int_{B_1} (v-u) \ dx\\
& = \int_{\Gamma} (q_+^2 - q_-^2)[|\grad \bar{P}(x)|^2-1] \ dx+ \tfrac{1}{4d}\mu^2\alpha^2r_*^d\\
&\geq 0+ \tfrac{1}{4d}\mu^2\alpha^2r_*^d\\
&\geq \tfrac{1}{4d}\mu^2\alpha^2r_*^d
\end{align*}
On the other hand, by \eref{sigma-almost-min}
\[\mathcal{J}(u) \leq \mathcal{J}(v) + \sigma.\]

So combining the previous, we find
\[\sigma \geq \frac{1}{4d}\mu^2\alpha^2r_*^d = c(d,\gamma)(\alpha/K)^{d/\gamma}\alpha^2\mu^{2+\frac{d}{\gamma}}\]
which is the contrapositive of the claim.
\end{proof}

For later applications, we will need the following corollary, which is essentially a rephrasing of the previous lemma.
\begin{corollary}\label{c.approximate-supersoln-cond-B}
Assume $u$ satisfies \eref{sigma-almost-min} and \eref{holder-assumption}. Suppose that $\alpha_+>\alpha_- > 0$ and $\mu>0$ satisfy
\[\alpha_+^2 - \alpha_-^2  \geq (1 + \mu)^2 \]
and $R$ is a smooth test function satisfying
\[ \Delta R  \geq \mu \ \hbox{ in } \ B_1 \ \hbox{ and }  \ |\partial_{x_d}R| \leq \frac{1}{4}\mu. \]
Then $\Phi_{\alpha_\pm}(x_d + R(x))$ cannot be below $u$ in $B_1$ and touch $u$ from below in $B_{1/2}$ if 
\[\sigma \leq c_d K^{-\frac{d}{\gamma}}(\a\mu)^{2+\frac{d}{\gamma}}, \quad \hbox{ where }  \alpha^2 := \frac{\alpha_-^2}{\alpha_+^2-\alpha_-^2}.
\]

\end{corollary}

\begin{proof}
    Without loss we can assume that $\alpha_+^2 - \alpha_-^2 =(1+\mu)^2 \leq 2$ since either $\alpha_+$ or $\alpha_-$ can be (respectively) decreased or increased without affecting the property that $\varphi(x):=\Phi_{\alpha_\pm}(x_d+P(x))$ touches $u$ from below at some $x_0 \in B_{1/2}$.  Specifically, if $x_0 \in \{u>0\}$ then increase $\alpha_-$ and if $x_0 \in \{u\leq 0\}$ decrease $\alpha_+$. Notice that the test function can be re-construed as
    \[\varphi(x):=\Phi_{\alpha_\pm}(x_d+R(x)) = \Phi_{\frac{\alpha_{\pm}}{\sqrt{\alpha_+^2-\alpha_-^2}}}\left(\sqrt{\alpha_+^2 - \alpha_-^2}(x_d+R(x))\right).\]
    Then
    \[\frac{\alpha_+^2}{\alpha_+^2-\alpha_-^2} - \frac{\alpha_-^2}{\alpha_+^2-\alpha_-^2} = 1.\]
    So calling 
    \[\alpha^2 := \frac{\alpha_-^2}{\alpha_+^2-\alpha_-^2} \ \hbox{ and } \ P(x) := \sqrt{\alpha_+^2 - \alpha_-^2}(x_d+R(x)) = (1+\mu)(x_d+R(x)) \]
    we can apply \lref{approximate-supersoln-cond} if we can verify the conditions on $P$.  Specifically, we check the subharmonicity
    \[\Delta P = (1+\mu) \Delta R \geq \mu, \]
    directional monotonicity
    \[\partial_{x_d}P = (1+\mu)(1+ \partial_{x_d}R) \geq (1+\mu)(1-\tfrac{1}{4}\mu) \geq 1\]
    and the slope lower bound
    \[|\grad P|^2 = (1+\mu)^2(1+2\partial_{x_d}R + |\grad R|^2) \geq (1+\mu)^2(1-\tfrac{1}{2}\mu)=(1+\mu)(1+\tfrac{1}{2}\mu(1-\mu)) \geq 1+\mu.\]
\end{proof}

Of course, there are subsolution versions of \lref{approximate-supersoln-cond} and \cref{approximate-supersoln-cond-B}; we state them here for later use. Their proofs follow from ideas similar to the supersolution case and are thus omitted. 

\begin{lemma}[Subsolution Condition at Free Boundary]\label{l.approximate-subsoln-cond}
Assume $u$ satisfies \eref{sigma-almost-min} and \eref{holder-assumption}. Suppose that $\alpha > 0$, $0 \leq \mu < \frac{1}{2}$ and $P$ a smooth test function with
\[ \Delta P \leq  -\mu, \  |\grad P(x)|^2 \leq 1- \mu, \ \hbox{ and } \ \partial_{x_d}P(x) \geq \frac{1}{2} \ \hbox{ in } \ B_1.\]
Then $\Phi_\alpha\circ P$ cannot be above $u$ in $B_1$ and touch $u$ from above at a point in $B_{1/2}$ if 
\[
\sigma \leq c K^{-\frac{d}{\gamma}}(\a\mu)^{2+\frac{d}{\gamma}}
\]

for a positive constant $c = c(d,\gamma)$.
\end{lemma}

\begin{corollary}\label{c.approximate-subsoln-cond-B}
Assume $u$ satisfies \eref{sigma-almost-min} and \eref{holder-assumption}. Suppose that $\alpha_+>\alpha_- > 0$ satisfy,
\[\alpha_+^2 - \alpha_-^2  \leq 1 - \mu \ \hbox{ for } \ 0 \leq \mu < \frac{1}{2} \]
and $R$ is a smooth test function satisfying
\[ \Delta R  \leq -\mu \ \hbox{ in } \ B_1, \ |\partial_{x_d}R| \leq \frac{1}{4}\mu. \]
Then $\Phi_{\alpha_\pm}(x_d + R(x))$ cannot be above $u$ in $B_1$ and touch $u$ from above in $B_{1/2}$ if 
\[\sigma \leq c K^{-\frac{d}{\gamma}}(\a\mu)^{2+\frac{d}{\gamma}}, \quad \hbox{ where }
 \alpha^2 := \frac{\alpha_-^2}{\alpha_+^2-\alpha_-^2}.
\]
\end{corollary}

\subsection{Harmonic approximation and consequences}\label{subsec:int-harnack}

We continue to study approximate minimizers $u$ satisfying \eref{sigma-almost-min} and \eref{holder-assumption}. We begin by showing that such functions are well approximated, in the interior of their positive and negative sets, by their local harmonic replacement. We then present two applications of this result.  First, we show an approximate viscosity solution property for test functions touching away from the zero level set. And second, we show that $u$ satisfies a version of the Harnack inequality, again away from the zero level set.

\begin{lemma}\label{l.harmonic-approximation}
Assume $u$ satisfies \eref{sigma-almost-min} and \eref{holder-assumption} and $B_1 \subset \{u \neq 0\}$. If $v$ is the harmonic replacement of $u$ in $B_1$, then there exists a constant $C = C(d,\gamma)$ such that
\[||u - v||_{L^{\infty}(B_{\frac{1}{2}})} \leq C K^{\frac{d}{d+2\gamma}}\sigma^{\frac{\gamma}{d+2\gamma}}.\]
\end{lemma}

\begin{proof}

Note that since $u$ satisfies \eref{holder-assumption}, we have by estimates for the Poisson kernel in $B_1$ that $||v||_{C^{\gamma}(B_1)} \leq CK$ for some $C = C(d,\gamma) > 1$ and so $||u-v||_{C^{\gamma}(B_1)} \leq CK=:K_1$. We may thus extend $u-v$ by zero outside of $B_1$ and still have $||u-v||_{C^{\gamma}(\R^d)} \leq K_1$.

By \lref{a-harmonic-replacement} applied with $r = 1$, we have for some universal constant $C$
\[\int_{B_1} |\grad (u-v)|^2 dx \leq C\sigma.\]
 So by Poincar\'e inequality
\begin{equation}\label{L2boundusingPoincare}
\int_{\R^d} (u-v)^2  
dx = \int_{B_1} (u-v)^2 \leq C \sigma.
\end{equation}
Let $M := ||u - v||_{L^{\infty}(B_{1/2})}$ and choose $x_0 \in B_{1/2}$ such that $|u(x_0) - v(x_0)| \geq \frac{M}{2}$. Let $\rho =\left(\frac{M}{4K_1}\right)^{\frac{1}{\gamma}}$. Then for any $x \in B_{\rho}(x_0)$, we have
\[|u(x)-v(x)| \geq |u(x_0) - v(x_0)| - K_1|x-x_0|^{\gamma} \geq \frac{M}{2} - K_1\rho^{\gamma} = \frac{M}{4}.\]
Consequently, by \eqref
{L2boundusingPoincare},
\[
C \sigma \geq \int_{\R^d} (u-v)^2 dx \geq \int_{B_{\rho}(x_0)} (u-v)^2 dx \geq \frac{M^2 \omega_d \rho^d}{16} = c_dK_1^{-\frac{d}{\gamma}}M^{2+\frac{d}{\gamma}}
\]
which implies
\[M \leq C K^{\frac{d}{d+2\gamma}}\sigma^{\frac{\gamma}{d+2\gamma}}.\]
\end{proof}

Now we proceed with the two corollaries of this harmonic approximation property. First, we establish an approximate viscosity solution property away from the zero level set.

\begin{corollary}[Interior Viscosity Solution Property]\label{cor:interior-viscosity}
Suppose $u$ satisfies \eref{sigma-almost-min} and \eref{holder-assumption} and $B_1 \subset \{u \neq 0\}$. Let $\alpha > 0$, $0 \leq \mu < \frac{1}{2}$ and $P$ be a smooth test function satisfying
\[ \Delta P \geq  \mu \ \hbox{ in } \ B_1.\]
 There exists a constant $c_0 = c_0(d,\gamma)$ such that if $\sigma \leq c_0 K^{-\frac{d}{\gamma}}(\a\mu)^{2+\frac{d}{\gamma}}$, then $\Phi_\alpha\circ P$ cannot be below $u$ in $B_1$ and touch $u$ from below in $B_{1/2}$.

\end{corollary}
\begin{Remark} The composition with $\Phi_\alpha$ does not play much role in the proof, but since we will use test functions of the form $\Phi_\alpha\circ P$ later, the statement above will be useful to make clear the dependence of $\sigma$ on $\alpha$.
\end{Remark}

\begin{proof}
     Assume that $B_1 \subset \{u<0\}$; then $\Phi_\alpha\circ P = \alpha P$. Suppose that $\alpha P$ touches $u$ from below at some $x_0 \in B_{1/2}$. Let $v$ be the harmonic replacement of $u$ in $B_1$; note that $\alpha P \leq v$ on $\partial B_1$.  By \lref{harmonic-approximation} and the upper bound on $\sigma$, we have 
    \begin{align*}
        \|u-v\|_{L^{\infty}(B_{1/2})} 
        & \leq C K^{\frac{d}{d+2\gamma}}\sigma^{\frac{\gamma}{d+2\gamma}} \\
        & \leq  C K^{\frac{d}{d+2\gamma}}\left[c_0 K^{-\frac{d}{\gamma}}(\a\mu)^{2+\frac{d}{\gamma}}\right]^{\frac{\gamma}{d+2\gamma}} \\
        & \leq \tfrac{1}{8d}\alpha \mu 
    \end{align*}
where the final inequality follows by choosing $c_0$ appropriately. Define 
\[w(x) = \alpha P(x)  + \frac{\alpha \mu}{4d} ( 1- |x|^2) - v(x) .\]
Note that $w$ is subharmonic in $B_1$ with $w \leq 0$ on $\partial B_1$ and
\begin{align*}
    w(x_0) &\geq \alpha P(x_0)+\tfrac{3}{16d}\alpha \mu - v(x_0) \\
    & = \tfrac{3}{16d}\alpha \mu + u(x_0) - v(x_0) \\
    & \geq \left(\tfrac{3}{16d} - \tfrac{1}{8d}\right)\alpha \mu >0
\end{align*}

which contradicts maximum principle for $w$.

The case $B_1 \subset \{u>0\}$ is similar, noting that if $\Phi_\alpha \circ P$ touches $u$ from below then so does $\alpha_+ P$. Due to positivity of $u$, $\alpha_+P$ is certainly below $u$ where $P < 0$.
    
\end{proof}

The second consequence of the harmonic approximation property \lref{harmonic-approximation} is a Harnack inequality away from the zero level set.
\begin{Corollary}[Interior Harnack]\label{cor:interiorharnack}
Suppose $u$ satisfies \eref{sigma-almost-min} and \eref{holder-assumption} and $B_1 \subset \{u \neq 0\}$. Let $w$ be a harmonic function such that $u - w$ is non-negative in $\partial B_1$ and $(u-w)(0) \geq \mu > 0$. There exists $c_1 = c_1(d, \gamma)$ such that if $\sigma \leq c_1 K^{-\frac{d}{\gamma}}\mu^{2 + \frac{d}{\gamma}}$, then $u-w \geq c_d \mu$ in $B_{\frac{1}{2}}$ for a dimensional constant $c_d$.

Similarly, if $w$ is a harmonic function such that $w - u$ is non-negative in $\partial B_1$ and $(w-u)(0) \geq \mu > 0$, then there exists $c_1(d, \gamma)$ such that if $\sigma \leq c_1K^{-\frac{d}{\gamma}}\mu^{2 + \frac{d}{\gamma}}$, then $w-u \geq c_d\mu$ in $B_{\frac{1}{2}}$ for a dimensional constant $c_d$.
\end{Corollary}

\begin{proof}
We only prove the first half of the claim. Let $v$ be the harmonic replacement of $u$ in $B_1$. Since $v = u \geq w$ on $\partial B_1$ and $w$ is harmonic, we must have $v \geq w$ in $B_1$ and $v - w > 0$ in the interior of $B_1$. Consequently, by the Harnack inequality, there exists $C_H(d)<1$ such that $v-w \geq C_H(v-w)(0)$ in $B_{\frac{1}{2}}$. Then, for any $x \in B_{\frac{1}{2}}$,  \lref{harmonic-approximation} implies
\begin{align*}
(u-w)(x) &= (u-v)(x) + (v-w)(x) \\
& \geq -C K^{\frac{d}{d+2\gamma}}\sigma^{\frac{\gamma}{d+2\gamma}} + C_H(v-w)(0) \\
& = -C K^{\frac{d}{d+2\gamma}}\sigma^{\frac{\gamma}{d+2\gamma}}+ C_H(v-u)(0) + C_H(u-w)(0) \\ 
& \geq -(1+C_H)C K^{\frac{d}{d+2\gamma}}\sigma^{\frac{\gamma}{d+2\gamma}} + C_H\mu \geq \frac{C_H \mu}{2}
\end{align*}
if we choose
\[K^{\frac{d}{d+2\gamma}}\sigma^{\frac{\gamma}{d+2\gamma}} \leq \frac{C_H}{2C(1+C_H)}\mu\]
or, equivalently,
\[\sigma \leq CK^{-\frac{d}{\gamma}}\mu^{2 + \frac{d}{\gamma}}.\]
\end{proof}

\section{Free boundary Harnack inequality and H\"older regularity of flat correction}

Throughout this section, we fix $\alpha > 0$ and denote $\alpha_+ = \sqrt{1 + \alpha^2}$. We continue to assume that $u$ satisfies the approximate-minimality property \eref{sigma-almost-min}, but upgrade the H\"older estimate \eref{holder-assumption} to the following scale-invariant version.

\begin{property}[Interior H\"older estimate at all scales]\label{pty.holder-est}
     There is $\gamma \in (0,1)$ and $\HolderConst\geq 1$ so that $u$ satisfies an interior H\"older estimate of the form
    \[r^{\gamma-1}[u]_{C^\gamma(B_{r})} \leq \HolderConst(1+\osc_{B_{2r}} u)\]
    for all balls $B_{2r} \subset B_1(0)$.
\end{property}

Note that if $u$ is a $\J$ minimizer, then $u$ satisfies Property \ref{pty.holder-est}, see \sref{regularityofJminimizer} for more details.

\begin{definition}
We say $u$ is $(\alpha, \delta)$-flat (in the $x_d$ direction) in $\O \subset \R^d$ if
\[\Phi_\alpha(x_d - \delta) \leq u(x) \leq \Phi_\alpha(x_d + \delta) \ \hbox{ for all } x \in \Omega.\]
\end{definition}

The main result of this section is the following:

\begin{theorem}[H\"older modulus of the flat correction]\label{t.holder-est-w}
    There exists $\delta_d>0$ universal so that if $u$ is $(\a,\delta)$-flat in $B_1$ for some $\alpha >0$ and $2\delta \leq \delta_d$ and $\sigma \leq c_1 (3 \HolderConst\a_+)^{-\frac{d}{\gamma}}(\a \delta)^{2 + \frac{d}{\gamma}}$,  then the flat correction
     \[w_\delta(x):=\frac{\Psi_{\alpha}(u(x)) - x_d}{\delta}\]
     satisfies the H\"older estimate
     \[|w_\delta(x) - w_\delta(y)| \leq C |x-y|^\eta \qquad \text{for all } x,y, \in B_{\frac{1}{2}} \text{ such that } |x-y| \geq 20\left(\delta/\delta_d\right)^{\frac{1}{1-\eta}}\]
     where $0 < \eta \leq \frac{(1-\gamma)d + 2\gamma}{d+2\gamma} < 1$.
\end{theorem}

The general idea of this proof was introduced by De Silva in \cite{DeSilva}, with further applications to two-phase problems and almost minimizers relevant to our work in \cites{DeSilvaSavinAM,DeSilvaFerrariSalsa-APDE2014,DeSilvaFerrariSalsaNonlinear}. The key ingredient is a Harnack property for flat almost minimizers near the free boundary, Proposition \ref{prop:fbharnack}, whose proof relies on a barrier argument and the viscosity supersolution property at the free boundary, \lref{approximate-supersoln-cond}, established earlier. What is somewhat different in our work is the additional hypothesis of the H\"older estimate at all scales, Property \ref{pty.holder-est}, which is needed to apply the approximate viscosity solution property.

Let us note that the $(\alpha,\delta)$-flatness assumption above implies 
\[\osc_{B_1} u \leq 2\a_+.\]
Coupled with Property~\ref{pty.holder-est}, we thus have
\[\|u\|_{C^\gamma(B_1)} \leq 3\HolderConst\alpha_+.\]
We will state this as a hypothesis in the following results, keeping in mind that it is a consequence of Property~\ref{pty.holder-est}.

\subsection{Free Boundary Harnack Property}

We begin with an important lemma.

\begin{lemma}\label{lem:fbharnack}
Suppose that $u$ satisfies
    \eref{sigma-almost-min}
    and $\|u\|_{C^\gamma(B_1)} \leq 3\HolderConst \alpha_+$ for some $\alpha \geq \underline{\alpha}>0$. There exists a universal constant $\delta_d > 0$ such that if $0 < \delta \leq\delta_d$ and $u$ satisfies
\begin{enumerate}
\item[(i)] $u(x) \geq \Phi_{\a} (x_d)$ for all  $x \in B_1$, and
\item[(ii)] $u(\bar{x}) \geq \Phi_{\a} (\bar{x}_d + \delta)$ where $\bar{x} := \frac{1}{5}e_d$,
\end{enumerate}
then there exists a dimensional constant $c \in (0,1)$ such that
\[u(x) \geq \Phi_{\a}(x_d + c \delta) \quad \text{ for all } x \in B_{\frac{1}{2}}\]
provided $\sigma \leq C(\HolderConst \a_+)^{-\frac{d}{\gamma}} (\a \delta)^{2 + \frac{d}{\gamma}}$ for some $C = C(d,\gamma)$.
\end{lemma}

\begin{proof} We have $u - \Phi_{\a} \geq 0$ in $B_{\frac{1}{10}}(\bar{x}) \subset \{u > 0\} \cap B_1$. By assumption, \[u(\bar{x}) \geq \Phi_{\a} (\bar{x}_d + \delta) \geq \Phi_{\a} (\bar{x}_d) + \a_+ \delta \ \hbox{ in } \ B_{\frac{1}{10}}(\bar{x}).\] Since $\Phi_{\a}$ is harmonic in $B_{\frac{1}{10}}(\bar{x})$, we can apply Corollary \ref{cor:interiorharnack} (with $\mu = \a \delta$ and $K = 3 \HolderConst\a_+$) and conclude
\begin{equation}\label{e.basic-harnack-lower-bound} 
u - \Phi_{\a} \geq c_d \a_+ \delta \ \hbox{ in } \ B_{\frac{1}{20}}(\bar{x})
\end{equation}
provided $\sigma \leq c_1 (3 \HolderConst\a_+)^{-\frac{d}{\gamma}}(\a \delta)^{2 + \frac{d}{\gamma}}$.

We will use \eref{basic-harnack-lower-bound} with a barrier argument to prove the conclusion of the lemma. For $p > d-2$, let $w(x) = c\left(|x-\bar{x}|^{-p} - \left(\frac{3}{4}\right)^{-p}\right)$, where $c > 0$ is chosen so that $w = 0$ on $\partial B_{\frac{3}{4}}(\bar{x})$ and $w = 1$ on $\partial B_{\frac{1}{20}}(\bar{x})$. It is easy to verify that $\Delta w \geq C(d) > 0$ for some dimensional constant $C(d) > 0$, and  so $0 \leq w \leq 1$ in $B_{\frac{3}{4}}(\bar{x}) \setminus B_{\frac{1}{20}}(\bar{x})$. We extend $w$ to be identically $1$ on $B_{\frac{1}{20}}(\bar{x})$ and identically zero outside $B_{\frac{3}{4}}(\bar{x})$.

For $t \geq 0$, let 
\[v_t(x) := \Phi_{\a}(x_d + \delta c_d w(x)+ \delta (t-c_d))\qquad x \in B_{\frac{3}{4}}(\bar{x}),\]
where $c_d$ is the constant in \eref{basic-harnack-lower-bound}. Since $w \leq 1$ in $B_{\frac{3}{4}}(\bar{x}) \subset B_1$, we have 
\[v_0(x) = \Phi_{\a}(x_d - \delta c_d (1-w(x))) \leq \Phi_{\a}(x) \leq u(x) \qquad \text{in } B_{\frac{3}{4}}(\bar{x}) \subset B_1.\]
Let 
\[\tilde{t} = \sup\{t : v_t \leq u \text{ in } B_{\frac{3}{4}}(\bar{x})\}.\] 
We claim $\tilde{t} \geq c_d$. If this holds, then
\begin{align*}
u(x) & \geq v_{\tilde{t}}(x) \\
& = \Phi_{\a}(x_d + \delta c_d w(x) + \delta(\tilde{t} - c_d) ) \\
& \geq \Phi_{\a}(x_d + \delta c_d w(x)) \\
& \geq \Phi_{\a}(x_d + c \delta) \qquad \text{for all } x \in B_{\frac{1}{2}} \Subset B_{\frac{3}{4}}(\bar{x})
\end{align*}
which is the desired conclusion.

It remains to prove the claim $\tilde{t} \geq c_d$. Suppose, for contradiction, that $\tilde{t} < c_0$. Let $\tilde{x} \in \overline{B_{\frac{3}{4}}(\bar{x})}$ be such that $v_{\tilde{t}}(\tilde{x}) = u(\tilde{x})$. We consider three cases.

\emph{Case 1:} $\tilde{x} \in \partial B_{\frac{3}{4}}(\bar{x})$.

Since $w \equiv 0$ on $\partial B_{\frac{3}{4}}(\bar{x})$, we have 
\[v_{\tilde{t}}(\tilde{x}) = \Phi_{\a}(\tilde{x}_d  - \delta(c_d - \tilde{t})) < \Phi_{\a}(\tilde{x}_d) \leq u(\tilde{x}).
\]
Therefore, $\tilde{x} \notin \partial B_{\frac{3}{4}}(\bar{x})$.

\emph{Case 2:} 
$\tilde{x} \in \overline{B_{\frac{1}{20}}(\bar{x})}$.

Since $w \equiv 1$ in $\overline{B_{\frac{1}{20}}(\bar{x})}$, we have
\begin{align*}
u(\tilde{x}) = v_{\tilde{t}}(\tilde{x}) & = \Phi_{\a}(\tilde{x}_d + \delta c_d w(\tilde{x}) + \delta(\tilde{t} - c_d))\\
& = \Phi_{\a}(\tilde{x}_d +  \delta\tilde{t}) \\
& = \a_+ (\tilde{x}_d +  \delta\tilde{t}) \qquad \text{provided } \delta \leq \frac{1}{20c_d} =: \delta_d\\
& = \Phi_{\a}(\tilde{x}_d) +  \a_+ \delta\tilde{t} \\
& < \Phi_{\a}(\tilde{x}_d) +  \a_+ \delta c_d
\end{align*}
which contradicts \eref{basic-harnack-lower-bound}. Therefore, $\tilde{x} \notin \overline{B_{\frac{1}{20}}(\bar{x})}$.

\emph{Case 3:} $\tilde{x} \in \mathcal{A} := B_{\frac{3}{4}}(\bar{x}) \setminus \overline{B_{\frac{1}{20}}(\bar{x})}$.

Let $P(x) := x_d + \delta c_d w(x) + \delta (\tilde{t}-c_d)$; then $v_{\tilde{t}}(x) = \Phi_{\a}\circ P(x)$. We have
\[\Delta P(x) = \delta c_d \Delta w(x) \geq c_1 \delta \qquad \text{in } \mathcal{A}.\]
If $\tilde{x} \in \mathcal{A} \cap \{x_d > \frac{1}{10} \}$, then Corollary \ref{cor:interior-viscosity} implies $v_{\tilde{t}}$ cannot touch $u$ from below provided $\sigma \leq c_0 (3\HolderConst\a_+)^{-\frac{d}{\gamma}} (\a c_1 \delta)^{2 + \frac{d}{\gamma}}$.
Therefore, $\tilde{x} \in \mathcal{A} \cap \{x_d \leq \frac{1}{10}\}$. Now on $\{x_d < \frac{1}{10}\} \cap \mathcal{A}$, the function $P(x)$ satisfies $e_d \cdot \nabla P(x) \geq 1 + c\delta$. Therefore, by \lref{approximate-supersoln-cond}, we see that $v_{\tilde{t}}$ cannot touch $u$ from below in $\mathcal{A} \cap \{x_d < \frac{1}{10}\}$ provided $\sigma \leq c_d (3\HolderConst\a_+)^{-\frac{d}{\gamma}} (\a c_1 \delta)^{2 + \frac{d}{\gamma}}$. We conclude that $\tilde{x} \notin \mathcal{A}$. 

Since we arrived at a contradiction in all three cases above, we  conclude that, if $\tilde{t} < c_0$, then $v_{\tilde{t}}$ cannot touch $u$ from below in $B_{\frac{3}{4}}(\bar{x})$, and so $\tilde{t} \geq c_0$. We note that all the requirements on $\sigma$ are satisfied if we assume, at the outset, that $\sigma \leq C(\HolderConst \a_+)^{-\frac{d}{\gamma}} (\a \delta)^{2 + \frac{d}{\gamma}}$ for some $C = C(d,\gamma)$.
\end{proof}

The Harnack property at the free boundary is a somewhat immediate consequence of the previous lemma.

\begin{Proposition}[Free Boundary Harnack Property] \label{prop:fbharnack} Suppose that $u$ satisfies \eref{sigma-almost-min} and $\|u\|_{C^\gamma(B_1)} \leq 3 \HolderConst\alpha_+$  for some $\alpha >0$. There exists a universal constant $\delta_d > 0$ such that if $u$ satisfies
\begin{equation}\label{flatnessassumption}
\Phi_{\a}(x_d + a_0) \leq u(x) \leq \Phi_{\a}(x_d + b_0) \quad \text{in } B_1
\end{equation} 
with $\delta:= b_0 - a_0\leq \delta_d$, $|a_0| \leq \frac{1}{5}$, then 
\[\Phi_{\a}(x_d + a_1) \leq u(x) \leq \Phi_{\a}(x_d + b_1) \quad \text{in } B_{\frac{1}{20}} \]
where $a_0 \leq a_1 < b_1 \leq b_0$ and $b_1 - a_1 \leq (1-c) \delta$ for some $c \in (0,1)$ universal, provided $\sigma \leq C(\HolderConst \a_+)^{-\frac{d}{\gamma}} (\a \delta)^{2 + \frac{d}{\gamma}}$ for some $C = C(d,\gamma)$.

\end{Proposition}

\begin{proof}

The proof is essentially the same as \cite{DeSilvaFerrariSalsa-APDE2014}*{Theorem 4.1} so we omit some of the details.

Let $\delta_d$ be as in Lemma \ref{lem:fbharnack}. By  \eqref{flatnessassumption}, we have 
\[\Phi_{\a}(x_d + a_0) \leq u(x) \leq \Phi_{\a}(x_d + a_0 + \delta) \quad \text{in } B_1.\]
Let $\bar{x} = \frac{1}{5}e_d$ and suppose $u(\bar{x}) \geq \Phi_{\a}(\bar{x}_d + a_0 + \frac{\delta}{2})$ (the other case is treated similarly). For $x \in B_{\frac{4}{5}}$, consider the function $v(x) := u(x-a_0 e_d)$. Then $\|v\|_{C^\gamma(B_\frac{4}{5})} \leq K\alpha_+$ and $v$ satisfies \eref{sigma-almost-min} in $B_{\frac{4}{5}}$ by translation invariance of $\J_0$. Moreover, $v(\bar{x}) \geq \Phi_{\a}(\bar{x}_d + \frac{\delta}{2})$ and, by \eqref{flatnessassumption},
\[\Phi_{\a}(x_d) \leq v(x) \leq \Phi_{\a}(x_d + \delta) \quad \text{in } B_{\frac{4}{5}}.\]
Thus, by Lemma \ref{lem:fbharnack}, $v(x) \geq \Phi_{\a}(x_d + c\delta)$ in $B_{\frac{2}{5}}$, which implies $u(x) \geq \Phi_{\a}(x_d + a_0 + c\delta)$ in $B_{\frac{3}{5}}$. Letting $b_1 = b_0$ and $a_1 = a_0 + c\delta$, we arrive at the desired conclusion.

\end{proof}

We are now ready to prove the main result of this section.

\subsection{Proof of \tref{holder-est-w}}  Let $c \in (0,1)$ be the universal constant from Proposition \ref{prop:fbharnack}. Let $m$ be the largest integer such that $(20(1-c))^m \delta \leq \delta_d$. We may iterate Proposition \ref{prop:fbharnack} $m$ times by combining the rescaling property, Lemma \ref{lem:scaleinvarianceunitscale}, with Property \ref{pty.holder-est}, provided $u$ satisfies \eref{sigma-almost-min} with $\sigma \leq C(20^{-m})^d(20^m(1-c)^m\delta)^{2 + \frac{d}{\gamma}}$; this actually holds for any positive integer $m$ if we choose $c$ in Proposition \ref{prop:fbharnack} sufficiently small so that $20^{-\eta} = 1-c$ with $0 < \eta \leq \frac{(1-\gamma)d + 2\gamma}{d+2\gamma}$. We thus obtain constants $a_m, b_m$ such that 
\[
\Phi_{\a}(x_d + a_m) \leq u(x) \leq \Phi_{\a}(x_d + b_m) \quad \text{in } B_{20^{-m}}
\]
and
\[b_m - a_m \leq (1-c)^m \delta.\]
This implies $w_{\delta}$ satisfies 
\[\osc_{B_{20^{-m}}} w_{\delta}\leq \frac{b_m - a_m}{\delta} \leq (1-c)^m = 20^{-m\eta}.\]\qed

    \section{Asymptotic expansion for flat approximate viscosity solutions}\label{s.asymptotic-exp}

In this section, we establish a one-step improvement of flatness for $\J_0$ approximate minimizers (i.e. $u$ satisfying \eref{sigma-almost-min}) that also enjoy the interior H\"older estimates at all scales as in Property~\ref{pty.holder-est}. Our main result is as follows:

     \begin{theorem}[One-step improvement of flatness for H\"older $\J_0$ approximate minimizers] \label{t.onestep-improve-flatness}
    Suppose $u \in H^1(B_1)$ satisfies \eref{sigma-almost-min} and Property~\ref{pty.holder-est}. For all $\beta \in (0,1)$ and $\underline{\alpha}>0$ there exists $\bar\delta, \theta, C>0$ so that if $u(0) = 0$ and $u$ is $(\a,\delta)$-flat in $B_1$ for some $\alpha \geq \underline{\alpha}$ and $0 \leq \delta \leq \bar\delta$, 
     then 
     \[\Phi_{\alpha'}(x \cdot e - \theta^{1+\beta}\delta-C(\underline{\alpha})\sigma^{\frac{\gamma}{d+3\gamma}}) \leq u(x) \leq \Phi_{\alpha'}(x\cdot e + \theta^{1+\beta} \delta+C(\underline{\alpha})\sigma^{\frac{\gamma}{d+3\gamma}}) \ \hbox{ in } \ B_{\theta} \]
      for some $\a'>0$ and unit vector $e$ satisfying
     \[|\alpha' - \alpha| \leq C \alpha \delta \qquad |e-e_d| \leq C \delta.\]
    \end{theorem}
    In \sref{asymptotic-exp-1} we show a formal asymptotic expansion which explains how the two-phase problem linearizes to a transmission problem near two-plane solutions. Then we recall some results on the regularity of transmission problems. 

    Next, in \sref{asymptotic-exp-2}, we apply the compactness contradiction strategy of \cites{DeSilva,DeSilvaFerrariSalsa-APDE2014,DeSilvaSavinAM} to show a rigorous asymptotic expansion of flat solutions. 

    Finally, in \sref{asymptotic-exp-3}, we return to derive the one-step improvement of flatness, \tref{onestep-improve-flatness}, by combining the rigorous asymptotic expansion of flat solutions with the regularity theory for the transmission problem.
     \subsection{Formal asymptotic expansion of flat solutions and the transmission problem}\label{s.asymptotic-exp-1}
     To motivate the results in this section, we make a formal asymptotic expansion argument for flat solutions of \eref{two-phase-PDE}. This follows ideas presented in \cite{DeSilva}. Fix $\a > 0$ and suppose that
    \[\sup_{B_1} \big|\Psi_\alpha(u(x)) - x_d \big| \leq \delta.\]
     We write the formal asymptotic expansion in the flatness parameter $\delta>0$
     \[u(x) = \Phi_{\a}(x_d + \delta w(x)+o(\delta)).\]
     Due to the flatness bounds we can expect $w \leq O(1)$ and
     \[\{u>0\} \to B_1^+ \ \hbox{ and } \ \{u<0\} \to B_1^- \ \hbox{ as } \ \delta \to 0^+.\]  
     In order for $u$ to be harmonic in $\{u \neq 0\}$, we should have
     \[\Delta w = 0 \ \hbox{ in } \ \{u\neq0\}.\]
     In order for $u$ to satisfy the free boundary condition, we should have
     \[  1 = |\grad u_+|^2 - |\grad u_-|^2 = \alpha_+^2 - \alpha_-^2 + 2 \delta [\alpha_+^2 \partial_{x_d}^+w-\alpha_-^2\partial_{x_d}^-w] + o(\delta)\]
     which reduces to
     \[\alpha_+^2 \partial_{x_d}^+ w -\alpha_-^2\partial_{x_d}^- w = 0 \ \hbox{ on } \ \partial \{u>0\}.\]
     This formal computation motivates the following \emph{transmission problem} for the first order term in the asymptotic expansion for $\alpha\geq0$
     \begin{equation}\label{e.transmission}\tag{T-$\alpha$}
     \begin{cases}
      \Delta w = 0 &\hbox{in } \ U^+ \cup U^-\\
      \textup{tr}_+ w = \textup{tr}_- w  &\hbox{on }  U'\\
      (1+\alpha^2)\partial_{x_d}^+ w -\alpha^2\partial_{x_d}^- w= 0 &\hbox{on } U',
     \end{cases}
     \end{equation}
     where $U^\pm := U \cap \{\pm x_d>0\}$, $U':=U \cap \{x_d = 0\}$, $\partial_{x_d}^\pm$ are respectively the derivatives from $U^\pm$ on $U'$, and $\textup{tr}_\pm$ are the traces on $U'$ from $U^\pm$ respectively.

Note that the transmission problem makes sense in the degenerate case $\alpha = 0$, as it becomes a Neumann problem in $U^+$ and a Dirichlet problem in $U^-$ with Dirichlet data on $U'$ given by $\textup{tr}_+ w$. We also define the formal $\alpha \to \infty$ limit to be so that the normal derivatives must match on $U'$, and so $w$ is actually harmonic in the entire domain:
\begin{equation}\label{e.transmission-infty}\tag{T-$\infty$}
  \Delta w = 0 \ \hbox{ in } \ U.
 \end{equation}

We will interpret \eref{transmission} in the viscosity sense.  
\begin{definition}
\begin{itemize}
\item We say that $w \in C(U)$ is a viscosity supersolution of \eref{transmission} if $w$ is harmonic in $U^+$ and $U^-$, and whenever $P$ is a quadratic polynomial of the form $P(x) = P'(x') + \lambda x_d^2$ with $\Delta P < 0$, and $\gamma_\pm \in \R$ so that
\[P(x)+ \gamma_+(x_d)_+ - \gamma_-(x_d)_-\]
touches $w$ from above strictly at $x_0 \in U'$ then
\[ (1+\alpha^2)\gamma_+ - \alpha^2\gamma_-\geq 0.\]
\item We say that $w \in C(U)$ is a viscosity subsolution of \eref{transmission} if $w$ is harmonic in $U^+$ and $U^-$, and whenever $P$ is a quadratic polynomial of the form $P(x) = P'(x') + \lambda x_d^2$ with $\Delta P >0$, and $\gamma_\pm \in \R$ so that
\[P(x)+ \gamma_+(x_d)_+ - \gamma_-(x_d)_-\]
touches $w$ from below strictly at $x_0 \in U'$ then
\[ (1+\alpha^2)\gamma_+ - \alpha^2\gamma_-\leq 0.\]
\item We say that $w \in C(U)$ solves \eref{transmission} if it is both a sub and supersolution. 
\end{itemize}
\end{definition}
We will need a continuity property of the set of solutions to \eref{transmission} with respect to $\alpha$, which follows from \cite{DeSilvaFerrariSalsa-APDE2014}*{Theorems 3.3 and 3.4}.  In particular the proof of \cite{DeSilvaFerrariSalsa-APDE2014}*{Theorem 3.3} shows an explicit representation for the solution of \eref{transmission} from which the following result can be deduced.
\begin{lemma}\label{l.varying-alpha}
Let $h \in C(\partial B_1)$. For all $\alpha \in [0,+\infty]$ there is a unique viscosity solution of \eref{transmission} $w_\alpha$ with $w_\alpha = h$ on $\partial B_1$ which is also a classical solution.  Furthermore $w_{\alpha'} \to w_{\alpha}$ uniformly in $\overline{B_1}$ as $\alpha' \to \alpha$.
\end{lemma}
We also quote the following $C^{1,1}$ estimate from   \cite{DeSilvaFerrariSalsa-APDE2014}*{Theorem 3.2}) for $w$ solving \eref{transmission} in $B_1$.
\begin{lemma}\label{l.C11-estimate-transmission}
 There is $\bar{C}(d) \geq 1$ independent of $\alpha \in [0,\infty]$ so that if $w$ solves \eref{transmission} in $B_1$
\[\|\grad w\|_{L^\infty(B_{1/2}^{\pm})}+\|D^2w\|_{L^\infty(B_{1/2}^{\pm})} \leq \bar{C} \|w\|_{L^{\infty}(B_1)},\]
and, more specifically, if $w \in \mathcal{T}_{\a}$, then
\[|w(x) - w(0) - \grad_{x'} w (0) \cdot x' - \partial_{x_d}^+w(0)(x_d)_+ -\partial_{x_d}^-w(0)(x_d)_- | \leq \bar{C}r^{2} \quad \hbox{for all } \ |x| \leq r \leq \frac{1}{4}\]
with
\[|\grad_{x'} w (0)|+|\partial_{x_d}^+w(0)|+|\partial_{x_d}^-w(0)| \leq \bar{C}.\]
\end{lemma}

    \subsection{Rigorous asymptotic expansion of flat solutions}\label{s.asymptotic-exp-2} Now we make the formal asymptotic expansion of flat solutions rigorous following the strategy introduced in \cites{DeSilva,DeSilvaSavinAM}.
 \begin{proposition}\label{p.asymptotic-expansion}
Let $u \in H^1(B_1)$ satisfy Property~\ref{pty.holder-est} and $u(0) = 0$. For all $\eta>0$, there exists $\delta_0(\eta,d,\kappa,\gamma)>0$ such that the following holds: if, for some $\alpha >0$ and $0 < \delta \leq \delta_0$,
\begin{enumerate}[label = $\bullet$]
\item $u$ satisfies \eref{sigma-almost-min} with $\sigma \leq  \alpha_+^{-\frac{d}{\gamma}}\alpha^{2+\frac{d}{\gamma}}\delta^{3 + \frac{d}{\gamma}}$, and
\item $u$ is $(\a,\delta)$-flat in $B_1$, i.e.\begin{equation}\label{e.delta-flatness-of-quasi-minimizer}
\Phi_\alpha(x_d - \delta) \leq u(x) \leq \Phi_\alpha(x_d + \delta) \ \hbox{ for all } \ x \in B_1,
\end{equation}
\end{enumerate}
then there exists $w$ solving \eref{transmission} in $B_{1/2}$ with $|w| \leq 1$ and $w(0) = 0$ such that
\begin{equation}\label{e.asymptotic-expansion-of-flat-quasiminimizer}
\Phi_\alpha(x_d +\delta w(x) - \eta \delta) \leq u(x) \leq \Phi_\alpha(x_d +\delta w(x) + \eta \delta) \ \hbox{ in } \ B_{1/2}.
\end{equation}
\end{proposition}
Note that although we allow $\alpha$ to be arbitrarily small, the approximate minimality requirement becomes more and more restrictive as $\alpha \to 0$, so one should think of Proposition \ref{p.asymptotic-expansion} as only dealing with the non-degenerate case.  We expect that some version of this result is also true in the degenerate case $\alpha \to 0$, but we do not address that in this paper.
\begin{proof}[Proof of \pref{asymptotic-expansion}]
The proof is by compactness. We split it into two steps for readability. 
    
    \emph{Step 1.}
     Suppose, towards a contradiction, that there is $\eta>0$, $\delta_k \to 0$, $\alpha_k > 0$, and $u_k$ satisfying $u_k(0) = 0$, interior H\"older estimates at all scales as in Property~\ref{pty.holder-est} and \eref{sigma-almost-min} for $\sigma = \sigma_k \leq (\alpha_k)_+^{-\frac{d}{\gamma}}\alpha_k^{2+\frac{d}{\gamma}}\delta_k^{3 + \frac{d}{\gamma}}$ with
  \[\Phi_{\alpha_k}(x_d - \delta_k) \leq u_k(x) \leq \Phi_{\alpha_k}(x_d + \delta_k)\]
     but
     \begin{equation}\label{e.transmission-contradiction}\sup_{B_{1/2}}\max\{\Phi_{\alpha_k}(x_d +\delta_k w(x) - \eta \delta_k) - u_k(x) , u_k(x) - \Phi_{\alpha_k}(x_d +\delta_k w(x) + \eta \delta_k)\} > 0
     \end{equation}
     for all $|w| \leq 1$ solving \hyperref[e.transmission]{(T-$\alpha_k$)} with $w(0) = 0$. 
     
     Since $\alpha_k>0$ we can define
     \[w_k (x) := \frac{\Psi_{\alpha_k}(u_k(x)) - x_d}{\delta_k}. \]
     By assumption $|w_k| \leq 1$. Applying \tref{holder-est-w}, the $w_k$ are uniformly H\"older continuous above a scale which goes to zero as $\delta_k \to 0$. 
     Therefore, up to a subsequence, the $w_k$ converge uniformly in $B_{1/2}$ to a H\"older continuous $\bar{w}$ (which satisfies the same H\"older estimate as in \tref{holder-est-w}) with $|\bar{w}| \leq 1$, and the $\alpha_k \to \bar{\alpha} \in [0,+\infty]$.
    
     We will show in Step 2 below that $\bar{w}$ solves the transmission problem \hyperref[e.transmission]{(T-$\bar{\alpha}$)} in $B_{1}$. Before proving this though, let us show how to derive a contradiction.  By \lref{varying-alpha} there is a solution $\bar{w}_k$ of \hyperref[e.transmission]{(T-$\alpha_k$)} in $B_{1/2}$ with
     \[\|\bar{w}_k - \bar{w}\|_{L^\infty(B_{1/2})} \to 0 \ \hbox{ as } \ k \to \infty.\]
     By considering the redefinition $\bar{w}_k \mapsto \bar{w}_k - \bar{w}_k(0)$, which still converges uniformly to $\bar{w}$ since $\lim_{k \to \infty} \bar{w}_k(0)  = \bar{w}(0) = 0$, we can assume that $\bar{w}_k(0) =0$. Now applying \eref{transmission-contradiction} with $\bar{w}_k$,
    \begin{align*}
    \eta \delta_k  &\leq \sup_{B_{1/2}}|\Psi_{\alpha_k}(u_k(x)) - (x_d + \delta_k\bar{w}_k(x))| \\
    &\leq \delta_k \left[\sup_{B_{1/2}} |\bar{w}_k - \bar{w}|+\sup_{B_{1/2}} |\bar{w} - w_k|\right].
     \end{align*}
     Since the term in square brackets goes to $0$ as $k \to\infty$, we obtain a contradiction for $k$ large enough. 
     
      \emph{Step 2.} We now show $\bar{w}$ solves the transmission problem \hyperref[e.transmission]{(T-$\bar{\alpha}$)} in $B_{1/2}$.  The continuity of $\bar{w}$ in $B_1$ is immediate from the uniform convergence and the uniform H\"older continuity property of the $w_k$. By \lref{harmonic-approximation} and the fact that $\delta_k^{-1}\sigma_k^{\frac{\gamma}{d+2\gamma}}\to 0$, we know that $\bar{w}$ is harmonic in $B_{1/2}^+ \cup B_{1/2}^-$.

     Next, we verify the transmission condition. Suppose $P(x) = P'(x') + \lambda x_d^2$ is a quadratic polynomial with $\Delta P < 0$ and
     \[\varphi(x) := P(x) + \gamma_+(x_d)_+ - \gamma_-(x_d)_-.\]
     Suppose that $\varphi$ touches $\bar{w}$ from above strictly at some $x_0\in B_{1/2} \cap \{x_d = 0\}$. In particular $\bar{w}(x) - \varphi(x)$ has a strict local maximum value of $0$ at $x_0 \in B_{r}(x_0)$. And suppose that the transmission condition is violated, i.e. $\gamma_\pm$ are such that, in the case $\bar{\alpha} \in [0, + \infty)$,
     \[L_{\bar{\alpha}}(\gamma_\pm):=(1+\bar{\alpha}^2)\gamma_+- \bar{\alpha}^2\gamma_- < 0,\]
     or, in the case $\bar{\alpha} = +\infty$,
     \[\gamma_+ - \gamma_- < 0.\] 
     Note that we can always increase $\gamma_+$ while maintaining the fact that $\varphi$ touches $\bar{w}$ (strictly) from above at $x_0$, so, up to making a $k$-dependent re-definition of $0 \leq \gamma_+^k < \gamma_-$ we can assume that
     \begin{equation}\label{e.gammak-purpose}
     -\frac{1}{4} \leq \alpha_k^2\delta_k(\gamma^k_+ - \gamma_-)  < 0.
     \end{equation}
     We will drop the $k$ dependence of $\gamma_+^k$ since it does not play a significant role.

      For any $t \in \R$, \lref{moving-gradient-dc-2} applied to $\varphi(x) + t$ implies
     \[\varphi(x)+t = \delta_k^{-1}\bigg[\Phi_{1+ \delta_k \gamma_{\pm}}(x_d + \delta_k (t+P(x))) -x_d\bigg] + O(\delta_k). \]
     Thus
     \[w_k(x) - \delta_k^{-1}\bigg[\Phi_{1+ \delta_k \gamma_{\pm}}(x_d + \delta_k (t+P(x))) -x_d\bigg] \to \bar{w}(x) - \varphi(x) - t \ \hbox{ uniformly as $k \to \infty$.}\]
     Since $\bar{w}(x) - \varphi(x) - t$ crosses $0$ strictly from below at $t=0$ in $B_{r}(x_0)$ there is a sequence of $(x_k,t_k) \to (x_0,0)$ so that
     \[\tilde{\varphi}_k(x):=\delta_k^{-1}\bigg[\Phi_{1+ \delta_k \gamma_{\pm}}\big(x_d + \delta_k (t_k+P(x))\big) -x_d\bigg]\]
     touches $w_k$ from above at $x_k$ in $B_r(x_0)$.  Or, unwinding the definition of $w_k$, 
     \[\varphi_k(x):=\Phi_{\beta_k^\pm}(x_d + \delta_k (t_k+P(x')))\]
     touches $u_k$ from above at $x_k$ where
     \[\beta_k^+ = \sqrt{1+\alpha_k^2}(1+\delta_k\gamma_+) \ \hbox{ and } \ \beta_k^-:= \alpha_k (1+\delta_k \gamma_-).\]

     Now we aim to apply \cref{approximate-subsoln-cond-B}.  First we compute, in the $\bar{\alpha} < +\infty$ case,
     \begin{align*}
         (\beta_k^+)^2 - (\beta_k^-)^2 &= (1+\alpha_k^2)(1+ \delta_k \gamma_{+})^2 - \alpha_k^2(1+ \delta_k \gamma_{-})^2  \\
         &= 1+2\delta_kL_{\alpha_k}(\gamma_\pm)+O(\alpha_k^2\delta_k^2)
     \end{align*}
     so that
     \[1+2\delta_k L_{\a_k}(\gamma_\pm)+O(\alpha_k^2\delta_k^2) \leq 1 - \frac{3}{4}\delta_k |L_{\a_k}(\gamma_\pm)|\]
     for $\delta_k$ small enough. 
    And, provided $r > 0$ sufficiently small,
    \[|\partial_{x_d} P(x)| \leq Cr \leq \frac{1}{8}|L_{\bar{\alpha}}(\gamma_\pm)|.\]
    Finally recall that, by the assumption of Property \ref{pty.holder-est}, $[u_k]_{C^\gamma(B_{1})} \leq \HolderConst(1+\osc_{B_{1}} u_k) \leq C\HolderConst \alpha_+$.

    Now applying \cref{approximate-subsoln-cond-B} in $B_r(x_0)$ with $\mu_k := \delta_k\min\{\frac{1}{2}|L_{\bar{\alpha}}(\gamma_\pm)|,r|\Delta P|\}$ we obtain a contradiction for sufficiently large $k$ as long as, using the energy rescaling property Lemma \ref{lem:scaleinvarianceunitscale},
    \[\sigma_k \leq cr^d\HolderConst^{-\frac{d}{\gamma}}\alpha_+^{-\frac{d}{\gamma}}\alpha_k^2 \mu_k^{2+\frac{d}{\gamma}} \leq cr^d[u_k]_{C^{0,\gamma}(B_r(x_0))}^{-\frac{d}{\gamma}}\alpha_k^2 \mu_k^{2+\frac{d}{\gamma}} \]
    which holds for sufficiently large $k$ since, by hypothesis, $\sigma_k \leq (\alpha_k)_+^{-\frac{d}{\gamma}}\alpha_k^{2+\frac{d}{\gamma}}\delta_k^{2 + \frac{d}{\gamma}}\delta_k$. 

    In the case $\bar{\alpha} = +\infty$ then we group terms differently to arrive at
      \[(\beta_k^+)^2 - (\beta_k^-)^2 = 1+\alpha_k^2\delta_k(\gamma_+-\gamma_-)[1+\delta_k(\gamma_++\gamma_-)]+2\gamma_+\delta_k+\gamma_+^2\delta_k^2\]
      so that, for $k$ sufficiently large,
      \[(\beta_k^+)^2 - (\beta_k^-)^2 \leq 1 - \frac{1}{2}\alpha_k^2\delta_k|\gamma_+-\gamma_-|.\]
     Note that, by \eref{gammak-purpose}, $\alpha_k^2\delta_k|\gamma_+-\gamma_-| \leq 1/4$ so we can apply \cref{approximate-subsoln-cond-B} in $B_r(x_0)$ with $\mu_k := \delta_k\min\{\frac{1}{2}\alpha_k^2|\gamma_+-\gamma_-|,r|\Delta P|\}$ and get a contradiction as in the previous case.

     \end{proof}

We are now ready to prove the main result of this section.

         \subsection{Proof of \tref{onestep-improve-flatness}}\label{s.asymptotic-exp-3} 
         Fix $\beta \in (0,1)$ and $\a \geq \underline{\a} > 0$. Let $\bar{C} \geq 1$ be as in \lref{C11-estimate-transmission}. Choose $\theta, \eta > 0$ small enough so that $\bar{C}\theta^2 + \eta \leq \frac{1}{2}\theta^{1+\beta}$. For this choice of $\eta$, let $\delta_0(\eta)$ be as in \pref{asymptotic-expansion}. 
         
          Let $0 < \delta < \delta_0 <\delta_d$ (where $\delta_d$ is from \tref{holder-est-w}). Suppose $u$ satisfies \eref{sigma-almost-min} with $\sigma \leq \alpha_+^{-\frac{d}{\gamma}} \a^{2+\frac{d}{\gamma}} \delta^{3 + \frac{d}{\gamma}}$ and
    \[
     \Phi_\alpha(x_d - \delta) \leq u(x) \leq \Phi_\alpha(x_d + \delta) \ \hbox{ in } \ B_1.
     \]
    Note that if $\sigma > \alpha_+^{-\frac{d}{\gamma}} \a^{2+\frac{d}{\gamma}} \delta^{3 + \frac{d}{\gamma}}$, then $u$ is $(\a, \tilde{\delta})$-flat with
        \[\tilde{\delta} := \left(\frac{\alpha_+}{\alpha}\right)^{\frac{d}{d+3\gamma}}\alpha^{-\frac{2\gamma}{d+3\gamma}}\sigma^{\frac{\gamma}{d+3\gamma}}\]
    so we can proceed with the rest of the proof with this choice of $\delta$. Note that the pre-factor $(\frac{\alpha_+}{\alpha})^{\frac{d}{d+3\gamma}}\alpha^{-\frac{2\gamma}{d+3\gamma}}$ is bounded on $\alpha \in [\underline{\alpha},\infty)$ for any $\underline{\alpha}>0$ so we can just treat it as an $\underline{\alpha}$-dependent factor $C(\underline{\alpha})$ as in the statement of the Theorem.
         
By \pref{asymptotic-expansion}, there is $w$ solving \eref{transmission} in $B_{1/2}$ with $|w| \leq 1$ and $w(0) = 0$ so that
\[\Phi_\alpha(x_d +\delta w(x) - \eta \delta) \leq u(x) \leq \Phi_\alpha(x_d +\delta w(x) + \eta \delta) \ \hbox{ in } \ B_{1/2}.\]
Define
\[\tau := \grad_{x'} w (0), \qquad \gamma_\pm:= \partial_{x_d}^\pm w(0), \qquad \hbox{and } \ e := \cos(\delta|\tau|) e_d + \sin(\delta|\tau|) \frac{\tau}{|\tau|}.\]
Since $w$ solves \eref{transmission}, we have $L_\alpha(\gamma_\pm) = 0$, and by \lref{C11-estimate-transmission}
\[|\gamma_\pm| + |\tau| \leq \bar{C}\]
for a universal $\bar{C}$, independent of $\alpha$. In particular
\[|e - e_d| \leq C\delta|\tau| \leq C\delta.\]
Applying \lref{moving-gradient-dc-1} with $P = w$, we obtain for all $|x| \leq \theta$
         \begin{align*}
         x_d + \delta w(x) + \eta \delta &  \leq (1+\delta \gamma_+)(x \cdot e)_+ - (1+\delta \gamma_-)(x \cdot e)_-+(\bar{C}|x|^2 + \eta)\delta + O(\delta^2) \\
         & \leq \Phi_{(1+\delta\gamma_\pm)}\left(x \cdot e + \tfrac{1}{2}\theta^{1+\beta} \delta + O(\delta^2)\right).
         \end{align*}
         Similarly
         \[x_d + \delta w(x) - \eta \delta \geq \Phi_{(1+\delta\gamma_\pm)}\left(x \cdot e - \tfrac{1}{2}\theta^{1+\beta}\delta - O(\delta^2)\right).\]
Then \lref{visc-soln-equiv} allows us to rewrite
\[\Phi_{\a}\circ\Phi_{(1+\delta\gamma_\pm)}\left(x \cdot e + \tfrac{1}{2}\theta^{1+\beta} \delta + O(\delta^2)\right)  =\Phi_{\alpha'}(\lambda(x\cdot e + \tfrac{1}{2}\theta^{1+\beta} \delta + O(\delta^2))),\]
where $\lambda$ and $\alpha'$ are, using $L_\alpha(\gamma_\pm)=0$,
         \[\lambda = 1 - \gamma_+\gamma_-\delta^2\]
        and
       \[\alpha' = \alpha\frac{1+\delta\gamma_-}{1-\gamma_+\gamma_-\delta^2} \]
       and therefore
       \[|\alpha'-\alpha|\leq C\alpha \delta\ \ \hbox{with $C$ universal.}\]
        It follows that
       \[\lambda(x\cdot e + \tfrac{1}{2}\theta^{1+\beta} \delta + O(\delta^2)) = x\cdot e + \tfrac{1}{2}\theta^{1+\beta} \delta + O(\delta^2).\]
       Note that all the terms $O(\delta^2) \leq C\delta^2$ with $C$ universal (no $\alpha$ dependence), by \lref{visc-soln-equiv} and $L_\alpha(\gamma_\pm)=0$. Therefore, choosing $\bar{\delta}$ small enough we can guarantee that $C\delta^2 \leq \tfrac{1}{2}\theta^{1+\beta}\delta$ so that
       \[\Phi_\alpha(x_d + \delta w(x) + \eta \delta) \leq \Phi_{\alpha'}(x \cdot e + \theta^{1+\beta}\delta)\]
       and similarly
       \[\Phi_\alpha(x_d + \delta w(x) - \eta \delta) \geq \Phi_{\alpha'}(x \cdot e - \theta^{1+\beta}\delta).\]

\part{Large scale regularity of $\J$ minimizers in periodic media.}

Our goal in this part of the paper is to use the results developed in Part 1 for approximate minimizers of the homogenized functional $\J_0$ to prove large scale regularity of minimizers of the inhomogeneous functional $\J$. 

\section{Homogenization setting and preliminary results}
\subsection{Assumptions on Coefficients of Two-Phase Functional} As described in the introduction, we will study the two-phase energy functional
\begin{equation*}
\mathcal{J}(u,U) = \int_U \grad u \cdot a(x)\grad u + Q_+(x)^2{\bf 1}_{\{u>0\}} + Q_-(x)^2{\bf 1}_{\{u\leq 0\}} \ dx.
\end{equation*}

We make precise the assumptions on the coefficients which will be in place for the remainder of this part, unless otherwise specified in some few locations.

\begin{enumerate}[label = (Q\arabic*)]
\item \label{a.Q1} The coefficients $Q_{\pm}: \R^d \to (0,\infty)$ satisfy the ellipticity condition
\begin{equation*}\label{e.Qellipticity}
\Lambda^{-1} \leq Q_{\pm}(x) \leq \Lambda.
\end{equation*}
\item \label{a.Q2} The coefficients $Q_{\pm}$ are  $\Z^d$-periodic and satisfy the normalization
\begin{equation*}\label{e.Qnormalization}\langle Q_+^2 \rangle - \langle Q_-^2 \rangle = 1.
\end{equation*}
\end{enumerate}
\begin{enumerate}[label = (a\arabic*)]
\item \label{a.a1} The coefficient field $a: \R^d \to M_{d \times d}^{sym}(\R)$ is measurable and satisfies the ellipticity bounds
\begin{equation*}\label{e.aellipticity}
\Lambda^{-1} I\leq a(x) \leq \Lambda I 
\end{equation*}
\item \label{a.a3}  The coefficient field $a: \R^d \to M_{d \times d}^{sym}(\R)$ is $\Z^d$-periodic.
\end{enumerate}

Constants $C$ and $c$ in the text below which depend at most on $d$, $\Lambda$ are called \emph{universal}.  Such constants may change from line to line without mention.  If we intend to fix the value of such a universal constant for a portion of an argument, we may denote it as $c_0$, $C_0$, $c_1$ etc.

\subsection{Initial regularity of $\mathcal{J}$-minimizers}\label{s.regularityofJminimizer} Let us discuss the rough (below Lipschitz) regularity estimates of quasi-minimizers of
\[\mathcal{J}(u,U) = \int_U \grad u \cdot a(x)\grad u + Q_+(x)^2{\bf 1}_{\{u>0\}} + Q_-(x)^2{\bf 1}_{\{u\leq 0\}} \ dx.\]
For the purposes of the present section we only assume that $a(x)$ is uniformly elliptic; no regularity or periodicity is necessary.

We consider quasi-minimizers of $\mathcal{J}$ in the following sense.  
\begin{equation}\label{e.J-quasi-minimal}
   \mathcal{J}(u, B_r) \leq (1+\sigma^2)\mathcal{J}(v, B_r)+\int_{B_r} \kappa(x)^2 \ dx \qquad \hbox{for all } v \in u + H^1_0(B_r).
\end{equation}
If $\sigma \leq 1$, then 
\begin{equation}\label{e.Dirichlet-quasiminimizer}
\int_{B_r} \grad u \cdot a(x)  \grad u \ dx \leq (1+\sigma^2)\int_{B_r} \grad v \cdot a(x)  \grad v \ dx + \int_{B_r} 4\max Q_\pm^2 + \kappa^2(x) \ dx.
\end{equation}
This shows any function $u$ satisfying \eref{J-quasi-minimal} is also a quasi-minimizer of the $a(x)$-Dirichlet energy
\[E(u,U) = \int_{U} \grad u\cdot a(x) \grad u \ dx.\]
Thus, we can derive certain regularity properties for $\mathcal{J}$ quasi-minimizers from analogous known results for Dirichlet quasi-minimizers with bounded, measurable coefficients.  This includes both interior and global (up to the boundary) regularity estimates.

In Appendix \ref{s.regularity-appendix} we explain, in detail, how to derive H\"older and $W^{1,p}$ estimates for quasi-minimizers of Dirichlet-type energies. While these results are not new, we have chosen to include detailed proofs in the appendix as a handy reference for readers starting out in this subject. Here, we will simply state some corollaries of the general Dirichlet energy quasi-minimizer regularity theory as applied to two-phase quasi-minimizers. 

First, we state the interior and global H\"older estimates; see \lref{interior-holder-est-quasi} for the interior estimates and \lref{global-holder-est-quasi} for the global estimate. To express the estimates in a concise fashion, we introduce the semi-norm
\[[f]_{\mathcal{C}^{0,\gamma}(B_r(0))} := \sup_{ B_\rho \subset B_{2r}(0) } \rho^{1-\gamma}\|\grad f\|_{\underline{L}^2(B_\rho \cap B_r(0))}.\]
By Poincar\'e inequality and the standard Morrey-Campanato characterization of H\"older spaces (see \cite{Schulz}*{Theorem 1.2.2} or \cite{Giaquinta-Martinazzi}*{Theorem 5.5})  this semi-norm bounds the standard H\"older semi-norm and scales in the same way:
\[[f]_{C^{0,\gamma}(B_r(0))} \leq C[f]_{\mathcal{C}^{0,\gamma}(B_r(0))}.\]

\begin{lemma}[H\"older estimates]\label{l.holder-FB}
There exist constants $ \sigma_0, \gamma_0 \in (0,1)$ and $C \geq 1$ depending only on $d, \Lambda$ such that if $u$ is a $\mathcal{J}$ quasi-minimizer in the sense of \eref{J-quasi-minimal} in $B_r$ with $0 \leq \sigma \leq \sigma_0$ then, for any $0 < \gamma \leq \gamma_0$, we have
\begin{itemize}
    \item (Interior estimates) For all $0 < \rho \leq r$ 
    \[\|\grad u\|_{\underline{L}^2(B_{\rho})} \leq C(\rho/r)^{\gamma-1}(\|\grad u\|_{\underline{L}^2(B_{r})}+\max_{\pm}\|Q_\pm\|_{\underline{L}^{\frac{d}{1-\gamma}}(B_r)}+\|\kappa\|_{\underline{L}^{\frac{d}{1-\gamma}}(B_{r})}),\]
    and, in particular, 
    \[[u]_{\mathcal{C}^{0,\gamma}(B_{r/2})} \leq Cr^{1-\gamma}(\|\grad u\|_{\underline{L}^2(B_{r})}+\max_{\pm}\|Q_\pm\|_{\underline{L}^{\frac{d}{1-\gamma}}(B_r)}+\|\kappa\|_{\underline{L}^{\frac{d}{1-\gamma}}(B_{r})}).\]
    \item (Global estimates) If also $u \in g+H^1_0(B_r)$ then
    \[[u]_{\mathcal{C}^{0,\gamma}(B_{r})} \leq C[g]_{\mathcal{C}^{0,\gamma}(B_{r})}+Cr^{1-\gamma}(\max_{\pm}\|Q_\pm\|_{\underline{L}^{\frac{d}{1-\gamma}}(B_r)}+\|\kappa\|_{\underline{L}^{\frac{d}{1-\gamma}}(B_{r})}).\]
\end{itemize}
\end{lemma}
Next, we state the interior and global Meyers' estimates; see \tref{meyers-CZ-sigma-global}.

\begin{lemma}[Meyers-type estimates]\label{l.meyers-FB}
There exist constants $\sigma_0>0$, $p_0(\Lambda,d) > 2$ and $C(\Lambda,d,\sigma) \geq 1$ such that if $u$ is a $\mathcal{J}$ quasi-minimizer in the sense of \eref{J-quasi-minimal} in $B_r$ with $0 \leq \sigma \leq \sigma_0$ then, for any $2 \leq p \leq p_0$:
\begin{itemize}
    \item (Interior estimates) \[\|\grad u\|_{\underline{L}^p(B_{r/2})} \leq C(\|\grad u\|_{\underline{L}^2(B_{r})}+\max_{\pm}\|Q_\pm\|_{\underline{L}^p(B_r)}+\|\kappa\|_{\underline{L}^p(B_{r})}).\]
    \item (Global estimates) If also $u \in g+H^1_0(B_r)$
        \[\|\grad u\|_{\underline{L}^p(B_{r})} \leq C(\|\grad g\|_{\underline{L}^p(B_{r})}+\max_{\pm}\|Q_\pm\|_{\underline{L}^p(B_r)}+\|\kappa\|_{\underline{L}^p(B_{r})}).\]
\end{itemize}
\end{lemma}

Note that these estimates can be significantly sharpened under stronger hypotheses on $a$ and $\sigma$; see Remarks \ref{Holder-estimate-example-scenarios} and \ref{CZ-estimate-example-scenarios}. However, they are always limited to sub-Lipschitz regularity. For example, using also the periodic structure of $a$, it would be possible to show (above unit scale) $C^{0,\gamma}$ and $W^{1,p}$ estimates for \emph{all} $\gamma \in (0,1)$ and $2 \leq p < +\infty$. The small $\gamma$-H\"older regularity, and small extra integrability $p>2$ of the gradient shown in the results above come solely from uniform ellipticity and are sufficient for our purposes. We work with these weaker, but sufficient, conclusions to avoid the additional notation that comes with the sharper estimates that only hold above unit scale.

We show one additional quick application of the same techniques.
\begin{lemma}
    If $u \in g+H^1(B_r)$ minimizes $\mathcal{J}_0$ in $B_r$ then for every $2 \leq p < \infty$
    \[\|\grad u\|_{\underline{L}^p(B_{r})} \leq C(\|\grad g\|_{\underline{L}^p(B_{r})}+\max_{\pm}\|Q_\pm\|_{\underline{L}^p(B_r)}).\]
\end{lemma}
This result was also proved as a part of the paper \cite{snelson2024bernoulli}. The proofs presented in Appendix \ref{s.regularity-appendix} indicate that the particular structure of the one-phase or two-phase problem is not essential to prove this kind of rough regularity results; the quasi-minimality property with respect to the Dirichlet energy suffices. See also the boundary regularity result for Bernoulli solutions in \cite{FernandezReal} and \cite{Feldman}*{Lemma 5.6}.

\begin{remark}
It is interesting to note that two-phase minimizers satisfy an interior Cacciopoli estimate and so the interior Meyers-type estimate can be proved in the usual way via Gehring's Lemma in that case. However, we were not able to derive a boundary Cacciopoli estimate for two-phase minimizers, and so we could not prove the global Meyers-type estimate in the classical way. Since subtracting off the boundary data ruins the particular structure of the energy functional anyway, we found it preferable to treat the more general case of Dirichlet quasi-minimizers, which behave well with respect to subtracting off the boundary data.
\end{remark}

\subsection{Energy estimates near the free and fixed boundary}
The key domain regularity estimate needed to prove the algebraic rate of homogenization is the following standard boundary strip energy estimate for $\mathcal{J}$ minimizers. This estimate requires only ellipticity bounds on the coefficients in the energy functional. 
\begin{lemma}[Boundary strip energy estimate (interior)]\label{l.bdry-strip}
Suppose $u$ minimizes $\mathcal{J}$ over $u + H^1_0(B_{2r})$. Then there exists $C \geq 1$ universal such that
\[\int_{\{ 0 < |u| < s\} \cap B_r} |\grad u|^2 + 1 \ dx \leq C( 1+ \|\grad u\|_{\underline{L}^2(B_{2r})}) r^{d-1}s \qquad \hbox{for all } 0 < s \leq r/2.\]
\end{lemma}
The proof follows the standard approach in the literature making an energy comparison argument with a competitor which zeros out the $ -t \leq u \leq t$ levels in $B_{r}$. See \cite{VelichkovBook}*{Lemma 5.6} for the details in the one-phase case, which is very similar. 

We will also need a global version of the previous estimate that holds up to the fixed boundary. This result will only be applied to $\mathcal{J}_0$-minimizers, specifically a minimizer of $\mathcal{J}_0$ in $B_r$ with boundary data given by a $\mathcal{J}$-minimizer in $B_{2r}$ (in other words, a kind of $\mathcal{J}_0$ harmonic replacement). Thus, we can consider boundary data that have the interior regularity of a $\mathcal{J}$ minimizer.
 \begin{lemma}[Boundary strip energy estimate (global)]\label{l.bdry-strip-dirichlet}
   Suppose $g \in H^1(B_{2r})$ satisfies, for some $p>2$ and $\gamma := 1-\frac{2}{p}$, the bounds
   \[ \|\grad g\|_{\underline{L}^p(B_r)}+r^{\gamma-1} [g]_{\mathcal{C}^{0,\gamma}(B_r(0))} =:G < + \infty .\]
   Let $u$ minimize $\mathcal{J}$ over $g + H^1_0(B_r)$. Then there exists $C \geq 1$ universal and $\gamma_0(\Lambda,d)$ such that
   \[\int_{\{0<|u| < t\} \cup \{r-t < |x| < r\}} |\grad u|^2 + 1 \ dx \leq C(1+G^2)(t/r)^{\gamma'}|B_r| \qquad \hbox{for all } 0 < t \leq r/4\]
   where $\gamma' = \min\{\gamma,\gamma_0\}$.
\end{lemma}
We will actually only use \lref{bdry-strip-dirichlet} in the even more restricted class of $\mathcal{J}_0$ minimizers; the general uniformly elliptic case is stated because the proof is the same.
\begin{proof}
    Without loss assume that $G \geq 1$. It suffices to consider the case $t = 2^{-N}r$ for some integer $N \geq 2$. Let $A_0 = B_{r/2}$ and define the annuli \[A_j :=\{(1-2^{-j}) r < |x| < (1-2^{-(j+1)})r\} \ \hbox{ for } \ 1 \leq j \leq N-1,\]
    and
    \[A_N:= \{(1-2^{-N}) r < |x| < r\} = \{r-t < |x| < r\}.\]
    Each $A_j$ can be covered by $M_j = C(d)2^{j(d-1)}$ balls $B_j$ of radius $2^{-(j+1)}r$.  For $0 \leq j \leq N-1$ we can also arrange that for each such covering ball $B_j$ the dilated $\frac{3}{2}B_j$ with $3/2$ times the radius and the same center is contained in $B_r(0)$.  For $0 \leq j \leq N-1$ and each such covering ball $B_j \subset \frac{3}{2} B_j \subset B_r(0)$, we can apply \lref{bdry-strip} and then the global H\"older estimate \lref{holder-FB} to find
    \begin{align*}
        \int_{\{0<|u| < t\} \cap B_j} |\grad u|^2 + 1 \ dx &\leq C(1+\|\grad u\|_{\underline{L}^2(\frac{3}{2}B_j)})t (2^{-j}r)^{d-1}\\
        &\leq CG2^{j(1-\gamma)}  (2^{-j}r)^{d-1} t.
    \end{align*}
    So summing over the covering balls we find
    \begin{align*}
        \int_{\{0<|u| < t\} \cup \{|x| < r-t \}} |\grad u|^2 + 1 \ dx &\leq \sum_{0 \leq j \leq N-1}\int_{\{0<|u| < t\}  \cap A_j} |\grad u|^2 + 1 \ dx  \\
        &\leq \sum_{0 \leq j \leq N-1}  CM_j G2^{j(1-\gamma)}  (2^{-j}r)^{d-1} t \\
        &\leq CGr^{d-1}t \sum_{0 \leq j \leq N-1} 2^{j(1-\gamma)} \\
        &\leq CGr^{d-1}t 2^{N(1-\gamma)}\\
        &=CGr^{d}\tfrac{t}{r}(\tfrac{r}{t})^{1-\gamma}\\
        &=CGr^{d}(\tfrac{t}{r})^\gamma
    \end{align*}
    For $j=N$, we use the global Meyers-type estimate, \lref{meyers-FB}, instead.
       \begin{align*}
        \int_{B_r \setminus B_{r-t}} |\grad u|^2 \ dx &\leq \left(\int_{B_r \setminus B_{r-t}} |\grad u|^{p} \ dx\right)^{2/p}|B_r \setminus B_{r-t}|^{\frac{p-2}{p}} \\
        &\leq |B_r|^{2/p}\|\grad u\|_{\underline{L}^p(B_r)}^2(Ctr^{d-1})^{\frac{2-p}{p}}\\
        & = C\|\grad u\|_{\underline{L}^p(B_r)}^2(t/r)^{\frac{2-p}{p}}|B_r|^{\frac{2-p}{p}+\frac{2}{p}},\\
        &\leq C(1+\|\grad g\|_{\underline{L}^p(B_r)}^2)(t/r)^{\gamma}|B_r|.
    \end{align*}
    Noting that $\int_{B_r \setminus B_{r-t}} 1 \ dx \leq C(t/r)|B_r|$ and combining the estimates above the proof is complete.
\end{proof}

 \subsection{Periodic homogenization of divergence-form elliptic equations}\label{s.homrecall}
We recall some basic concepts from periodic elliptic homogenization, referring to \cite{Shen}*{Chapter 2} for details. Consider elliptic energy functionals of the type
\[ E(v,U) = \int_{U} \grad v\cdot a(x) \grad v \ dx\]
evaluated on a domain $U \subset \R^d$. Given $g \in W^{1,p}(U)$ for some $p \geq 2$ we can consider the Dirichlet problem
\[ \min \{ E(v,U) : v \in g + H^1_0(U)\}\]
with the associated Euler-Lagrange PDE
\begin{equation}
\begin{cases}
-\grad \cdot (a(x) \grad v) = 0 & \hbox{in } U\\
v =g & \hbox{on } \partial U
\end{cases}
\end{equation}
all interpreted in the $H^1$ weak sense. The correctors $\chi_q$ are the unique global $\Z^d$ periodic and mean zero solutions of the problem
\[ - \grad \cdot (a(x)(q+\grad \chi_q)) = 0  \ \hbox{ in } \ \R^d\]
existing for each $q \in \R^d$. Note that $\chi_q$ depends linearly on $q$.  

Let us state bounds on the correctors that we will use later. Multiplying the corrector equation by $q$ and integrating over a unit period cell, we obtain the bound
\begin{equation*}\label{e.correctorest}
\|\grad \chi_q\|_{L^2([0,1)^d)}\leq \Lambda^2|q|.
\end{equation*}
By Poincar\'e inequality and the mean zero condition, we also have
\begin{equation*}\label{e.correctorest2}
\|\chi_q\|_{L^2([0,1)^d)}\leq C|q|
\end{equation*}
and using periodicity we can also conclude that for all $r \geq 1$
\begin{equation}\label{e.correctorest3}
\|\grad \chi_q\|_{\underline{L}^2(B_r)}+\|\chi_q\|_{\underline{L}^2(B_r)}\leq C|q|
\end{equation}
for a universal $C \geq 1$.  By standard elliptic regularity theory, \eref{correctorest3} can be upgraded to an $L^\infty$ estimate if $a$ is H\"older continuous, and the new constant $C$ will also depend on the H\"older norm of $a$.

The homogenized matrix is defined as
\begin{equation}\label{e.baradef}
 \bar{a}_{ij} =\langle (e_i + \grad \chi_{e_i}(x))a(x)(e_j + \grad \chi_{e_j}(x)) \rangle 
 \end{equation}
where $e_i$ are the standard basis vectors.  We define the homogenized energy functional
\[E_0(v,U) = \int_{U}   \grad v\cdot \overline{a}\grad v \ dx. \]
From here onward, we will use the normalization $\bar{a} = I$ by a linear transformation of the lattice which we still call $\Z^d$ (see \rref{normalizationremark}), so that $E_0$ is the standard Dirichlet energy.  

We will make use of the following large scale interior estimates for $a$-harmonic functions, see Avellaneda and Lin \cite{Avellaneda-Lin}*{Lemmas 15 and 16}.

\begin{theorem}[Interior regularity]\label{t.hominteriorreg}
\begin{enumerate}[label = (\roman*)] 
\item ($C^{1,\beta}$ estimate) There exist constants $C > 0$ and $\beta \in (0,1)$ such that if $v$ is $a$-harmonic in $B_R, \ R > 1$, then for any $1 \leq r \leq R$
\[\frac{1}{r}\|v - (v(0)+\xi \cdot x + \chi_\xi)\|_{L^\infty(B_r)} \leq C\left(\frac{r}{R}\right)^{\beta}\|\grad v\|_{\underline{L}^2(B_R)}\]
with $ |\xi| \leq C\|\grad v\|_{\underline{L}^2(B_R)}$.  Moreover, since $v(0) + \xi \cdot x + \chi_\xi$ is $a$-harmonic,
\[ \|\grad v - (\xi + \grad \chi_\xi)\|_{\underline{L}^2(B_r)} \leq C\left(\frac{r}{R}\right)^{\beta}\|\grad v\|_{\underline{L}^2(B_R)}.\]
\item (Lipschitz estimate) There exists a universal constant $C> 0$ such that if $v$ is $a$-harmonic in $B_R$, $R > 1$, then for any $1 \leq r \leq R$,
\[ \|\grad v\|_{\underline{L}^2(B_r)} \leq C\|\grad v\|_{\underline{L}^2(B_R)}.\]
\end{enumerate}
\end{theorem}
Note that \cite{Avellaneda-Lin} actually proves that if $a$ is H\"older continuous and $v$ is $a$-harmonic in $B_r$ then one actually has the Lipschitz estimate
\[ \|\grad v\|_{L^\infty(B_{r/2})} \leq C\|\grad v\|_{\underline{L}^2(B_r)}.\]
The large scale Lipschitz result listed above in \tref{hominteriorreg}, with $a$ only bounded measurable, can be proved by the Schauder-type technique introduced by Armstrong and Smart \cite{Armstrong-Smart}. An analogous $C^{1,1}$ estimate is also true, the general theory in \cite{Armstrong-Kuusi-Mourrat} explains this, but the statements there are for random media and in a less convenient form for our purposes.

\section{Quantitative homogenization of the energy}\label{s.quant-hom}
In this section we establish a (suboptimal) quantitative rate of homogenization of the energy for minimizers of the two-phase functional in periodic media. 

\begin{theorem}[$\J$-minimizers are large scale $\J_0$ approximate minimizers]\label{t.energy-hom-error}
There exist $\omega , \gamma \in (0,1)$ and ${\bf r}_0\geq 1$ depending on $\Lambda$ and $d$ such that if $u$ minimizes $\mathcal{J}$ over $u + H^1_0(B_{2r}(0))$ with $r \geq {\bf r}_0$ then
\[\mathcal{J}_0(\bar{u},B_r) \leq \mathcal{J}_0(v,B_r) + C(1+\| \grad u\|_{\underline{L}^2(B_{2r})}^2)r^{-\omega}|B_r| \ \hbox{ for all } \ v \in \bar{u} + H^1_0(B_r)\]
where $\bar{u} \in u + H^1_0(B_r(0))$ is a regularization of $u$ satisfying
\begin{equation}
    \frac{1}{r}\| \bar{u} - u\|_{L^\infty(B_r)} \leq C \| \grad u\|_{\underline{L}^2(B_{2r})} r^{-\gamma},  \ \Omega_\pm(\bar{u}) = \Omega_\pm(u).
\end{equation}
and also satisfying the interior H\"older estimate
\begin{equation}
 \rho^{\gamma-1}[\bar{u}]_{C^\gamma(B_{\rho})} \leq C(1+\osc_{B_{2\rho}} \bar{u}) \ \hbox{ for all } B_{2\rho} \subset B_{r}(0).
\end{equation}
\end{theorem}

Let us make the regularization (somewhat) explicit immediately. Let $\zeta : \R \to [0,1]$ be a fixed smooth cutoff function with $\zeta(s)\equiv 0$ for $s \leq 1$, $\zeta \equiv 1$ for $s \geq 2$, and $|\zeta'| \leq C$.  Let
\begin{equation}\label{e.cutoff-dist}
    d(x):= \textup{dist}(x,\partial (\{|u|>0\} \cap B_{r})\cup \partial B_r).
\end{equation}
With $ 0 < t < r$, an intermediate scale of regularization to be chosen later as $r^\beta$ for some $ \beta \in (0,1)$, define
\begin{equation}\label{e.upscaling}
 \xi_t(x) := \dashint_{B_{t}(x)} u(y) \ dy \quad \hbox{and} \quad \bar{u}(x) := \zeta(\tfrac{d(x)}{t}) \xi_t (x) + (1-\zeta(\tfrac{d(x)}{t}))u(x).
 \end{equation}
It should be noted that $\bar{u}$ depends on the choice of scale $r$ in \tref{energy-hom-error}.

\subsection{Homogenization error for the upscaling}\label{s.upscaling}  

We begin with a boundary strip energy estimate for the regularization $\bar{u}$.

 \begin{lemma}\label{l.grad-baru-bounds}
 Suppose that $u$ minimizes $\mathcal{J}$ over $u+H^1_0(B_{2r})$. Let $0 < t \leq r/2$ arbitrary, and $d(x)$ and $\bar{u}$ as defined in \eref{cutoff-dist} and \eref{upscaling}. Then
  \[\int_{\{ 0 < d(x) < 2t\} \cap B_{r}} |\grad \bar{u}|^2 \ dx  \leq C( 1+ \|\grad u\|_{\underline{L}^2(B_{2r})}^2)(t/r)^{\gamma}|B_r|\]
  for some $\gamma(n,\Lambda) \in (0,1)$.
 \end{lemma}
\begin{proof}
Computing the gradient
\[ \grad \bar{u} = \zeta(\tfrac{d(x)}{t}) \grad \xi_t (x) + (1-\zeta(\tfrac{d(x)}{t}))\grad u(x) + \grad d(x)\zeta'(\tfrac{d(x)}{t})\frac{1}{t}( \xi_t (x) - u(x)).\]
The main interesting term is the last one on the right-hand side, which still behaves like an $L^2$ norm of the gradient when integrated. Specifically, we can estimate
\begin{align*}
    \int_{B_r} \zeta'(\tfrac{d(x)}{t})^2 (\xi_t(x) - u(x))^2 \ dx &=\int_{B_r} \zeta'(\tfrac{d(x)}{t})^2\left(\dashint_{B_t}(u(x+h)-u(x)) \ dh\right)^2 \ dx \\
    &\leq\int_{B_r} \zeta'(\tfrac{d(x)}{t})^2\dashint_{B_t}(u(x+h)-u(x))^2 \ dh \ dx\\
    &=\dashint_{B_t}\int_{B_r} \zeta'(\tfrac{d(x)}{t})^2\left(\int_0^1 \grad u(x+\lambda h) \cdot h \ d \lambda\right)^2 \ dx \ dh\\
    &\leq t^2\int_{0}^1\int_{B_t}\int_{B_r}\zeta'(\tfrac{d(x)}{t})^2|\grad u(x+\lambda h)|^2 \ dx \ dh \ d\lambda\\
    &\leq t^2 \int_{\{0 < d(x) < 3t\} \cap B_r} |\grad u(x)|^2 \ dx.
\end{align*}
A similar, and even simpler, argument gives,
\[ \int_{B_r \cap \{ 0 < d(x) < 2t\}} \zeta(\tfrac{d(x)}{t})^2 |\grad \xi_t (x)|^2 \ dx \leq \int_{\{0 < d(x) < 3t\} \cap B_r} |\grad u(x)|^2 \ dx.\]
We know from H\"older regularity of $\mathcal{J}$-minimizers, \lref{holder-FB},
\begin{equation}\label{e.u-d-bound}
    u(x) \leq C( 1+ \|\grad u\|_{\underline{L}^2(B_{2r})})r^{1-\gamma}\textup{dist}(x,\partial \{u>0\})^\gamma
\end{equation}
so
\begin{equation}\label{e.d(x)-contained}
    \{ 0 < d(x) < 3t\} \cap B_{r} \subset \{0 < |u(x)| < C( 1+ \|\grad u\|_{\underline{L}^2(B_{2r})})r(t/r)^\gamma\} \cup \{ r- 3t < |x| < r\}
\end{equation}
and so we can bound
\begin{align*}
    \int_{\{ 0 < d(x) < 2t\} \cap B_r} |\grad \bar{u}|^2 \ dx &\leq 9\int_{\{ 0 < d(x) < 3t\} \cap B_r} |\grad u|^2 \ dx\\
    &=  9\left[\int_{\{0 < |u(x)| < C( 1+ \|\grad u\|_{\underline{L}^2(B_{2r})})r(t/r)^\gamma\}\cap B_r}+\int_{B_r \setminus B_{r-3t}}\right] |\grad u|^2 \ dx.
\end{align*}
The first integral on the right can be bounded using the boundary strip energy estimate \lref{bdry-strip},
\[\int_{\{0 < |u(x)| < C( 1+ \|\grad u\|_{\underline{L}^2(B_{2r})})r(t/r)^\gamma\}\cap B_r} |\grad u|^2 \ dx \leq C( 1+ \|\grad u\|_{\underline{L}^2(B_{2r})}^2)(t/r)^\gamma|B_r|.\]
The second integral can be bounded using the interior Meyers' estimate \lref{meyers-FB} for some universal $p>2$
\begin{align*}
        \int_{B_r \setminus B_{r-3t}} |\grad u|^2 \ dx &\leq C\|\grad u\|_{\underline{L}^p(B_r)}^2(t/r)^{\frac{2-p}{p}}|B_r|\\
        &\leq C(1+\|\grad u\|_{\underline{L}^2(B_{2r})}^2)(t/r)^{\frac{2-p}{p}}|B_r|
\end{align*}
using H\"older's inequality for the first inequality, in a similar way as shown in the proof of \lref{bdry-strip-dirichlet}. The claimed result follows with $\gamma$ in the statement taken to be the minimum of $\gamma$ used in the proof  and $1-\frac{2}{p}$.
\end{proof}

\begin{lemma}[Upscaling error estimate]\label{l.upscaling-err}
Suppose that $u$ minimizes $\mathcal{J}$ over $u + H^1_0(B_{2r})$. Let $\gamma = \gamma(d,\Lambda)$ be as in \lref{grad-baru-bounds} and define the regularization $\bar{u}$ as in \eref{upscaling} with $t = r^{\frac{2\gamma}{2\gamma+1}}$. There exists $C = C(d,\Lambda) \geq 1$ and $\mathbf{r}_0= {\bf r}_0(d,\Lambda)$ such that
\[ E_0(\bar{u},B_r) \leq E(u,B_r) +C( 1+ \|\grad u\|_{\underline{L}^2(B_{2r})}^2)r^{-\frac{\gamma}{2\gamma+1}}|B_r| \quad \hbox{for all } r \geq \mathbf{r}_0.\]
\end{lemma}

\begin{proof} The proof follows along the lines of \cite{Feldman}*{Proposition 5.1}, which in turn is based on the statements and arguments in \cite{Shen}*{Ch. 3.2 and 3.3}. A notable difference, however, is that we do not yet have a Lipschitz estimate for minimizers and must rely on H\"older regularity and Meyers-type estimates instead.

We will make use of the dual energy
\begin{equation}\label{e.mumin}
 \mu(U,q) = \inf \left\{ \dashint_{U} \tfrac{1}{2}\grad v \cdot  a(x) \grad v - q \cdot \grad v \ dx : \ v \in H^1(U)\right\}.
 \end{equation}
This quantity has the following large scale limit
\begin{equation}\label{e.mu-convergence}
    \bigg|\mu(B_t,q)+\frac{1}{2}|q|^2 \bigg| \leq C(\Lambda,d)|q|^2 t^{-{1/2}}.
\end{equation}
Recall that we have normalized $\bar{a} = I$. The proof of this estimate can be found in \cite{Feldman}*{Proposition 5.1}.

Call
\[ p(x) = \grad \xi_t(x) = \dashint_{B_{t}(x)} \grad u(y) \ dy.\]
The dual slope $q(x):=p(x)$ is the same since we have normalized $\bar{a} = I$.

For $t$ to be determined as a power of $r$, we denote $U_t:= B_{r} \cap \{d(x) \geq 2t\}$, the ball minus a neighborhood of $\partial B_{r} \cup \{u=0\}$. The dual quantity $\mu$ is exactly useful for lower bounds of the energy
 \begin{align*} \int_{B_{r}}  \tfrac{1}{2}\grad u(x) \cdot a(x)\grad u(x) \ dx  &\geq \int_{U_t} \left(\dashint_{B_t(y)}  \tfrac{1}{2}\grad u(x) \cdot a(x)\grad u(x) \ dx\right) \ dy \\
 &\geq \int_{U_t} \mu(B_{t},q(y))  +q(y) \cdot p(y) \ dy \\
 &=\int_{U_t} \mu(B_t,q(y))  +  |p(y)|^2 \ dy\\
 & = \int_{U_t} \tfrac{1}{2}|p(y)|^2 \ dy + \int_{U_t} \left[\tfrac{1}{2}|q(y)|^2 + \mu(B_{t},q(y))\right] \ dy \\
 & \geq \int_{U_t} \tfrac{1}{2}|p(y)|^2 \ dy - \frac{C}{t^{1/2}}|B_r|\dashint_{U_t} |q(y)|^2  \ dy
 \end{align*}
 where we have used \eref{mu-convergence} pointwise in $y$ for the last inequality. Arguing similarly to the proof of \lref{grad-baru-bounds}, we can bound
  \[\dashint_{U_t} |q(y)|^2 dy \leq C\dashint_{B_{r}} |\grad u|^2 dx.\]
Finally we need to estimate the difference between the homogenized energy of $\bar{u}$ on $U_t$ versus the entire $B_{r}$. For this, we use \lref{grad-baru-bounds},
 \begin{align*} \int_{U_t}  \tfrac{1}{2}|p(x)|^2 dx -   \int_{B_{r}}  \tfrac{1}{2}|\grad \bar{u}|^2 dx&= -\int_{B_{r} \setminus U_t} \tfrac{1}{2}|\grad \bar{u}|^2 dx \\
 &\geq -C( 1+ \|\grad u\|_{\underline{L}^2(B_{2r})}^2)(t/r)^{\gamma}|B_r|.
 \end{align*}
 Together we have estimated
 \[ \int_{B_r} |\grad \bar{u}|^2 \ dx \leq \int_{B_r} a(x) \grad u \cdot \grad u \ dx +C( 1+ \|\grad u\|_{\underline{L}^2(B_{2r})})^2|B_r|\left[(t/r)^\gamma+t^{-1/2}\right]. \]
 Choosing $t = r^{\frac{2\gamma}{2\gamma +1}}$ makes the two error terms the same order. We need $t \leq r/2$ to apply \lref{grad-baru-bounds} so we choose ${\bf r}_0(d,\Lambda)$ based on the requirement $ {\bf r}_0^{\frac{2\gamma}{2\gamma +1}}= {\bf r}_0/2$.
\end{proof}

\begin{lemma}\label{l.Q-error-est}
Suppose that $u$ minimizes $\mathcal{J}$ over $u + H^1_0(B_{2r})$.  There is $C \geq 1$ depending on $d$ and $\max Q_\pm$ so that
\[\left|\int_{\{\pm u>0\} \cap B_{r}} Q_\pm(x)^2 \ dx-\int_{\{\pm u>0\} \cap B_{r}} q_\pm^2 \ dx\right| \leq C( 1+ \|\grad u\|_{\underline{L}^2(B_{2r})}^2)r^{-\gamma}|B_r|.\]
\end{lemma}

\begin{proof}
On any lattice cube $\Box$ with $\Box \subset (\{u>0\} \cap B_r)$
\[\int_{\{ u>0\}\cap B_r \cap \Box} Q_\pm(x)^2 \ dx=\int_{\{ u>0\}\cap B_r \cap \Box} q_\pm^2 \ dx.\]
Similarly for lattice cubes contained in $(\{u<0\} \cap B_r)$. So it suffices to bound the measure of the union of the $\Z^d$ lattice cubes intersecting $\partial (\{u\neq 0\} \cap B_r)$. For any such $\Box$, by similar arguments as in \eref{d(x)-contained},
\[ B_r \cap \Box \subset \{  d(x) < \sqrt{d}\} \subset \{ |u(x)| <C( 1+ \|\grad u\|_{\underline{L}^2(B_{2r})})r^{1-\gamma}\} \cup \{r-\sqrt{d} < |x| < r\}. \]
Applying \lref{bdry-strip} we obtain the desired upper bound
\end{proof}

\subsection{Homogenization error for the downscaling}\label{s.downscaling}
In this section we will consider $u_0$ which is a $\mathcal{J}_0$ minimizer in $B_r$ in the sense
\[\mathcal{J}_0(u_0) \leq \mathcal{J}_0(v) \ \hbox{ for all } \ v \in u_0 + H^1_0(B_r).\]
 Let $\zeta$ be the same cutoff as in the previous subsection and $d(x)$ be as in \eref{cutoff-dist}.  Let $t \ll r$ be a scale to be determined and
\begin{equation}\label{e.downscaling-definition}
 \tilde{u}_0(x) = u_0(x) + \chi(x) \cdot \grad u_0(x) \zeta(\tfrac{d(x)}{t}).
 \end{equation}
where $\chi$ is the vector of correctors defined in \sref{homrecall}.

\begin{lemma}[Downscaling error estimate]\label{l.downscaling-err}
Let $g \in H^1(B_{2r})$ satisfy, for some $p>2$ and $\gamma := 1-\frac{2}{p}$, the bounds
   \[ \|\grad g\|_{\underline{L}^p(B_r)}+r^{\gamma-1}\sup_{B_\rho \subset B_{\frac{3}{2}r}(0)} \rho^{1-\gamma}\| \grad g\|_{\underline{L}^2(B_{\rho})} =:G < + \infty .\] Suppose $u_0$ minimizes $\mathcal{J}_0$ over $g + H^1_0(B_r)$. Then
   \[ \left|E(\tilde{u}_0,B_r) - E_0(u_0,B_r) \right| \leq C(1+G^2)r^{-\gamma'}|B_r| \]
where $\gamma ' = \frac{\gamma}{2+\gamma}$.
\end{lemma}
\begin{proof}
    This is similar to \cite{Shen}*{Theorem 3.3.2} and \cite{Feldman}*{Lemma 5.3}. We split the proof into two steps for readability. The first step establishes all the necessary bounds on the downscaling and the second step estimates the difference of the energies.
    
    \emph{Step 1.}
    Let $t \gg 1$ to be chosen and define $U_t$ to be the union of the $\Z^d$-lattice squares contained in $\{x \in B_{r}: d(x) \geq 2t+\sqrt{d}\}$. 
    
 Using the definition of the downscaling \eref{downscaling-definition}, we calculate the gradient of $\tilde{u}_0$, with the primary term listed first and error terms listed after:
 \[ \grad \tilde{u}_0 = [\grad u_0 + (D \chi) \grad u_0]\zeta(\tfrac{d(x)}{t})+ (D^2u_0 \chi ) \zeta(\tfrac{d(x)}{t}) + (\chi \cdot \grad u_0)\zeta'(\tfrac{d(x)}{t})\tfrac{\grad d(x)}{t}+(1-\zeta(\tfrac{d(x)}{t}))\grad u_0. \]
  Note that the latter two error terms are supported in the complement of $U_t$. 
  
  We make note of the following bounds, which follow from the corrector bounds \eref{correctorest3}:
 \begin{equation}\label{e.downscaling-grad-tilde-bds}
     \begin{array}{c}
     |\grad u_0 + (D \chi) \grad u_0| \leq C|\grad u_0|, \\
     |(1-\zeta(\tfrac{d(x)}{t}))(D \chi )\grad u_0| \leq C(1-\zeta(\tfrac{d(x)}{t}))|\grad u_0|, \hbox{ and} \\
     |(\chi \cdot \grad u_0)\zeta'(\tfrac{d(x)}{t})\tfrac{\grad d(x)}{t}| \leq Ct^{-1}|\grad u_0|.
     \end{array} 
 \end{equation} 
Finally, we explain how to bound the Hessian of $u_0$ term. Recall that $u_0$ is harmonic in $\{u_0 \neq 0\}$ and so we have the $C^{1,1}$ estimate
 \[ |D^2u_0(x)|\leq\frac{C}{t}\|\grad u_0\|_{\underline{L}^2(B_{t}(x))} \ \hbox{ on } \{d(x) \geq t\} = \textup{supp}(\zeta(\tfrac{d(x)}{t})).\]
 By the above bound and Fubini
\begin{equation}\label{e.downscaling-hessian-bd}
     \int_{B_r} |D^2u_0(x)|^2 \zeta^2(\tfrac{d(x)}{t}) \ dx \leq \frac{C}{t^2}\int_{\{d(x) \geq t\}}\dashint_{B_t(x)}|\grad u_0(y)|^2 dy \ dx \leq \frac{C}{t^2}\|\grad u_0\|_{L^2(B_r)}^2.
 \end{equation}

 Let $p(x)$ be the piecewise constant function which is the average of $\grad u_0$ over the lattice box $\Box(x)$, which is the $\Z^d$ translation of $[0,1)^d$ that contains $x$. The definition of $U_t$ implies $\Box(x) \subset U_t$ for all $x \in U_t$, and so, by the previous $C^{1,1}$ estimate,
 \[ |p(x)- \grad u_0(x)| \leq C\sup_{\Box(x)}|D^2u_0| \leq \frac{C}{t}\|\grad u_0\|_{\underline{L}^2(B_{t+\sqrt{d}}(x))} \ \hbox{ for } \ x \in U_t.\]
  Thus, again arguing with Fubini similar to \eref{downscaling-hessian-bd},
   \begin{align}
     \int_{U_t} |p(x)- \grad u_0(x)|^2 \ dx &\leq \frac{C}{t^2}\int_{U_t} \dashint_{B_{t+\sqrt{d}}(x)}|\grad u_0(y)|^2 dy \ dx \label{e.downscaling-p-gradu0-bd}\\
     &\leq  \frac{C}{t^2}\int_{\{d(x) > t\}}|\grad u_0(y)|^2 dy \leq \frac{C}{t^2}\|\grad u_0\|_{L^2(B_r)}^2.\notag
 \end{align}

\emph{Step 2.} We will argue separately about the energy in $U_t$ and in the boundary / free boundary strip $B_r \setminus U_t$.
 
 First we handle the interior energy. By standard algebra with the quadratic energy, using \eref{downscaling-hessian-bd} and \eref{downscaling-p-gradu0-bd} to bound error terms, 
\[ \left|\int_{U_t} \grad \tilde{u}_0 \cdot a(x) \grad \tilde{u}_0 \ dx - \int_{U_t}(p(x) + D \chi(x) p(x)) \cdot a(x)(p(x) + D \chi(x) p(x)) \ dx\right| \leq Ct^{-2}\|\grad u_0\|_{\underline{L}^2(B_r)}^2|B_r|.\]
By the the properties of the corrector (see \eref{baradef}), since $U_t$ is exactly a union of lattice squares, and since $\bar{a} = \textup{id}$:
\[\int_{U_t}(p(x) + D\chi(x)p(x)) \cdot a(x)(p(x) + D\chi(x)p(x)) \ dx = \int_{U_t} |p(x)|^2 \ dx\]
and, applying \eref{downscaling-p-gradu0-bd} again,
\[ \left|\int_{U_t} |p(x)|^2 \ dx - \int_{U_t} \grad u_0(x) \cdot \grad u_0(x) \ dx \right| \leq Ct^{-2}\|\grad u_0\|_{\underline{L}^2(B_r)}^2|B_r|.\]
So applying the previous inequalities
\begin{equation}\label{e.downscaling-interior-final-bd}
    \left|\int_{U_t} \grad \tilde{u}_0 \cdot a(x) \grad \tilde{u}_0 \ dx - \int_{U_t} \grad u_0(x) \cdot \grad u_0(x) \ dx \right| \leq Ct^{-2}\|\grad u_0\|_{\underline{L}^2(B_r)}^2|B_r|.
\end{equation}
Next we bound the energy of $u_0$ and $\tilde{u}_0$ in $B_r \setminus U_t$. First, \lref{bdry-strip-dirichlet} implies
\[\int_{B_r \setminus U_t}|\grad u_0|^2\ dx  \leq C(1+G^2)(t/r)^\gamma|B_r|.\]
Then applying \eref{downscaling-grad-tilde-bds} and \eref{downscaling-hessian-bd}, we get
\begin{equation}\label{e.downscaling-bdry-final-bd}
     \int_{B_r \setminus U_t}|\grad \tilde{u}_0|^2\ dx\leq C\int_{B_r \setminus U_t}|\grad u_0|^2\ dx +\frac{C}{t}\|\grad u_0\|_{\underline{L}^2(B_r)}^2|B_r| \leq C(1+G^2)[(t/r)^\gamma+t^{-2}]|B_r|  .
\end{equation}

Combining the interior \eref{downscaling-interior-final-bd} and boundary \eref{downscaling-bdry-final-bd} energy bounds we find
\[\left|\int_{B_r} \grad \tilde{u}_0 \cdot a(x) \grad \tilde{u}_0 \ dx - \int_{B_r} \grad u_0(x) \cdot \grad u_0(x) \ dx\right| \leq C(1+G^2)[(t/r)^\gamma+t^{-2}]|B_r|.\]
Choosing $ t = r^{\frac{\gamma}{2+\gamma}}$ yields the result.

\end{proof}

\subsection{Proof of \tref{energy-hom-error}} Let $\bar{u}$ be the regularization of $u$ in $B_r$ defined in \sref{upscaling}. Let $u_0$ be the minimizer of $\mathcal{J}_0$ over $u + H^1_0(B_r)$, and let $\tilde{u}_0$ be the corresponding downscaling defined in \sref{downscaling}.  Note that all of $u$, $\bar{u}$, $u_0$, and $\tilde{u}_0$ have the same trace on $\partial B_r$. 

By the interior H\"older estimate \lref{holder-FB} and Meyers' estimate \lref{meyers-FB} for $\mathcal{J}$ minimizers 
\begin{equation}\label{e.u-as-bdrydata-bound}
\|\grad u\|_{\underline{L}^p(B_r)}+r^{\gamma-1}\sup_{B_\rho \subset B_{\frac{3}{2}r}(0)} \rho^{1-\gamma}\| \grad u\|_{\underline{L}^2(B_{\rho})} \leq C (1+\| \grad u\|_{\underline{L}^2(B_{2r})}).
\end{equation}
Combining \lref{upscaling-err} and \lref{Q-error-est}, we obtain for some $\omega(d,\Lambda)\in(0,1)$
\[\mathcal{J}_0(\bar{u}) \leq \mathcal{J}(u) + C( 1+ \|\grad u\|^2_{\underline{L}^2(B_{2r})})r^{-\omega}|B_r|.\]
Since $u$ minimizes $\mathcal{J}$ over $u + H^1_0(B_r)$ and $\tilde{u}_0 \in u + H^1_0(B_r)$ we also know
\[\mathcal{J}(u) \leq \mathcal{J}(\tilde{u}_0).\]
So invoking \lref{downscaling-err}, with $g = u$ as the boundary data for $u_0$ on $\partial B_r$ and using \eref{u-as-bdrydata-bound}, we find
\[\mathcal{J}(\tilde{u}_0) \leq \mathcal{J}_0(u_0)+C( 1+ \|\grad u\|^2_{\underline{L}^2(B_{2r})})r^{-\omega}|B_r|,\]
making the universal $\omega>0$ smaller if necessary. Finally 
\[\mathcal{J}_0(u_0) \leq \mathcal{J}_0(v) \ \hbox{ for all } \ v \in u_0+H^1_0(B_r)= \bar{u} + H^1_0(B_r)\]
and so combining the four previous displayed inequalities gives the result. \qed

\section{Large scale regularity of flat minimizers}\label{sec:large-scale-regularity-of-J-minimizers}

We begin with some convenient notations related to flatness with respect to two-plane solutions. Since we are performing a $C^{1,\beta}$ iteration in this section this is helpful to reduce the length of formulas. We define the flatness of $u$ with respect to the two-plane solutions $\Phi_{\a}$ as follows:
\begin{equation}\label{e.flat-def}
    \textup{flat}(u,U) := \inf \{ \delta >0 : \ \exists \alpha \geq 0, \nu \in S^{d-1} \hbox{ s.t. } \Phi_\alpha(x \cdot \nu - \delta) \leq u(x) \leq \Phi_\alpha(x\cdot \nu + \delta) \ \forall x \in U\}.
\end{equation}
Define $\nu(u,U)$ to be the infimizer and $\alpha(u,U)$ to be the corresponding slope (the maximal one if there are multiple). Note that if $\alpha(u,B_r)>0$ then
\[\textup{flat}(u,U) = \inf_{(\nu,\alpha) \in S^{d-1} \times \R_+} \sup_{x\in U}|\Psi_\alpha(u(x)) - x \cdot \nu|.\]

Our main result in this section is the following:
\begin{theorem}[Improvement of flatness down to microscale for $\J$ minimizers] \label{t.improvement-of-flatness-to-microscale}

  Suppose $u$ minimizes $\mathcal{J}$ in $B_R$ and $u(0) = 0$. For all $\beta \in (0,1)$ and $\underline{\alpha}>0$ there exists $\bar\delta>0$ so that if
    \[\frac{1}{R}\textup{flat}(u,B_R) \leq \delta  \leq \overline{\delta} \ \hbox{ with } \ \alpha(u,B_R) \geq \underline{\alpha}\]
    then
    \[\frac{1}{r}\textup{flat}(u,B_r)   \leq C(\delta (\tfrac{r}{R})^\beta+(\tfrac{\mathbf{r}_0}{r})^\omega) \quad \hbox{ for all } \ \mathbf{r}_0 \leq r \leq R/2.\]
    Moreover,
    \[\osc_{\mathbf{r}_0 \leq s \leq r} \nu(u,B_s) +  \osc_{\mathbf{r}_0 \leq s \leq r} \log \alpha(u,B_{s})  \leq C(\delta (\tfrac{r}{R})^\beta+(\tfrac{\mathbf{r}_0}{r})^\omega) \quad \hbox{ for all } \ \mathbf{r}_0 \leq r \leq R/2.\]
    In the above, $\omega \in (0,1)$ is universal, depending only on $d,\Lambda$, while $C \geq 1$ depends on $d$, $\Lambda$, $\beta \in (0,1)$, and $\underline{\alpha}>0$.
\end{theorem}

The idea behind the proof of \tref{improvement-of-flatness-to-microscale} is to use the quantitative homogenization result, \tref{energy-hom-error}, to show that a flat $\J$ minimizer is a flat quasi-minimizer of $\J_0$ in a ball of radius larger than the microscale. This combined with the one-step improvement of flatness for flat H\"older quasi-minimizers of $\J_0$, \tref{onestep-improve-flatness}, leads to the conclusion of \tref{improvement-of-flatness-to-microscale}.

Before embarking on the proof of \tref{improvement-of-flatness-to-microscale}, we show an immediate application to a Lipschitz estimate for flat $\mathcal{J}$-minimizers. This will be used as a tool in the proof of the Lipschitz estimate for general $\mathcal{J}$-minimizers in the next section. First, we need a simple lemma.

\begin{lemma}\label{l.flat-slope-controls-gradient}
    Suppose $u$ minimizes $\J$ in $B_r$. Then
    \[\|\grad u\|_{\underline{L}^2(B_{r/2})} \leq C\left(1 + \frac{\osc_{B_r} u}{r}\right)\]
\end{lemma}
\begin{proof}
    Let $v_r$ be the $a$-harmonic replacement of $u$ in $B_r$. By maximum principle, $\osc_{B_r} v_r \leq \osc_{B_r} u$.  Applying the Caccioppoli inequality to $v_r$, we get
    \[\|\grad v_r\|_{\underline{L}^2(B_{3r/4})} \leq \frac{C}{r} \| v_r - (v_r)_{B_r}\|_{\underline{L}^2(B_r)} \leq\frac{C}{r} \osc_{B_r} u.\]
    The conclusion now follows from \lref{a-harmonic-replacement}.
\end{proof}

We now use \tref{improvement-of-flatness-to-microscale} to prove a Lipschitz estimate for flat $\mathcal{J}$-minimizers.

\begin{corollary}[Lipschitz estimate for flat $\J$-minimizers]\label{c.flat-implies-Lipschitz-fcn}
    For any $\underline{\alpha}>0$ there is $\overline{\delta}>0$ so that if $u$ minimizes $\mathcal{J}$ over $u + H^1_0(B_r)$ with $u(0) = 0$ and
    \[\frac{1}{R}\textup{flat}(u,B_R) \leq \delta  \leq \overline{\delta} \ \hbox{ and } \ \alpha(u,B_R) \geq \underline{\alpha}\]
    then
    \[\|\grad u\|_{\underline{L}^2(B_r)} \leq  C(1+\alpha(u,B_R)) \ \hbox{ for all } \ \mathbf{r}_0 \leq r \leq R/2.\]
\end{corollary}
Although we have stated the result with an  arbitrary $\underline{\alpha}>0$, actually we will only need it for sufficiently large, i.e. $\geq O(1)$, slopes.
\begin{proof}
By \tref{improvement-of-flatness-to-microscale}, we have 
\[\frac{1}{r}\textup{flat}(u,B_r) \leq C(1+\delta) \quad \text{ for all } \mathbf{r}_0 \leq r \leq R/2\]
and with $\alpha(u,B_r) \leq C\alpha(u,B_R)$. This implies 
\[\frac{\osc_{B_{r}} u}{r} \leq C(1+\alpha(u,B_R)) \quad \text{ for all } \mathbf{r}_0 \leq r \leq R/2.\] \lref{flat-slope-controls-gradient} then implies $\|\grad u\|_{\underline{L}^2(B_r)} \leq C(1+\alpha(u,B_R))$ for all $\mathbf{r}_0 \leq r \leq R/2$.
\end{proof}

\subsection{One-step improvement of flatness} The next lemma establishes a one-step improvement of flatness for $\J$ minimizers, similar to \tref{onestep-improve-flatness}.

\begin{lemma}\label{l.onestep-improvement-of-flatness-Jmin}
 Suppose $u$ minimizes $\mathcal{J}$ in $B_r$ and $u(0) = 0$. For all $\beta \in (0,1)$ and $\underline{\alpha}>0$ there exists $\bar\delta, \theta, \bar{C} >0$ so that if
    \[\frac{1}{r}\textup{flat}(u,B_r) \leq \delta  \leq \overline{\delta} \ \hbox{ with } \ \alpha(u,B_r) \geq \underline{\alpha},\]
    then
    \[\frac{1}{\theta r}\textup{flat}(u,B_{\theta r}) \leq \theta^{\beta}(\delta+\bar{C}(\theta r)^{-\omega})  \]
    and
    \[|\nu(u,B_{\theta r}) - \nu(u,B_r)| +\left|\frac{\alpha(u,B_{\theta r})}{\alpha(u,B_{r})}-1\right| \leq \bar{C}\delta.\]
    In the above, $\omega \in (0,1)$ is universal, depending only on $d,\Lambda$.
\end{lemma}

\begin{proof}[Proof of \lref{onestep-improvement-of-flatness-Jmin}]

Let $\bar{\delta}, \theta$ be as in \tref{onestep-improve-flatness}. By \tref{energy-hom-error}, we know that since $u$ minimizes $\mathcal{J}$ over $u + H^1_0(B_{r})$,
\[\mathcal{J}_0(\bar{u},B_{r/2}) \leq \mathcal{J}_0(v,B_{r/2}) + C(1+\| \grad u\|^2_{\underline{L}^2(B_{r})})r^{-\omega_0}|B_r| \ \hbox{ for all } \ v \in \bar{u} + H^1_0(B_{r/2})\]
where $\bar{u}$ is the regularization defined in \eref{upscaling} and $\omega_0 \in (0,1)$ depends only on $d$ and $\Lambda$. The flatness assumption implies $\frac{\osc_{B_{r}} u}{r} \leq C\alpha$; applying \lref{flat-slope-controls-gradient} gives us $\| \grad u\|^2_{\underline{L}^2(B_{r})} \leq  C(1+ \alpha^2)$. 

Rescaling to the unit ball, we find that the function
\[\bar{u}_r(x):= (r/2)^{-1} \bar{u}((r/2)x), \quad x \in B_1\]
satisfies $\bar{u}_r(0) = 0$ and
\[\mathcal{J}_0(\bar{u}_r,B_1) \leq \mathcal{J}_0(w,B_1) + C(1+\alpha^2)r^{-\omega_0} \ \hbox{ for all } \ w \in \bar{u}_r + H^1_0(B_1).\]
Furthermore, if $\frac{1}{r}\textup{flat}(u,B_r) \leq \delta$, then $\textup{flat}(\bar{u}_r,B_1) \leq \delta$. Finally, if $u$ is a $\J$-minimizer, then it satisfies the H\"older estimate at all scales as in Property \ref{pty.holder-est} and, consequently, so does $\bar{u}_r$. Thus, we may apply \tref{onestep-improve-flatness} with $\sigma = C(1+ \alpha^2)r^{-\omega_0}$ to conclude that \[\textup{flat}(\bar{u}_r,B_{\theta}) \leq \theta^{1+\beta}\delta +  C(1+\alpha^2)^Nr^{-\omega_0 N} \qquad N = \frac{d}{d+3\gamma}.\]
Note that the infimizers $\a' = \a(\bar{u}_r, B_{\theta})$ and $\nu' = \nu(\bar{u}_r, B_{\theta})$ satisfy the estimates
\[|\alpha' - \alpha| \leq C \alpha \delta \qquad |\nu'-e_d| \leq C \delta.\]
Restated in terms of $\bar{u}$,
\[\frac{1}{\theta r}\textup{flat}(\bar{u},B_{\theta r}) \leq \theta^{\beta}(\delta + \bar{C}(\theta r)^{-\omega}) \qquad \omega := \omega_0 N.\]
The estimate for $u$ now follows from the estimate for $\bar{u} - u$ in \tref{energy-hom-error}, which only costs an extra H\"older error. Note that the H\"older exponent $\gamma \in (0,1)$ is universal, depending only on $(d,\Lambda)$, so $\omega = N\omega_0$ is universal as well.
\end{proof}

\subsection{Proof of \tref{improvement-of-flatness-to-microscale}}

Let $\bar\delta, \theta, \bar{C}, \omega >0$ be as in \lref{onestep-improvement-of-flatness-Jmin}. Let $\alpha_0, \nu_0$ denote the infimizers for $\textup{flat}(u,B_R)$. Note that $\a_0 \geq \underline{\a}$. Define
\[\delta_0 = \frac{\overline{\delta}}{2}, \quad r_0  = R, \quad r_{k+1} = \theta r_k, \quad \hbox{ and } \quad  \delta_{k+1} = \theta^{\beta}(\delta_k + \bar{C} r_k^{-\omega}).
\]
Let $\a_k$ and $\nu_k$ denote the infimizers for $\textup{flat}(u,B_{r_k})$.

To apply \lref{onestep-improvement-of-flatness-Jmin} iteratively, we need $\delta_{k+1} \leq \bar{\delta}$ and $\alpha_{k+1}\geq \frac{\alpha_0}{2}$. Let us first verify the bound on $\delta_{k+1}$:
\begin{align*}
\delta_{k+1} & = \theta^{\beta}(\delta_k + \bar{C}r_k^{-\omega}) \\
& = \theta^{\beta (k+1)} \delta_0 +  \bar{C} \sum\limits_{j=0}^{k} \theta^{(k+1-j) \beta} r_j^{-\omega} \\
& = \theta^{\beta (k+1)} \delta_0 +  \bar{C} \sum\limits_{j=0}^{k} \theta^{(k+1-j) \beta} r_0^{-\omega} \theta^{-j \omega} \\
& = \theta^{\beta (k+1)} \delta_0 +  \bar{C} \theta^{\beta (k+1)}  r_0^{-\omega} \sum\limits_{j=0}^{k}  \theta^{-j(\beta + \omega)}\\
& \leq \theta^{\beta (k+1)} \left(\delta_0 +  C r_0^{-\omega} \theta^{-k(\beta + \omega)}\right)\\
& = \theta^{\beta (k+1)} \delta_0 +  C r_0^{-\omega} \theta^{-k\omega} \\
& \leq (r_k/R)^{\beta} \delta_0 + C r_k^{-\omega}.
\end{align*}
Let $K$ be the largest integer such that $Cr_k^{-\omega} = C (r_0\theta^k)^{-\omega} \leq \frac{\overline{\delta}}{2}$. It then follows that $\delta_{k+1} \leq \overline{\delta}$ for all $k \leq K$ and thus for all $r_k \geq \theta^{-1} \left(\frac{2C}{\bar{\delta}}\right)^{1/\omega} = : \mathbf{r}_0$.

Next, we check the lower bound on $\a_{k+1}$. Suppose $\alpha_j \in [\a_0/2, 2\a_0]$ for each $j = 1,\cdots, k$. Then
\begin{align*}
    |\a_{k+1}  -  \a_0| & \leq \sum_{j = 0}^k |\a_{j+1} - \a_j| 
     \leq C \sum_{j = 0}^k \a_j \delta_j \\
    & \leq C \sum_{j = 0}^k\left(\theta^{\beta (j+1)} \delta_0 +  C r_0^{-\omega} \theta^{-j\omega} \right) 
     \leq C\left(\theta^{\beta} \delta_0 + r_k^{-\omega}\right) \leq \frac{\a_0}{2}
\end{align*}
where the final inequality holds with $\delta_0$ smaller and $\mathbf{r}_0$ larger if needed, but still depending on universal parameters and on $\underline{\a}$. The estimate for $|\nu_{k+1} - \nu_0|$ is similar.\qed

\section{Lipschitz regularity of minimizers}
In this section, we finally establish the central result of this paper, which is the large scale Lipschitz regularity of $\mathcal{J}$-minimizers. 

\begin{theorem}[Large scale Lipschitz regularity of $\J$ minimizers]\label{t.u-large-scale-lipschitz}
Suppose that $u$ minimizes $\mathcal{J}$ over $u + H^1_0(B_{R})$ with $u(0)=0$. Then
\[\|\grad u\|_{\underline{L}^2(B_r)} \leq C(1+\| \grad u\|_{\underline{L}^2(B_{R})}) \ \hbox{for all } \ \mathbf{r}_0 \leq r \leq R.\]
\end{theorem}

Note that \tref{u-large-scale-lipschitz} includes the additional hypothesis $u(0) =0$. Let us briefly explain the standard way to derive the general Lipschitz estimate \tref{u-lipschitz} from \tref{u-large-scale-lipschitz} by combining with the interior regularity \tref{hominteriorreg}.

\begin{proof}[Proof of \tref{u-lipschitz}]

If $u \neq 0$ in $B_{R/2}(0)$ then $u$ is $a$-harmonic in $B_{R/2}$ and \tref{hominteriorreg} gives the result. Otherwise, let $x_0 \in B_{R/2}$ be the closest point in $\{u=0\}$ to $0$. Then \tref{u-large-scale-lipschitz} implies that
\[\|\grad u\|_{\underline{L}^2(B_r(x_0))} \leq C(1+\| \grad u\|_{\underline{L}^2(B_{R/2}(x_0))}) \leq C(1+\| \grad u\|_{\underline{L}^2(B_{R}(0))}) \ \hbox{for all } \ \mathbf{r}_0 \leq r \leq R/2.\]
In particular, take $r_1 = \max\{\mathbf{r}_0,|x_0|\} \leq R/2$, then
\[\|\grad u\|_{\underline{L}^2(B_{r}(0))} \leq C \|\grad u\|_{\underline{L}^2(B_{2r}(x_0))} \leq  C(1+\| \grad u\|_{\underline{L}^2(B_{R})}) \ \hbox{ for } \ r_1 \leq r \leq R/2.\]
Since $ u \neq 0$ in $B_{r_1}(0)$ then $u$ is $a$-harmonic in $B_{r_1}(0)$ and so \tref{hominteriorreg} implies
\[\|\grad u\|_{\underline{L}^2(B_{r}(0))} \leq C\|\grad u\|_{\underline{L}^2(B_{r_1}(0))} \leq C(1+\| \grad u\|_{\underline{L}^2(B_{R})})\]
for all $1 \leq r \leq r_1$.
\end{proof}

\subsection{Dichotomy for large slopes: geometric decay of the slope or flat implies Lipschitz}

Lipschitz regularity of $\J$ minimizers in their positive/negative phase follows from the Avellaneda-Lin interior estimates, quoted above in \tref{hominteriorreg}, so we only need to establish regularity near the free boundary; this is where the Lipschitz estimate for flat $\J$-minimizers, \cref{flat-implies-Lipschitz-fcn}, will be used within a dichotomy argument inspired by ideas of De Silva and Savin that can be found in \cite{DeSilvaSavinLipschitz}*{Proposition 2.6} and \cite{DeSilvaSavinAM}*{Proposition 2.2}.

\begin{proposition}\label{p.harmonic-replacement-dichotomy}
    There exist universal constants $C_0,L_0 \geq 1$ and $\eta\in (0,1)$ so that if $R \geq \mathbf{r}_0$, $u$ minimizes $\mathcal{J}$ over $u + H^1_0(B_R)$, $u(0) = 0$, and 
    \[\| \grad u\|_{\underline{L}^2(B_{R})} \geq L_0,\]
    then one of the following  alternatives holds:
    \begin{enumerate}[label = (\roman*)]
     \item\label{part.lipschitz-est-dichotomy2} 
        \[\|\grad u\|_{\underline{L}^2(B_{\eta R})} \leq \frac{1}{2}\|\grad u\|_{\underline{L}^2(B_R)}, \text{ or }\]
        \item\label{part.lipschitz-est-dichotomy1} \[\|\grad u\|_{\underline{L}^2(B_r)} \leq C_0(1+\| \grad u\|_{\underline{L}^2(B_{R})}) \ \hbox{for all } \ \mathbf{r}_0 \leq r \leq \eta R\]
    \end{enumerate}
\end{proposition}
\begin{proof}
Let $v \in u + H^1_0(B_R)$ be the $a$-harmonic replacement of $u$ in $B_R$.  By the large scale $C^{1,\beta}$ estimates for $a$-harmonic functions with $R \geq r \geq \mathbf{r}_0$ (see \tref{hominteriorreg}), we have
\[\frac{1}{r}\|v - (v(0) + \xi \cdot x + \chi_\xi)\|_{L^\infty(B_r)} \leq C\left(\frac{r}{R}\right)^{\beta}\|\grad v\|_{\underline{L}^2(B_R)}\]
for some $\beta \in (0,1)$ with $|\xi| \leq C\|\grad v\|_{\underline{L}^2(B_R)}$.  In addition,
\[ \|\grad v - (\xi + \grad \chi_\xi)\|_{\underline{L}^2(B_r)} \leq C\left(\frac{r}{R}\right)^{\beta}\|\grad v\|_{\underline{L}^2(B_R)}.\]

We split into cases depending on whether
\[\hbox{\partref{lipschitz-est-dichotomy2} } \ |\xi| \leq c_1\| \grad u\|_{\underline{L}^2(B_{R})} \qquad \hbox{ or } \qquad \hbox{\partref{lipschitz-est-dichotomy1} } \ c_1\| \grad u\|_{\underline{L}^2(B_{R})} \leq |\xi| \leq C\| \grad u\|_{\underline{L}^2(B_{R})}\]
for some sufficiently small constant $c_1 > 0$ to be determined. 
    
For case \partref{lipschitz-est-dichotomy2}, we use \lref{a-harmonic-replacement}, \eref{a-replacement-energy-bound}, and the corrector bounds \eref{correctorest3} to obtain for any $\eta < 1$
\begin{align*}
\|\grad u\|_{\underline{L}^2(B_{\eta R})} 
&\leq \|\grad v\|_{\underline{L}^2(B_{\eta R})} + \|\grad (u-v)\|_{\underline{L}^2(B_{\eta R})}\\
&\leq \|\grad v - (\xi + \grad \chi_\xi)\|_{\underline{L}^2(B_{\eta R})} +|\xi|(1+\|\grad \chi_{\xi/|\xi|}\|_{\underline{L}^2(B_{\eta R})})+ \|\grad (u-v)\|_{\underline{L}^2(B_{\eta R})}\\
&\leq C(\eta^\beta+c_1) \|\grad u\|_{\underline{L}^2(B_R)} + C_2
\end{align*}
for some universal constant $C_2$. Choosing $c_1, \eta^{\beta} \leq \frac{1}{8C}$ and $L_0 \geq 4 C_2$, we find
\[ \|\grad u\|_{\underline{L}^2(B_{\eta R})}  \leq \frac{1}{2} \|\grad u\|_{\underline{L}^2(B_R)}.\]
This establishes alternative \partref{lipschitz-est-dichotomy2}.

For case \partref{lipschitz-est-dichotomy1}, we will use \cref{flat-implies-Lipschitz-fcn}. Using the large scale $C^{1,\beta}$ estimate for $v$ in $B_{2\eta R}$ (we may assume $\eta < 1/4$), we have
\[\frac{1}{2\eta R}\|v - (v(0) + \xi \cdot x + \chi_\xi)\|_{L^\infty(B_{2\eta R})} \leq C\eta^{\beta}\|\grad u\|_{\underline{L}^2(B_R)} \leq C\eta^{\beta} |\xi|. \]
Since $\xi$ is large this will also be a good flatness estimate with respect to a two-plane solution, as two-plane solutions with large slope are more and more similar to planes relative to their size $|\xi|$. Note that, since $\sqrt{1+|\xi|^2} - |\xi| \leq |\xi|(1+C|\xi|^{-1})$ for $|\xi| \geq 1$, we have
    \begin{align*}
        |\Phi_{|\xi|}(\hat{\xi} \cdot x) - \xi \cdot x| & \leq 
    \begin{cases}
        0 &  \xi \cdot x \leq 0\\
        C|\xi|^{-1}|x| &\xi \cdot x > 0 
    \end{cases}
    \end{align*}
    It follows that, for all $x \in B_{2\eta R}$,
    \[|\Phi_{|\xi|}(\hat{\xi} \cdot x) - \xi \cdot x| \leq C|\xi|^{-1}\eta R.\]
    Therefore, using (in the second inequality below) that $\Psi_\alpha$ is $\frac{1}{\alpha}$-Lipschitz,
        \begin{align*}
        &\frac{1}{2\eta R}\textup{flat}(u,B_{2\eta R}) \leq \frac{1}{2\eta R}\sup_{x\in B_{2\eta R}}|\Psi_{|\xi|}(u(x)) - \hat{\xi} \cdot x|  \\
        &\quad \quad \leq \frac{1}{2\eta R |\xi|}\sup_{x\in B_{2\eta R}}|u(x) - \Phi_{|\xi|}(\hat{\xi} \cdot x)| \\
        &\quad \quad \leq \frac{1}{2\eta R |\xi|}\sup_{x\in B_{2\eta R}}\bigg[|u(x) - \xi \cdot x|+C|\xi|^{-1}\eta R\bigg]\\
        &\quad \quad \leq \frac{1}{2\eta R|\xi|}\sup_{x \in B_{2\eta R}} \bigg[|u(x)-v(x)| +|v(0)|+ |\chi_\xi(x)| + C \eta^{\a} |\xi| + C|\xi|^{-1}\eta R\bigg] \\
        & \quad \quad \leq \frac{C}{2\eta R|\xi|}\bigg[||u - v||_{L^{\infty}(B_{2\eta R})} +|v(0)|+ (1 + \eta^{\a}) |\xi| + |\xi|^{-1}\eta R\bigg].
    \end{align*}
 By \cref{L2-Linfty-quasi-harmonic-replacement} 
    \[\|u-v\|_{L^\infty(B_{2\eta R})} \leq CR,\]
    and so, since $u(0) =0$,
    \[|v(0)| \leq CR.\]
    It follows that
    \[\frac{1}{2\eta R}\textup{flat}(u,B_{2\eta R}) \leq  \frac{C}{2\eta R|\xi|}\bigg[R  + (1 + \eta^{\a}) |\xi| + |\xi|^{-1}\eta R\bigg].\]
    Note that $|\xi| \approx L_0$ if \partref{lipschitz-est-dichotomy1} holds, so we may choose $L_0, \mathbf{r}_0 \geq 1$ sufficiently large so that $u$ satisfies the hypotheses of \cref{flat-implies-Lipschitz-fcn}, which we can invoke to obtain alternative \partref{lipschitz-est-dichotomy1}.
    \end{proof}

\subsection{Proof of \tref{u-large-scale-lipschitz}}

Finally, we present the proof of the large scale Lipschitz estimate.  The proof iterates using the dichotomy established in \pref{harmonic-replacement-dichotomy}.
    
    Let $\eta>0$ and $C_0,L_0 \geq 1$ be as in \pref{harmonic-replacement-dichotomy}, and define
    \[\ell(r) = \| \grad u\|_{\underline{L}^2(B_{r})} \ \hbox{ for } \ \mathbf{r}_0 \leq r \leq R.\]
    Let $k_1$ be the largest integer $k \geq 0$ so that $\eta^kR \geq \mathbf{r}_0$. We aim to show that 
        \begin{equation}\label{e.lipschitz-iteration-main-claim}
        \ell(\eta^kR) \leq C_0(1+\eta^{-d}L_0+\ell(R)) \ \hbox{ for all } \ 0 \leq k \leq k_1.
    \end{equation}
    From this, it is standard to derive the estimate for $\ell(r)$, up to an additional universal factor depending on $\delta$, for all $\mathbf{r}_0 \leq r \leq R$.
    
    First consider the set of scales satisfying the stronger criterion
    \begin{equation}\label{e.lipschitz-iteration}
        K = \{0 \leq k \leq k_1 : \ \ell(\eta^kR) \leq \eta^{-d}L_0+\ell(R) \}.
    \end{equation}
    Clearly, $0 \in K$.  If all $0 \leq k \leq k_1$ belong to $K$ then we are done, as this is a stronger conclusion than \eref{lipschitz-iteration-main-claim}.  Otherwise, we can take $0 \leq \bar{k} \leq k_1-1$ to be the smallest integer so that $\bar{k}+1 \not\in K$. We only need to establish \eref{lipschitz-iteration-main-claim} for $\bar{k}+1 \leq k \leq k_1$.

We claim $\ell(\eta^{\bar{k}}R) > L_0$. Indeed, if $\ell(\eta^{\bar{k}}R) \leq L_0$ then 
    \[\ell(\eta^{\bar{k}+1}R) \leq \eta^{-d}\ell(\eta^{\bar{k}}R) \leq \eta^{-d}L_0\]
which implies $\bar{k}+1 \in K$, a contradiction. 

Since $\ell(\eta^{\bar{k}}R) > L_0$, we see that \pref{harmonic-replacement-dichotomy} applies. We will show that only outcome \partref{lipschitz-est-dichotomy1} in the dichotomy can hold.

If \partref{lipschitz-est-dichotomy2} holds for $u$ in $B_{\eta^{\bar{k}}R}$ then since $\bar{k} \in K$, we have
\[\ell(\eta^{\bar{k}+1}R) \leq \frac{1}{2}\ell(\eta^{\bar{k}}R) \leq \frac{1}{2}(\eta^{-d}L_0+\ell(R)).\]
This implies $\bar{k}+1 \in K$, a contradiction. Thus, the only remaining case is that \partref{lipschitz-est-dichotomy1} holds for $u$ in $B_{\eta^{\bar{k}}R}$.  In this case,
\[\ell(r) \leq C_0(1+\ell(\eta^{\bar{k}}R)) \leq C_0(1+\eta^{-d}L_0+\ell(R)) \ \hbox{ for all } \ \mathbf{r}_0 \leq r \leq \eta^{\bar{k}+1}R.\]
Therefore \eref{lipschitz-iteration-main-claim} holds for all $\bar{k}+1 \leq k \leq k_1$.

\section{Liouville property of entire minimizers}\label{s.liouville}

In this section, we show an application of the Lipschitz estimate, \tref{u-large-scale-lipschitz}, to entire $\J$ minimizers. We refer the reader to the introduction for the motivation behind this result.

\begin{definition}
    We say that $u \in H^1_{loc}(\R^d)$ minimizes $\mathcal{J}$ with respect to compact perturbations if for every ball $B_r \subset \R^d$
    \[ \mathcal{J}(u;B_r) \leq \mathcal{J}(v;B_r)  \ \hbox{ for all } \ v \in H^1_0(B_r).\]
\end{definition}

The main result of this section is a Liouville-type theorem classifying the minimizers of $\mathcal{J}$ on $\R^d$ which have a non-trivial negative phase in the blow-down limit.

\begin{proposition}\label{p.Liouville}
    Suppose $u \in H^1_{loc}(\R^d)$ minimizes $\mathcal{J}$ with respect to compact perturbations and $u(0) = 0$. Suppose, in addition, that
    \[  \liminf_{r \to \infty} \| \grad u\|_{\underline{L}^2(B_r)} < + \infty, \ \hbox{ and } \  \inf_{x \in \R^d}\liminf_{r \to \infty} \frac{u(rx)}{r} <0. \]
    Then there exist $e \in S^{d-1}$, $\alpha > 0$ and $\omega(d,\Lambda) \in (0,1)$ such that 
    \[\frac{1}{r}\sup_{x \in B_r(0)}|\Psi_{\alpha}(u(x)) - x \cdot e| \leq Cr^{-\omega} \ \hbox{ for all } \ r \geq 1.\]
\end{proposition}
Let us provide a sketch of the proof. First, we blow down and show that there is a subsequential limit which is a global $\mathcal{J}_0$ minimizer with non-trivial negative phase. Then we apply the classification of global solutions of \eref{homogenized-two-phase-PDE} to conclude that this blow-down is a two-plane solution. Of course, this is only a subsequential limit; to conclude that $u$ is close to the same two-plane solution at all large scales, we apply the large scale $C^{1,\alpha}$ estimate, \tref{improvement-of-flatness-to-microscale}. This procedure of using large scale $C^{1,\alpha}$ estimates from homogenization to prove a Liouville theorem is well known, see for example \cites{AKM,ArmstrongShen,AvellanedaLinHO,MoserStruwe} and also, in the free boundary context, the previous work of the second author \cite{Feldman}*{Theorem 6.7}. The interesting additional wrinkle here, in contrast to the standard homogenization setting, is that the classification of the global solutions of the homogeneous problem is a much trickier issue -- we are essentially using this result as a black-box.

Let us note the following estimates, which will be used in the proof of \pref{Liouville}.
\begin{lemma}\label{l.alpha0-alpha1-ests}
Suppose that for some $\a_0 > 0$, $|e_0| = 1$, and $0<\eta < 1/100$,
\begin{equation}\label{e.flat-eta-alpha0}
    \frac{1}{r}\sup_{B_r} |\Psi_{\alpha_0}(u(x)) - x \cdot e_0| \leq \eta 
\end{equation}
then
\[\bigg|\frac{\sqrt{1+\alpha_0^2}}{\sqrt{1+\alpha(u,B_r)^2}} -1\bigg| \leq 24\eta \ \hbox{ and } \ | \nu(u,B_r)-e_0| \leq  6\sqrt{\eta} \]
where $\alpha(u,B_r)$ and $\nu(u,B_r)$ were defined in \eref{flat-def}.
\end{lemma}
The proof of \lref{alpha0-alpha1-ests} is postponed to the end of this section.

\begin{proof}[Proof of \pref{Liouville}]
Call $M:= \liminf_{r \to \infty} \| \grad u\|_{\underline{L}^2(B_r)}$. Then \tref{u-large-scale-lipschitz} implies that 
\[\| \grad u\|_{\underline{L}^2(B_r)} \leq C(1+M) \ \hbox{ for all } \ 1 \leq r < +\infty.\]
Consider the blow-down sequence
\[u_{1/R}(y) = \frac{1}{R} u(Ry).\]
Then $u_{1/R}$ minimizes $\mathcal{J}_{1/R}$ with respect to compact perturbations where
\[\mathcal{J}_{1/R}(v;U) := \int_{U} \grad v \cdot a(Ry)\grad v + Q_+(Rx)^2{\bf 1}_{v>0} + Q_-(Ry)^2{\bf 1}_{v \leq 0} \ dy.\]
Also, by the hyperbolic scaling invariance of the norm,
\[\|\grad u_{1/R}\|_{\underline{L}^2(B_r)} \leq C(1+M) \ \hbox{ for all } \ \frac{1}{R} \leq r < +\infty.\]
By \lref{holder-FB}, the $u_{1/R}$ are uniformly H\"older continuous. Since $u(0) = 0$ they are also uniformly bounded. By the hypothesis we can take a subsequence $R_j \to \infty$ so that $\inf_y\lim_{j \to \infty} u_{1/R_j}(y) <0$. Taking a further subsequence, not relabeled, the $u_{1/R_j}$ converge locally uniformly on $\R^d$ and weakly in $H^1_{loc}$ to some $u_0$. By \tref{energy-hom-error}, $u_0 \in H^1_{loc}(\R^d)$ minimizes $\mathcal{J}_0$ with respect to compact perturbations. Thus $u_0\in H^1_{loc}(\R^d)$ minimizes $\mathcal{J}_0$ with respect to compact perturbations, $\|\grad u_0\|_{L^\infty} \leq C(1+M)$, and $(u_0)_- \not\equiv 0$.

By the classification of global Lipschitz solutions of the homogeneous two-phase problem \eref{homogenized-two-phase-PDE} (see, for instance, \cite{Alt-Caffarelli-Friedman-1984}*{Section 6}, \cite{Petrosyan-Shahgholian-Uraltseva}*{Theorem 2.9}, \cite{DeSilva-Savin-Global}*{Theorem 1.1}, or \cite{DePhilippis-Spolaor-Velichkov-2021}*{Lemma 2.2}) either $(u_0)_- \equiv 0$ or $u_0(x) = \Phi_{\alpha_0}(x \cdot e_0)$ for some $\alpha_0 >0$ and $e_0 \in S^{d-1}$. Since the first case is ruled out by our hypothesis we must have $u_0 = \Phi_{\alpha_0}(x \cdot e_0)$.

Now we use the $C^{1,\alpha}$ improvement of flatness \tref{improvement-of-flatness-to-microscale} to show that $u$ is close to the same $\Phi_{\alpha_0}(x \cdot e_0)$ at all scales $r \geq 1$.  Since $u_{1/R_j} \to \Phi_{\alpha_0}(x \cdot e_0)$ uniformly on $B_1$
\begin{equation}\label{e.liouville-flat-delta}
\frac{1}{R_j}\textup{flat}(u,B_{R_j}) \leq \delta_j:=\frac{1}{R_j}\sup_{B_{R_j}}|\Psi_{\alpha_0}(u(x)) - x \cdot e_0| \to 0 \ \hbox{ as } \ R_j \to \infty.
\end{equation}
So \tref{improvement-of-flatness-to-microscale} gives
\[\frac{1}{r}\textup{flat}(u,B_r)   \leq C(\delta_j (\tfrac{r}{R_j})^\beta+r^{-\omega}) \quad \hbox{ for all } \ 1 \leq r \leq R_j/2\]
    and
    \[\osc_{1 \leq s \leq r} \nu(u,B_s) +  \osc_{1 \leq s \leq r} \log \alpha(u,B_{s})  \leq C(\delta_j (\tfrac{r}{R_j})^\beta+r^{-\omega}) \quad \hbox{ for all } \ 1 \leq r \leq R_j/2.\]
    Then sending $j \to \infty$, and hence $\delta_j \to 0$, in the above gives
    \[\frac{1}{r}\textup{flat}(u,B_r)   \leq C  r^{-\omega} \quad \hbox{ for all } \ 1 \leq r <+\infty\]
    and
    \[\osc_{1 \leq s \leq r} \nu(u,B_s) +  \osc_{1 \leq s \leq r} \log \alpha(u,B_{s})  \leq Cr^{-\omega} \quad \hbox{ for all } \ 1 \leq r <+\infty.\]
    Combining this with \lref{alpha0-alpha1-ests} and \eref{liouville-flat-delta} gives
    \[|\nu(u,B_r) - e_0| \leq Cr^{-\omega/2} \ \hbox{ and } \ |\alpha(u,B_r) - \alpha_0| \leq C(\alpha_0)r^{-\omega} \]
    and so
    \[\frac{1}{r}\sup_{B_{r}}|\Psi_{\alpha_0}(u(x)) - x \cdot e_0| \leq \frac{1}{r} \textup{flat}(u,B_r) + C(\alpha_0)r^{-\omega/2} \leq C(\alpha_0)r^{-\omega/2}.\]
\end{proof}

\begin{proof}[Proof of \lref{alpha0-alpha1-ests}]
    By rescaling, it suffices to consider the case $r =1$. Call $\alpha_1 :=\alpha(u,B_1)$ and $e_1 := \nu(u,B_1)$. By definition as infimizers in \eref{flat-def}
    \begin{equation}\label{e.flat-eta-alpha1}
    \sup_{B_1} |\Psi_{\alpha_1}(u(x)) - x \cdot e_1| \leq \eta.
    \end{equation}
    Let $|y|=1$, $y \perp e_0$, and call $y_\eta := (1-4\eta^2)^{1/2}y+2 \eta e_0$ which is also in $B_1$.  Then \eref{flat-eta-alpha0} gives
    \[|\Psi_{\alpha_0}(u(y_\eta)) - 2\eta| \leq \eta.\]
    In particular $u(y_\eta) >0$ and also $\Psi_{\alpha_0}(u(y_\eta)) \leq 3 \eta$. Thus,
    \[|\Psi_{\alpha_1}(u(y_\eta))| =
        \frac{\sqrt{1+\alpha_0^2}}{\sqrt{1+\alpha_1^2}}\Psi_{\alpha_0}(u(y_\eta)) \leq 3\frac{\sqrt{1+\alpha_0^2}}{\sqrt{1+\alpha_1^2}}\eta.
    \]
    Call $\rho := \frac{\sqrt{1+\alpha_0^2}}{\sqrt{1+\alpha_1^2}}$ to simplify expressions in the following. Evaluating \eref{flat-eta-alpha1} at the same $y_\eta$ gives
    \[|y \cdot e_1| \leq  |y_\eta \cdot e_1| + 4\eta\leq |\Psi_{\alpha_1}(u(y_\eta))|+\eta+4\eta\leq (5+3\rho)\eta.\]
    Taking $y = e_1 - (e_1 \cdot e_0)e_0$ in the above gives
    \[ 1-(e_1\cdot e_0)^2  \leq (5+3\rho)\eta\]
    On the other hand, taking $x = e_0$ in \eref{flat-eta-alpha0} and \eref{flat-eta-alpha1}, respectively,
    \begin{equation}\label{e.alpha0-alpha1-e1eval}
        |\Psi_{\alpha_0}(u(e_0)) - 1| \leq \eta \ \hbox{ and } \ |\Psi_{\alpha_1}(u(e_0)) - e_0 \cdot e_1| \leq \eta
    \end{equation}
    Since $\eta < 1/100$ the first estimate gives that $u(e_0)>0$; in fact, using the formula for $\Psi_{\alpha_0}(z)$ when $z>0$ gives 
    \begin{equation}
        \label{e.alpha0-alpha1-e1eval-c1}
    u(e_0) > \frac{99}{100}\sqrt{1+\alpha_0^2}.
    \end{equation}
    This implies $\Psi_{\alpha_1}(u(e_0)) \geq 0$ as well, so the second estimate in \eref{alpha0-alpha1-e1eval} now implies $e_0 \cdot e_1 \geq -\eta \geq - 1/100$. So, using that we now know $1+e_1\cdot e_0 \geq 99/100$,
    \begin{equation}\label{e.alphas-e1-e0-est}
        \frac{1}{2}|e_1-e_0|^2 = 1-e_1 \cdot e_0 \leq \frac{100}{99}(1-(e_1 \cdot e_0)^2) \leq \frac{100}{99}(5+3\rho)\eta \leq (6+4\rho)\eta.
    \end{equation}
    Returning again to \eref{alpha0-alpha1-e1eval}, and using again that $u(e_0)>0$, we find
    \[\bigg|\frac{1}{\sqrt{1+\alpha_1^2}}-\frac{1}{\sqrt{1+\alpha_0^2}}\bigg|u(e_0) = |\Psi_{\alpha_0}(u(e_0))-\Psi_{\alpha_1}(u(e_0))| \leq 2\eta + |1-e_0 \cdot e_1| \leq (8+4\rho)\eta.\]
    Applying \eref{alpha0-alpha1-e1eval-c1}, we obtain
\begin{equation}\label{e.alpha0-alpha1-rhoest}
    |\rho-1| = \bigg|\frac{\sqrt{1+\alpha_0^2}}{\sqrt{1+\alpha_1^2}}-1\bigg| \leq \frac{100}{99}(8+4\rho)\eta \leq (9+5\rho)\eta.
\end{equation}   
    In particular,
    \[ \rho \leq 1 + (9+5\rho)\eta\]
    and since $\eta \leq 1/100$ and $1-5 \eta \geq 1/2$, we can rearrange the above to arrive at $\rho \leq 3$. Plugging this into \eref{alphas-e1-e0-est} gives
    \[|e_1-e_0| \leq 6\sqrt{\eta}.\]
    Plugging it into \eref{alpha0-alpha1-rhoest} gives
    \[|\rho -1| \leq 24\eta.\]

\end{proof}

\part{Appendices}
\appendix

\section{Technical lemmas}
  
In this appendix we present several technical results that are used in \sref{asymptotic-exp} to show the linearization of the two-phase problem to a transmission problem.
\begin{lemma}\label{l.moving-gradient-dc-2}
         Suppose that $P$ is smooth. Then for all $0<\delta<1$
         \[x_d + \delta \bigg[P(x) + \gamma_+(x_d)_+ - \gamma_-(x_d)_-\bigg] = (1+\delta \gamma_+)(x_d+\delta P(x))_+ - (1+\delta \gamma_-)(x_d+\delta P(x))_-+ O(\delta^2)\]
         where the constant factors in the $O(\delta^2)$ term depends only on $\gamma_{\pm}$ and $||P||_{L^{\infty}}$.
     \end{lemma}
     \begin{proof}
         In $x_d\geq0$
         \begin{align*}
             x_d + \delta (P(x) + \gamma_+(x_d)_+ - \gamma_-(x_d)_-)&=(1+\delta \gamma_+)x_d+\delta P(x)\\
             &=(1+\delta \gamma_+)(x_d+\delta P(x)) + O(\delta^2)
         \end{align*}
         and similarly in $x_d \leq 0$
     \[x_d + \delta (P(x) + \gamma_+(x_d)_+ - \gamma_-(x_d)_-)=(1+\delta \gamma_-)(x_d+\delta P(x))+O(\delta^2).\]
     This completes the proof when $\textup{sgn}(x_d) = \textup{sgn}(x_d+\delta P(x))$.  Otherwise, on the set
     \[\{\textup{sgn}(x_d) \neq \textup{sgn}(x_d+\delta P(x))\}\]
     we have $|x_d| \leq |\delta P(x)|$ so
     \[x_d+\delta P(x) = O(\delta)\]
     and so
     \[(1+\delta \gamma_+)(x_d+\delta P(x)) = (1+\delta \gamma_-)(x_d+\delta P(x)) + O(\delta^2).\]
     \end{proof}
      \begin{lemma}\label{l.moving-gradient-dc-1}
         Suppose that $P \in C(B_1) \cap C^1(B_1^+ \cup B_1') \cap C^1(B_1^- \cup B_1')$, $P(0) = 0$, and $\grad'P$ continuous at $0$. Define 
         \[\tau := (\grad')^+P(0) = (\grad')^-P(0) \neq 0, \qquad \gamma_\pm:= \partial_{x_d}^\pm P(0),\]
         and for $\delta < 1$,
         \[e^\delta := \cos(\delta|\tau|) e_d + \sin(\delta|\tau|) \frac{\tau}{|\tau|}.  \]
        Then for $x$ near the origin 
         \[x_d + \delta P(x) = (1+\delta \gamma_+)(x \cdot e^\delta)_+ - (1+\delta \gamma_-)(x \cdot e^\delta)_-+o(|x|)\delta + O(\delta^2) \]
         where the $o(|x|)$ term is the $C^1$ modulus of $P$ near $0$ (in $\overline{B_{1}^\pm}$).
     \end{lemma}
     \begin{proof}
         We will write $\hat{\tau}:= \frac{\tau}{|\tau|}$. Taylor expand $P$ near the origin to get
\begin{align*}
    x_d + \delta P(x) = x_d + \delta [\gamma_+(x_d)_+ -\gamma_-(x_d)_- +   \tau \cdot x'] + o(|x|)\delta.
\end{align*}
In $\{x_d\geq 0 \}$
\begin{align*}
    x_d + \delta P(x) &=  (\delta\tau,1+\delta\gamma_+) \cdot x + o(|x|)\delta \\
    &=(1+\delta \gamma_+)(\delta \tau , 1) \cdot x + o(|x|)\delta +O(\delta^2)\\
    &=(1+\delta \gamma_+)(|\tau|\delta \hat{\tau} , 1) \cdot x+ o(|x|)\delta +O(\delta^2)\\
    &=(1+\delta \gamma_+)(\sin(|\tau|\delta) \hat{\tau} , \cos(|\tau|\delta)) \cdot x+ o(|x|)\delta +O(\delta^2)\\
    &=(1+\delta\gamma_+)e^\delta \cdot x + o(|x|)\delta +O(\delta^2).
\end{align*}
Similar arguments in $\{x_d\leq 0 \}$ give
\[x_d + \delta P(x) = (1+\delta\gamma_-)e^\delta \cdot x + o(|x|)\delta +O(\delta^2).\]
This gives the desired estimate when $\textup{sgn}(x_d) = \textup{sgn}(x \cdot e^\delta)$.  When $\textup{sgn}(x_d) \neq \textup{sgn}(x \cdot e^\delta)$ and $x \in B_1$ then
\[ |x \cdot e^\delta| \leq C\delta \]
and so
\[(1+\delta\gamma_+)e^\delta \cdot x  = (1+\delta\gamma_-)e^\delta \cdot x +O(\delta^2).\]
     \end{proof}
     \begin{lemma}\label{l.visc-soln-equiv}
       Call $\alpha_- = \alpha$, $\alpha_+^2 = \alpha^2 + 1$, $\gamma_\pm \in \R$, and define
       \begin{equation}
    L_\alpha(\gamma_\pm) := (1+\alpha^2) \gamma_+ - \a^2 \gamma_-.
\end{equation}
Then for any $\delta > 0$ sufficiently small,
       \[\Phi_{\alpha_\pm(1\pm \delta \gamma_{\pm})}(z) = \Phi_{\alpha'}(\lambda z)\]
       where
       \[\lambda^2 = 1+2\delta L_{\alpha}(\gamma_\pm)+\delta^2((\gamma_++\gamma_-)L_\alpha(\gamma_\pm)-\gamma_+\gamma_-)\]
       and
       \[\alpha' = \alpha\lambda^{-1}(1+\delta\gamma_-) .\]
   \end{lemma}

   \begin{proof}
   We want to choose $\alpha'$ and $\lambda$ so that
       \begin{align*}
       \Phi_{\alpha_\pm(1\pm \delta \gamma_{\pm})}(z) &= \sqrt{1+\alpha^2} (1+\delta\gamma_+)z_+-\alpha(1+\delta \gamma_-) z_-\\
       &=\sqrt{\alpha'^2+1} (\lambda z)_+-\alpha' (\lambda z)_-\\
       &= \Phi_{\alpha'}(\lambda z).
   \end{align*}
   This results in the two equations, in the two unknowns $\alpha'$ and $\lambda$,
   \begin{align*}
    \lambda^2(\alpha'^2+1) &= (1+\delta\gamma_+)^2(1+\alpha^2)\\
       \lambda^2\alpha'^2 &= \alpha^2(1+\delta \gamma_-)^2.
   \end{align*}
   Plugging the second equation into the first,
   \[\alpha^2(1+\delta \gamma_-)^2+\lambda^2 = (1+\delta\gamma_+)^2(1+\alpha^2)\]
   and solving for $\lambda^2$ results in
   \[\lambda^2 = (1+\delta\gamma_+)^2(1+\alpha^2) - \alpha^2(1+\delta \gamma_-)^2.\]
   Expanding the squares
   \[\lambda^2 = (1+2\delta\gamma_++\delta^2\gamma_+^2)(1+\alpha^2) - \alpha^2(1+2\delta \gamma_-+\delta^2\gamma_-^2)\]
   and grouping terms by order of $\delta$
   \[\lambda^2 = 1+2\delta[(1+\alpha^2)\gamma_+ -\alpha^2\gamma_-]+\delta^2[(1+\alpha^2)\gamma_+^2-\alpha^2\gamma_-^2 ].\]
   Note that the linearized operator $L_\alpha(\gamma_\pm)$ arises in the linear order term in $\delta$
   \[\lambda^2 = 1+2\delta L_\alpha(\gamma_\pm)+\delta^2[(1+\alpha^2)\gamma_+^2-\alpha^2\gamma_-^2].\]
    We can simplify the coefficient of $\delta^2$ in the following way
    \begin{align*}
        (1+\alpha^2)\gamma_+^2-\alpha^2\gamma_-^2 &= (L_\alpha(\gamma_\pm)+\alpha^2\gamma_-)\gamma_+-\alpha^2\gamma_-^2 \\
        &=L_\alpha(\gamma_\pm)\gamma_+ + \gamma_-\alpha^2(\gamma_+-\gamma_-)\\
        &=L_\alpha(\gamma_\pm)\gamma_+ + \gamma_-(L_\alpha(\gamma_\pm)-\gamma_+)\\
        &=(\gamma_++\gamma_-)L_\alpha(\gamma_\pm)-\gamma_+\gamma_-.
    \end{align*}
    In particular note that this coefficient is independent of $\alpha$ if $L_\alpha(\gamma_\pm) = 0$.
   
   The following rearranged formula is useful in the case $\alpha \to +\infty$
   \[\lambda^2 = 1 + \alpha^2[\delta(\gamma_+ - \gamma_-)+\delta^2(\gamma_+^2 - \gamma_-^2)] + 2 \delta \gamma_+ + \delta^2 \gamma_+^2.\]
   \end{proof}

\section{H\"older and Sobolev regularity of Dirichlet quasi-minimizers}\label{s.regularity-appendix}
In this section we will consider quasi-minimizers of Dirichlet functionals of the type
\[ u \mapsto \int_{U} a(x) \grad u \cdot \grad u - f(x) \cdot \grad u\ dx .\]
The domain $U$ will typically be sufficiently regular. We will always assume, at least, that $a$ is measurable and uniformly elliptic
\[\Lambda^{-1} \leq a(x) \leq \Lambda \ \hbox{ in } \ U.\]
We will consider quasi-minimality for the Dirichlet functional in the following sense. Let $V \subset U$ be an open set, $\sigma \geq 0$, $\kappa \in L^2(V)$, and $f \in L^2(V)$. To be clear, we will be considering $V = U \cap B_r(x_0)$ for some ball $B_r(x_0)$ either contained in $U$ or centered on the boundary of $U$. Consider $u \in H^1(U)$ satisfying \begin{equation}\label{e.dirichlet-quasi-min}
        \int_V \frac{1}{2} a \grad u \cdot \grad u - f \cdot \grad u \ dx\leq (1+\sigma^2)\int_V \frac{1}{2} a \grad v \cdot \grad v \ dx- \int_{V} f \cdot \grad v \ dx + \int_{V} \kappa^2 \ dx \ \hbox{ for all} \ v \in u+ H^1_0(V).
    \end{equation}
We allow for both multiplicative ($\sigma$-term) and additive ($\kappa$-term) type approximate minimality properties, both of which arise naturally in applications.  Multiplicative approximate minimality implies additive almost minimality once a Lipschitz estimate is known.  However, our concern here is on initial / rough / sub-Lipschitz regularity. In particular, we will focus on H\"older, Calder\'on-Zygmund, and Meyers-type (i.e. $W^{1,2+\delta}$) estimates.   

\begin{remark}
    It suffices to consider the case of homogeneous Dirichlet boundary conditions. If we wanted to consider a similar quasi-minimality property over $g + H^1_0(U)$ we could instead study $\bar{u} = u - g$ which is in $H^1_0(U)$. If $u$ satisfies \eref{dirichlet-quasi-min} for some $V\subset U$ then $\bar{u}$ will satisfy:
    \[ \int_V \frac{1}{2} a \grad (\bar{u}+g) \cdot \grad (\bar{u}+g) - f \cdot \grad \bar{u} \ dx\leq (1+\sigma^2)\int_V \frac{1}{2} a \grad (v+g) \cdot \grad (v+g) \ dx- \int_{V} f \cdot \grad v \ dx + \int_{V} \kappa^2 \ dx\]
    for all $v \in \bar{u} +  H^1_0(V)$. Rearranging this, and applying Young's inequality to the leftover term $\sigma^2 a \grad v \cdot \grad g$ that arises on the right hand side, we arrive at
\begin{align*}
    &\int_V\frac{1}{2} a \grad \bar{u} \cdot \grad \bar{u} - (f-a\grad g) \cdot \grad \bar{u} \ dx \\
    & \quad \quad \leq (1+2\sigma^2)\int_V \frac{1}{2} a \grad v \cdot \grad v \ dx- \int_{V} (f-a \grad g) \cdot \grad v \ dx + \int_{V} \kappa^2+\sigma^2a \grad g \cdot \grad g \ dx
\end{align*}    
for all $v \in \bar{u} +  H^1_0(V)$. This is the same type of minimality property as \eref{dirichlet-quasi-min}, but now $\bar{u} \in H^1_0(U)$.
\end{remark}
 
There is a vast literature on the regularity of quasi-minimizers of Dirichlet-type energies. We refer to the classical text \cite{Giustibook} for a comprehensive treatment and relevant historical references regarding H\"older estimates and the paper by Caffarelli and Peral \cite{CaffarelliPeral} for Calder\'on-Zygmund-type estimates. In order to keep this work as self-contained as possible, we have chosen to present some of these regularity results in the precise form required for applications in the main body of the manuscript. The proof ideas are classical and well known to experts, but we provide sufficient details for the reader's convenience. The key point is that almost minimizers are well approximated by their $a$-harmonic replacement. This $a$-harmonic replacement property can then be used within an iteration argument to show that quasi-minimizers (almost) inherit the interior regularity of $a$-harmonic functions.

\subsection{Harmonic Replacement}
The main tool will be the following harmonic approximation property of quasi-minimizers.  We write a fairly general form of almost minimality, but further generalizations are possible within this framework as long as a similar harmonic approximation property can be shown.

\begin{definition}\label{d.harmonic-replacement-bdry}
    Let $V$ be an open set in $\R^d$, and $u$ be an $H^1_{\text{loc}}$ function on $\R^d$. Define the $a$-harmonic replacement of $u$ to be the $H^1$ weak solution of
\begin{equation}\label{e.harmonic-replacement-bdry}
    \begin{cases}
    - \grad \cdot (a(x) \grad u_V) = 0 &\hbox{in } \ V,\\
    u_V = u & \hbox{on } \partial V.
    \end{cases}
\end{equation}
\end{definition}
In other words $u_V$ is the minimizer of $v \mapsto \int_V a(x) \grad v \cdot \grad v$ over $u + H^1_0(V)$. Canonically extend $u_V:= u$ in $\R^d \setminus V$. It is useful to note that
\begin{equation}\label{e.a-replacement-energy-bound}
\| \grad u_{V}\|_{L^2(V)} \leq \Lambda \| \grad u\|_{L^2(V)}.
\end{equation}
\begin{lemma}[$a$-Harmonic Approximation]\label{l.a-harmonic-replacement}
    Let $V \subset \R^d$ be an open set, $\sigma \geq 0$, $\kappa \in L^2(V)$, and $f \in L^2(V)$. Suppose $u \in H^1(V)$ satisfies \eref{dirichlet-quasi-min}. Then the $a$-harmonic replacement $u_V \in u + H^1_0(V)$ satisfies
        \begin{equation}\label{e.a-harmonic-replacement-estimate}
        \| \grad u - \grad u_V\|_{L^2(V)} \leq C (\|f\|_{L^2(V)}+\|\kappa\|_{L^2(V)}+\sigma \|\grad u\|_{L^2(V)}).
        \end{equation}
\end{lemma}

\begin{remark}
    It is clear from \eref{a-harmonic-replacement-estimate} that $\kappa$ plays essentially the same role as $|f|$ as far as rough regularity estimates are concerned.  However, with the $\kappa$ term present, there is no Euler-Lagrange equation or Cacciopoli inequality for the function $u$, which is a source of additional challenges when dealing with regularity of quasi-minimizers.
\end{remark}
\begin{remark}
As we will emphasize in the remainder of this appendix, it is the \emph{conclusion} (i.e. the $a$-harmonic replacement estimate \eref{a-harmonic-replacement-estimate}) rather than the hypothesis of \lref{a-harmonic-replacement} that is important for the regularity estimates proved later.  Thus, some flexibility in the nature of the property \eref{dirichlet-quasi-min} is possible.    
\end{remark}

\begin{proof}
Since $u_V$ is $a$-harmonic and $u - u_V \in H^1_0(V)$, we have
\[\int_{V} a(x) \grad u_V \cdot \grad u \ dx = \int_{V} a(x) \grad u_V \cdot \grad (u_V+u-u_V) \ dx = \int_{V} a(x) \grad u_V \cdot \grad u_V \ dx\]
and so
\begin{align*}
& \Lambda^{-1}\int_{V} |\grad(u-u_V)|^2 \ dx \\
& \leq \int_{V} a(x) \grad(u-u_V) \cdot \grad (u-u_V) \ dx \\
& = \int_{V} a(x) \grad u \cdot \grad u \ dx - 2\int_{V} a(x) \grad u \cdot \grad u_V \ dx + \int_{V} a(x) \grad u_V \cdot \grad u_V \ dx\\
& = \int_{V} a(x)\nabla u\cdot \nabla u  \ dx  - \int_{V} a(x)\nabla u_V\cdot \nabla u_V \ dx  \\
& \leq  \int_{V}   \sigma^2 a(x)\nabla u_V\cdot \nabla u_V + f \cdot \grad (u-u_V) + \kappa^2 \ dx 
\end{align*}
where in the final line we have used \eref{dirichlet-quasi-min}.  We then bound the first term using the fact that $\int_{V}  a(x)\nabla u_V\cdot \nabla u_V \leq \int_{V}  a(x)\nabla u\cdot \nabla u$ by the $a$-harmonic replacement property, and we can use Young's inequality in the second term to obtain the desired estimate \eref{a-harmonic-replacement-estimate}.

\end{proof}

\subsection{Interior and global H\"older Regularity}
In this section we show how to combine an $a$-harmonic approximation property like \eref{a-harmonic-replacement-estimate} with a $C^{0,\gamma_0}$-estimate for $a$-harmonic functions to get $C^{0,\gamma_0-}$ H\"older estimates for quasi-minimizers. As we remarked above, it is the  estimate \eref{a-harmonic-replacement-estimate} that is crucial for us going forward, so we state both the $C^{0,\gamma_0}$-estimate of $a$-harmonic functions and the $a$-harmonic approximation property \eref{a-harmonic-replacement-estimate} as hypotheses below, keeping in mind that the latter holds for functions that satisfy \eref{dirichlet-quasi-min}. We also introduce a minimal scale $r_0 \geq 0$ in the hypotheses, as this is useful for applications to homogenization.

\begin{itemize}
    \item ($C^{0,\gamma_0}$-estimate for $a$-harmonic functions) There exists $\gamma_0 >0$ universal such that if $v$ is $a$-harmonic in $B_r$ then
    \begin{equation}\label{e.a-harmonic-Cgamma}
        \| \grad v\|_{\underline{L}^2(B_{\rho})} \leq C_1\left(\frac{\rho}{r}\right)^{\gamma_0-1}\| \grad v\|_{\underline{L}^2(B_{r})} \qquad \hbox{ for all } r_0\leq \rho \leq r \leq r_1.
    \end{equation}
    \item ($a$-harmonic approximation property) There exist ${h} \in L^2(B_{r_1})$ and $\sigma>0$ such that for any $r_0 \leq \rho \leq r_1$
    \begin{equation}\label{e.a-harmonic-approx-Cgamma}
        \|\grad (u - u_{B_\rho})\|_{\underline{L}^2(B_\rho)} \leq C_2(\|{h}\|_{\underline{L}^2(B_\rho)}+\sigma \|\grad u\|_{\underline{L}^2(B_\rho)} ).
    \end{equation}
\end{itemize}

\begin{lemma}[interior $C^{0,\gamma}$-estimate]\label{l.interior-holder-est-quasi}
Suppose $a$ and $u$ satisfy the properties \eref{a-harmonic-Cgamma} and \eref{a-harmonic-approx-Cgamma} above. For all $0<\gamma < \gamma_0$ there is $\sigma_0(\gamma)>0$ sufficiently small and $C \geq 1$ so that if $\sigma \leq \sigma_0$ then
\[ \| \grad u\|_{\underline{L}^2(B_{\rho})} \leq C\left(\frac{\rho}{r}\right)^{\gamma-1}\left(r^{\gamma-1}\sup_{\rho \leq s \leq r} \left[s^{1-\gamma}\|{h}\|_{\underline{L}^2(B_s)}\right]+ \| \grad u\|_{\underline{L}^2(B_{r})} \right) \ \hbox{ for all } \ r_0 \leq \rho \leq r \leq r_1 .\]
The constant $C$ depends on $C_1$, $C_2$, $\gamma_0 - \gamma$, $\Lambda$, and $d$. 
\end{lemma}

\begin{remark}\label{Holder-estimate-example-scenarios}
    We list some scenarios where the estimate \eref{a-harmonic-Cgamma} holds and what \lref{interior-holder-est-quasi} implies for the regularity of quasi-minimizers.
    \begin{itemize}
        \item If $a \equiv \textup{id}$, or constant coefficient uniformly elliptic, then $a$-harmonic functions enjoy interior Lipschitz estimates at all scales. In other words $\gamma_0=1$, $(r_0,r_1) = (0,\infty)$. In this case, \lref{interior-holder-est-quasi} implies interior H\"older $C^{0,\gamma}$ estimates of quasi-minimizers for every $\gamma \in (0,1)$ and $(r_0,r_1) = (0,\infty)$.
        \item If $a(x)$ is measurable and uniformly elliptic, then classical results of De Giorgi-Nash-Moser yield an interior H\"older estimate for $a$-harmonic functions with some small H\"older exponent $\gamma_0(d,\Lambda)$ and for all scales $(r_0,r_1) = (0,\infty)$. In this case, \lref{interior-holder-est-quasi} implies interior H\"older $C^{0,\gamma}$ estimates of quasi-minimizers for every $\gamma \in (0,\gamma_0)$ and $(r_0,r_1) = (0,\infty)$.
        \item If $a(x)$ is measurable, uniformly elliptic, and $\Z^d$-periodic then the Lipschitz estimate due to Avellaneda and Lin \cite{Avellaneda-Lin} holds at all scales above $1$ (see \tref{hominteriorreg}). In other words, \eref{a-harmonic-Cgamma} holds with $r_0 =1$, $r_1 = +\infty$, and $\gamma_0 = 1$.  In this case, \lref{interior-holder-est-quasi} implies a large-scale  $C^{0,\gamma}$ estimate for quasi-minimizers for every $\gamma \in (0,1)$.
    \end{itemize}
\end{remark}

\begin{remark}\label{r.g-norm-explanation}
    Let us mention two ways that the norm of $g$ appearing on the right hand side in \lref{interior-holder-est-quasi} can be bounded.  
    \begin{itemize}
        \item If $d\geq 2$ and ${h} \in L^p, p \geq \frac{d}{1-\gamma}$, we can use the inequality
    \[ r^{\gamma-1}\sup_{\rho \leq s \leq r} \left[s^{1-\gamma}\|{h}\|_{\underline{L}^2(B_s)}\right] \leq \|{h}\|_{\underline{L}^{\frac{d}{1-\gamma}}(B_{r})}\]
    which follows from Jensen's inequality, the bound $\|{h}\|_{L^{\frac{d}{1-\gamma}}(B_{s})} \leq \|{h}\|_{L^{\frac{d}{1-\gamma}}(B_{r})}$, and scaling.
    \item If $h$ satisfies the $C^{-1,\gamma}$-type estimate
    \[\| h\|_{\underline{L}^2(B_{\rho})} \leq H\left(\frac{\rho}{r}\right)^{\gamma-1}\| h\|_{\underline{L}^2(B_{r})} \qquad \hbox{ for all } r_0\leq \rho \leq r \leq r_1\]
    then immediately
    \[r^{\gamma-1}\sup_{\rho \leq s \leq r} \left[s^{1-\gamma}\|{h}\|_{\underline{L}^2(B_s)}\right] \leq H.\]
    In particular this is useful when $h = a \grad g$ and $g$ satisfies the $C^{0,\gamma}$-type estimate as in \eref{a-harmonic-Cgamma}. 
    \end{itemize}
\end{remark}

\begin{proof}[Proof of \lref{interior-holder-est-quasi}]
Let $0 < s < r $ and $\mu \in (0,1)$. The H\"older estimate of $a$-harmonic functions \eref{a-harmonic-Cgamma} implies
\[\| \grad u_{B_{s}}\|_{\underline{L}^2(B_{\mu s})} \leq C_1 \mu^{\gamma_0-1} \| \grad u_{B_s}\|_{\underline{L}^2(B_{s})}.\]
By the $a$-harmonic replacement property of $u$, \eref{a-harmonic-approx-Cgamma},
\[\| \grad u - \grad u_{B_{s}} \|_{\underline{L}^2(B_{\mu s})} \leq \mu^{-d/2} \| \grad u - \grad u_{B_{s}} \|_{\underline{L}^2(B_{s})} \leq C_2 \mu^{-d/2}(\|{h}\|_{\underline{L}^2(B_{s})}+\sigma \|\grad u\|_{\underline{L}^2(B_{s})}).\]
Therefore,
\begin{align*}
\| \grad u\|_{\underline{L}^2(B_{\mu s})} & \leq  \| \grad u_{B_{s}} \|_{\underline{L}^2(B_{\mu s})} + \| \grad u - \grad u_{B_{s}} \|_{\underline{L}^2(B_{\mu s})}\\
& \leq C_1 \mu^{\gamma_0-1} \| \grad u_{B_{s}}\|_{\underline{L}^2(B_{s})} + C_2 \mu^{-d/2}(\|{h}\|_{\underline{L}^2(B_{s})}+\sigma \|\grad u\|_{\underline{L}^2(B_{s})}) \\
& \leq (\Lambda C_1 \mu^{\gamma_0-1}+\sigma C_2\mu^{-d/2}) \| \grad u\|_{\underline{L}^2(B_{s})}  + C_2 \mu^{-d/2}\|{h}\|_{\underline{L}^2(B_{s})} 
\end{align*}
where we have used \eref{a-replacement-energy-bound} in the last inequality.

For any $\gamma \in (0,\gamma_0)$, we can write $\mu^{\gamma_0 - 1} = \mu^{\gamma_0 - \gamma} \mu^{\gamma - 1}$ and choose $\mu > 0$ sufficiently small so that $\mu^{\gamma_0 - \gamma} \leq \frac{1}{2C_1\Lambda}$. Then, with this fixed choice of $\mu = \mu(\gamma)$, we let $\sigma_0= \frac{1}{2C_2}\mu^{d/2}\mu^{\gamma - 1}$ so that
\[\sigma C_2\mu^{-d/2} \leq \frac{1}{2}\mu^{\gamma-1} \ \hbox{ when } \ \sigma \leq \sigma_0(\gamma).\]
Plugging in these choices into the previous bound we find
\[ \|\grad u\|_{\underline{L}^2(B_{\mu s})} \leq \mu^{\gamma-1} \| \grad u\|_{\underline{L}^2(B_{s})} + C_2 \mu^{-d/2}\|{h}\|_{\underline{L}^2(B_{s})} .\]
Note that, by Jensen's inequality,
\[\|{h}\|_{\underline{L}^2(B_{s})} \leq \|{h}\|_{\underline{L}^{\frac{d}{1-\gamma}}(B_{s})} \leq (s/r)^{\gamma-1}\|{h}\|_{\underline{L}^{\frac{d}{1-\gamma}}(B_{r})}.\]
We thus arrive at the one-step inequality of our planned iteration:
\begin{equation}\label{e.holder-one-step}
    \|\grad u\|_{\underline{L}^2(B_{\mu s})} \leq \mu^{\gamma-1} \| \grad u\|_{\underline{L}^2(B_{s})} + C(s/r)^{\gamma-1}\|{h}\|_{\underline{L}^{\frac{d}{1-\gamma}}(B_{r})} \ \hbox{ for all } 0 < s < r \
\end{equation}
where $C = C_2 \mu^{-d/2}$ is, at this point, a constant depending only on the parameters in the statement.

Iterating the inequality \eref{holder-one-step}, we see that for any positive integer $k$,
\begin{align*}
\|\grad u\|_{\underline{L}^2(B_{\mu^k r})}  & \leq \mu^{\gamma-1} \|\grad u\|_{\underline{L}^2(B_{\mu^{k-1} r})}  +C\mu^{k(\gamma-1)}\|{h}\|_{\underline{L}^{\frac{d}{1-\gamma}}(B_{r})} \\
& \leq \mu^{k(\gamma-1)} \|\grad u\|_{\underline{L}^2(B_{r})}   +\left(\sum_{j=0}^{k-1}C\mu^{(\gamma-1)(k-j)}\right)\|{h}\|_{\underline{L}^{\frac{d}{1-\gamma}}(B_{r})} \\
& \leq \mu^{k(\gamma-1)}\left(\|\grad u\|_{\underline{L}^2(B_{r})} +C\|{h}\|_{\underline{L}^{\frac{d}{1-\gamma}}(B_{r})}\right).
\end{align*}
Now if $\rho \leq r$, we can find a positive integer $N$ such that $\mu^{N+1} \leq \frac{\rho}{r} \leq \mu^N$. Therefore, applying the inequality above with $k = N$, we find
\begin{align*}
\|\grad u\|_{\underline{L}^2(B_{\rho})} & \leq \left(\frac{\mu^Nr}{\rho}\right)^{\frac{d}{2}}  \|\grad u\|_{\underline{L}^2(B_{\mu^N r})} \\
& \leq \left(\frac{\mu^Nr}{\rho}\right)^{\frac{d}{2}}  \mu^{N(\gamma-1)} \left(\|\grad u\|_{\underline{L}^2(B_{r})}+C\|{h}\|_{\underline{L}^{\frac{d}{1-\gamma}}(B_{r})} \right) \\
& \leq \mu^{-\frac{d}{2}} \left(\frac{\rho}{r}\right)^{\gamma - 1}\left(\|\grad u\|_{\underline{L}^2(B_{r})}+C\|{h}\|_{\underline{L}^{\frac{d}{1-\gamma}}(B_{r})} \right) .
\end{align*}
This implies the desired estimate on $\| \grad u\|_{\underline{L}^2(B_{\rho})}$. 
\end{proof}

Similar ideas can be used to prove a global version of \lref{interior-holder-est-quasi}. More specifically, let $U$ be a domain in $\R^d$ with $0 \in \overline{U}$, $a$ be a $\Lambda$-uniformly elliptic operator, and $u \in H^1_0(U)$ extended by $0$ outside of $U$. The following are global versions of \eref{a-harmonic-Cgamma} and \eref{a-harmonic-approx-Cgamma}. We write $B_r = B_r(0)$.

\begin{itemize}
    \item (global $C^{0,\gamma_0}$-estimate of $a$-harmonic functions) If $v \in H^1(B_r)$ is $a$-harmonic in $B_r \cap U$ and zero on $B_r \setminus U$ then
    \begin{equation}\label{e.a-harmonic-Cgamma-global}
        \| \grad v\|_{\underline{L}^2(B_{\rho} \cap U)} \leq C_1\left(\frac{\rho}{r}\right)^{\gamma_0-1}\| \grad v\|_{\underline{L}^2(B_{r} \cap U)} \qquad \hbox{ for all } r_0\leq \rho \leq r \leq r_1.
    \end{equation}
    \item (global $a$-harmonic approximation property) There is $g \in L^2(B_{r_1})$ and $\sigma>0$ so that for any $r_0 \leq \rho \leq r_1$
    \begin{equation}\label{e.a-harmonic-approx-Cgamma-global}
    \|\grad (u - u_{B_\rho \cap U})\|_{\underline{L}^2(B_\rho)} \leq C_2(\|{h}\|_{\underline{L}^2(B_\rho)}+\sigma \|\grad u\|_{\underline{L}^2(B_\rho)} ).
    \end{equation}
\end{itemize}

\begin{lemma}[global $C^{0,\gamma}$-estimate]\label{l.global-holder-est-quasi}
Let $a$, $U$, and $u$ satisfying the properties \eref{a-harmonic-Cgamma-global} and \eref{a-harmonic-approx-Cgamma-global} above. For all $0<\gamma < \gamma_0$ there is $\sigma_0(\gamma)>0$ sufficiently small and $C \geq 1$ so that if $\sigma \leq \sigma_0$ then for all $r_0 \leq \rho \leq r \leq r_1$
\[ \| \grad u\|_{\underline{L}^2(B_{\rho} \cap U)} \leq C\left(\frac{\rho}{r}\right)^{\gamma-1}\left(r^{\gamma-1}\sup_{\rho \leq s \leq r} \left[s^{1-\gamma}\|{h}\|_{\underline{L}^2(B_s \cap U)}\right]+ \| \grad u\|_{\underline{L}^2(B_r \cap U)}\right).\]
The constant $C$ depends on $C_1$, $C_2$, $\gamma_0 - \gamma$, $\Lambda$, and $d$.
\end{lemma}
The proof is exactly the same as the proof of \lref{interior-holder-est-quasi} above, replacing all instances of $B_\rho$ with $B_\rho \cap U$ etc.

\subsection{Some $L^2$-to-$L^\infty$ estimates} Finally, we discuss some useful applications of the H\"older estimates to proving $L^2$-to-$L^\infty$ estimates. Of course, this reverses the usual order in which one proves such estimates, as for instance in the De Giorgi approach, where $L^2$-to-$L^\infty$ estimates are important ingredients in the proof of H\"older estimates. Such arguments are typical in this Schauder-type harmonic replacement approach because the $L^\infty$ estimates are borderline and are trickier to prove via iteration directly.  We only show the interior estimates here, although analogous global estimates are also possible.

\begin{lemma}\label{l.L2-Linfty-quasi-0}
        Let $a$ be uniformly elliptic and $u \in H^1_0(B_r)$ satisfy the $a$-harmonic approximation property \eref{a-harmonic-approx-Cgamma} for all $0 \leq \rho \leq r$. There is $\sigma_0(d,\Lambda)>0$ sufficiently small so that if $\sigma \leq \sigma_0$ then for any $\gamma \in (0,1)$
    \[\|u\|_{L^\infty(B_{r/2})} \leq Cr(H+\|\grad u\|_{\underline{L}^2(B_r)}) \ \hbox{ where } \ H:=r^{\gamma-1}\sup_{x \in B_{r/2}}\sup_{0 <  s \leq r/2} \left[s^{1-\gamma}\|{h}\|_{\underline{L}^2(B_s(x))}\right]\]
    and $C$ depends on $C_1$, $C_2$, $\Lambda$, $\gamma>0$ and $d$. 
\end{lemma}

The Campanato-type norm of $g$ on the right hand side appears somewhat complicated, but the generality of the statement will be useful below. However, just to clarify the conclusions in a more useful statement we can also derive, using \rref{g-norm-explanation}:

\begin{corollary}\label{c.L2-Linfty-quasi}
    Let $a$ be uniformly elliptic and $u$ satisfy the $a$-harmonic approximation property \eref{a-harmonic-approx-Cgamma} for all $x \in \overline{B_r}$ and all $0 \leq \rho \leq r$ with ${h} \in L^p$ for some $p > d$. There is $\sigma_0(d,\Lambda)>0$ sufficiently small so that if $0 \leq \sigma \leq \sigma_0$, then
    \[\|u\|_{L^\infty(B_{r/2})} \leq Cr(\|{h}\|_{\underline{L}^p(B_r)}+\|\grad u\|_{\underline{L}^2(B_r)})\]
    where $C$ depends on $C_1$, $C_2$, $\Lambda$, and $p$. 
\end{corollary}

\begin{proof}[Proof of \lref{L2-Linfty-quasi-0}]
Fix $x \in B_{r/2}$ and let $\rho \leq r/2$. Denote $(u)_{\rho} := \dashint_{B_\rho(x)} u(y) \ dy$. We will prove a bound on $(u)_{B_\rho}$.
    \begin{align*}
        |(u)_\rho|  &\leq | (u)_\rho-(u)_{2\rho}| + |(u)_{2\rho}|\\
        &\leq \dashint_{B_\rho} |u(y) - (u)_{2\rho}| \ dy + |(u)_{2\rho}|\\
        & \leq 2^d\dashint_{B_{2\rho}} |u(y) - (u)_{2\rho}| \ dy + |(u)_{2\rho}|\\
        &\leq 2^d\|u-(u)_{2\rho}\|_{\underline{L}^2(B_{2\rho})}+ |(u)_{2\rho}|\\
        &\leq C(d)\rho\|\grad u\|_{\underline{L}^2(B_{2\rho})}+ |(u)_{2\rho}| \\
        &\leq C(d)2^{\gamma-1}\rho^{\gamma}r^{1-\gamma}\left(r^{\gamma-1}\sup_{0 <  s \leq r/2} \left[s^{1-\gamma}\|{h}\|_{\underline{L}^2(B_s(x))}\right]+\|\grad u\|_{\underline{L}^2(B_r)} \right)+ |(u)_{2\rho}| \\
        & = C(d,\gamma)r^{1-\gamma} \left(G+\|\grad u\|_{\underline{L}^2(B_r)} \right) \rho^{\gamma} + |(u)_{2\rho}|.
    \end{align*}
    where we applied \lref{interior-holder-est-quasi} in the second-to-last inequality (note the Campanato-type norms are increasing in $\gamma$). Iterating this inequality, we find
    \[|(u)_{2^{-k}r}| \leq C(d,\gamma)r^{1-\gamma}\left({H}+\|\grad u\|_{\underline{L}^2(B_r)}\right) \left(\sum_{j=1}^k (2^{-j}r)^\gamma \right) \leq C(d,\gamma)r({H}+\|\grad u\|_{\underline{L}^2(B_r)}).\]
    Then since we can argue centered at an arbitrary point $x \in B_{r/2}$ we obtain the $L^\infty$ bound. Note that the constants $C(d,\gamma) \to +\infty$ as $\gamma \searrow 0$.
\end{proof}
Another useful corollary is the a posteriori improvement of the $a$-harmonic approximation property \eref{a-harmonic-approx-Cgamma} from $L^2$ to $L^\infty$.
\begin{corollary}\label{c.L2-Linfty-quasi-harmonic-replacement}
    Let $a$ be uniformly elliptic and $u \in H^1(B_r)$ satisfy \eref{a-harmonic-approx-Cgamma} for all balls of radius $\rho$ centered at $x \in B_{r/2}$ and all $0 \leq \rho \leq r/2$. There is $1 \geq \sigma_0(d,\Lambda)>0$ sufficiently small so that if $0 \leq \sigma \leq \sigma_0$ then
    \[\|u-u_{B_r}\|_{L^\infty(B_{r/2})} \leq Cr(\|{h}\|_{L^\infty(B_r)}+\sigma \|\grad u\|_{\underline{L}^2(B_r)})\]
    where $C$ depends on $C_1$, $C_2$, $\Lambda$, and $d$.  
\end{corollary}
\begin{proof}
    Note that $w:= u-u_{B_r}$ has the property that $w - w_{B_\rho(x)} = u - u_{B_\rho(x)}$ for all $ x \in B_{r/2}$ and $0 < \rho \leq r/2$. Thus, $w$ satisfies a version of \eref{a-harmonic-approx-Cgamma}. Specifically, for all $x \in B_{r/2}$ and all $0 \leq \rho \leq r/2$
    \[\|\grad w - \grad w_{B_\rho(x)}\|_{\underline{L}^2(B_\rho(x))} \leq C_2(\|{h}\|_{\underline{L}^2(B_\rho(x))} + \sigma \|\grad u\|_{\underline{L}^2(B_\rho(x))}).\]  
    We now treat $\sigma \grad u$ as additional forcing type term, while \lref{interior-holder-est-quasi} implies that the Campanato-type norm is bounded
    \[r^{\gamma-1}\sup_{x \in B_{r/2}}\sup_{0 <  s \leq r/2} \left[s^{1-\gamma}\|\sigma \grad u\|_{\underline{L}^2(B_s(x))}\right] \leq C\sigma(\|{h}\|_{L^\infty(B_r)} + \|\grad u\|_{\underline{L}^2(B_r(x))}).\]
    We can thus apply \lref{L2-Linfty-quasi-0} to $w$ and obtain
    \[\|w\|_{L^\infty(B_{r/2})}\leq Cr(\|{h}\|_{L^\infty(B_r)}+\|\grad w\|_{\underline{L}^2(B_r)} + \sigma  \|\grad u\|_{\underline{L}^2(B_r(x))}  ).\]
    Applying \eref{a-harmonic-approx-Cgamma} again at scale $r$,
    \[\|\grad w\|_{\underline{L}^2(B_r)} \leq C_2(\|{h}\|_{\underline{L}^2(B_r)} + \sigma\|\grad u\|_{\underline{L}^2(B_r(x))})\] 
    and so finally, also bounding $\|{h}\|_{\underline{L}^2(B_r)} \leq \|{h}\|_{L^\infty(B_r)}$,
    \begin{equation}
        \|u-u_{B_r}\|_{L^\infty(B_{r/2})}\leq Cr(\|{h}\|_{L^\infty(B_r)}+\sigma\|\grad u\|_{\underline{L}^2(B_r(x))}).
    \end{equation}
\end{proof}

\subsection{Interior and global $W^{1,p}$ estimates}

In this section, we present proofs of the Calder\'on-Zygmund and Meyers estimates for quasi-minimizers following the approach of Caffarelli and Peral \cite{CaffarelliPeral}. Similar to the proof of H\"older regularity in the previous sections, the idea is to inherit regularity from an approximate problem. Specifically, the two key ingredients in the proof are (i) the gradient $L^p$ estimate for $a$-harmonic functions, (ii) a $a$-harmonic approximation property at all scales.  We will only consider the case of global (up to the boundary) estimates; the interior estimates follow as a special case.

More specifically let $U$ be a domain in $\R^d$ with $0 \in \overline{U}$, $a$ be a $\Lambda$-uniformly elliptic operator, and $u \in H^1_0(U)$ extended by $0$ outside of $U$. We add the following hypotheses, writing $B_r = B_r(0)$:
\begin{itemize}
    \item (global $L^{p_0}$-estimate of gradient of $a$-harmonic functions) There exists $p_0 \in [2,\infty)$ such that if $2B \subset B_r(0)$ is a ball and $v \in H^1(2B \cap U)$ is $a$-harmonic in $2B \cap U$ and zero on $2B \setminus U$, then
    \begin{equation}\label{e.a-harmonic-Lp0-global}
        \| \grad v\|_{\underline{L}^{p_0}(B)} \leq C_1\| \grad v\|_{\underline{L}^2(2B)}.
    \end{equation}
    \item (global $a$-harmonic approximation property) There is ${h} \in L^2(B_r)$ and $\sigma>0$ so that for any $B \subset B_r$\begin{equation}\label{e.harmonic-replace-Lp0-global}
        \|\grad (u - u_{B \cap U})\|_{\underline{L}^2(B )} \leq C_2(\|{h}\|_{\underline{L}^2(B )}+\sigma \|\grad u\|_{\underline{L}^2(B )} ).
    \end{equation}
\end{itemize}

\begin{theorem}\label{t.meyers-CZ-sigma-global}
    Let $a$, $U$, and $u$ satisfying the properties above. For all $2 \leq p < p_0$ there is $1 \geq \sigma_0(p)>0$ sufficiently small and $C \geq 1$ so that if $\sigma \leq \sigma_0$ then
\[ \| \grad u\|_{\underline{L}^p(B_{r/2})} \leq C(\|{h}\|_{\underline{L}^{p}(B_r )}+ \| \grad u\|_{\underline{L}^2(B_{r})}) .\]
The constant $C$ depends on $C_1$, $C_2$, $\frac{1}{p_0-2}$, $p_0 - p$, $\Lambda$, and $d$.  
\end{theorem}

\begin{remark}\label{CZ-estimate-example-scenarios}
    We list some scenarios under which the (global) gradient $L^{p_0}$ estimate \eref{a-harmonic-Lp0-global} holds. We will assume $U$ is a Lipschitz domain.    
    \begin{itemize}
        \item If $a(x)$ is uniformly continuous, then the Calder\'on-Zygmund estimates for divergence form equations (see \cite{Armstrong-Kuusi-Mourrat}*{Chapter 7} and \cite{Giaquinta-Martinazzi}*{Chapter 7}) implies \eref{a-harmonic-Lp0-global} holds for all $p_0 \in [2,\infty)$.
        \item If $a(x)$ is measurable and uniformly elliptic, then the classical Meyers estimate (see \cite{Armstrong-Kuusi-Mourrat}*{Chapter 7}) implies \eref{a-harmonic-Lp0-global} holds for some $p_0 \in (2,\infty)$.
    \end{itemize}
\end{remark}

\begin{remark}
    It is possible to generalize to the case when the $L^{p_0}$-estimate of $a$-harmonic functions holds above a minimal scale.  Specifically, letting $\eta_\rho$ be a standard family of mollifiers and $r_0$ be the minimal length scale, if $v$ is $a$-harmonic in $B_r$ then
    \[\| \grad (\eta_{\rho} \star v)\|_{\underline{L}^{p_0}(B_{r/2})} \leq C\| \grad v\|_{\underline{L}^2(B_{r})} \ \hbox{ for all } \ r_0 \leq \rho \leq r. \]
    This type of estimate is natural in the setting of homogenization theory, i.e. if $a$ is periodic, or random stationary with good mixing properties.  The result would be an analogous $\underline{L}^p$ estimate of quasi-minimizers above the minimal scale $r_0$.
\end{remark}
\begin{remark}\label{r.scaling-property-harmonic-replace}
    Note that the property \eref{harmonic-replace-Lp0-global} satisfies a convenient scaling property. If $u$ satisfies \eref{harmonic-replace-Lp0-global} and $\lambda>0$ then $\lambda u$ satisfied \eref{harmonic-replace-Lp0-global} with the same $\sigma$ and with $g \mapsto \lambda g$.
\end{remark}
We will argue via an iteration argument using harmonic approximation, the Hardy-Littlewood maximal function, and a covering argument. The arguments come from Caffarelli and Peral's proof of the Calderon-Zygmund estimates \cite{CaffarelliPeral}, and we also follow the presentation by Lihe Wang \cite{Wang}.

Let us first recall some basic concepts from real analysis in the setting of locally integrable functions $f : \R^d \to \R$; a standard reference is \cite{Stein-HarmonicAnalysis}.  The Hardy-Littlewood maximal function is defined as
\[(\mathcal{M}f)(x) := \sup_{\rho>0} \dashint_{B_\rho(x)} |f(y)| \ dy.\]
Also define the localized maximal function
\[\mathcal{M}_Vf := \mathcal{M}({\bf 1}_V f).\]
The maximal function is subadditive
\[\mathcal{M}(f_1+f_2) \leq \mathcal{M}f_1 + \mathcal{M}f_2.\]
We also recall the weak-type $1$-$1$ estimate
\begin{equation}\label{e.weak-type}
    |\{|\mathcal{M}f(x)| > \lambda\}| \leq C(d)\lambda^{-1} \|f\|_{L^1(\R^d)}
\end{equation}
and the strong $L^p$ to $L^p$ estimate
\begin{equation}\label{e.strong-type}
    \|\mathcal{M}f\|_{L^p(\R^d)} \leq C(p,d)\|f\|_{L^p(\R^d)} \ \hbox{ for any } \ p \in(1,\infty].
\end{equation}

We will estimate $L^p$ norms via the superlevel sets using the layer-cake decomposition
\begin{equation}\label{e.layer-cake-cts}
 \|f\|_{L^p(\R^d)}^p = \int_0^\infty p\lambda^{p-1}|\{x: |f(x)| \geq \lambda\}| \ d\lambda
\end{equation}
and its discrete form, for any geometric factor $\lambda_0>1$
\begin{equation}\label{e.layer-cake-discrete}
\|f\|_{L^p(\R^d)}^p \sim_{\lambda_0,p} \sum_{k=-\infty}^\infty \lambda_0^{kp}|\{x: |f(x)| \geq \lambda_0^k\}|.
\end{equation}
Note that when $f$ is supported in a region $U$ of finite measure, then all the negative order summands in the above geometric decomposition can be bounded by $|\{x: |f(x)| \geq \lambda_0^k\}| \leq |U|$, so
\begin{equation}\label{e.layer-cake-finite-support}
\|f\|_{L^p(\R^d)}^p \lesssim_{\lambda_0,p} |U|+\sum_{k=1}^\infty \lambda_0^{kp}|\{x: |f(x)| \geq \lambda_0^k\}|.
\end{equation}
We denote balls by $B$ and we call $\lambda B$ the ball with the same center as $B$ and $\lambda$ times the radius.

Below we will consider the maximal function localized in the main ball $\mathcal{M}_{B_r}$. Since that will consistently be the primary maximal function considered throughout the proofs below, we will drop the subscript and write $\mathcal{M}$.  Equivalently, one can consider that we have chosen to extend by zero both $|\grad u|^2$ and ${h}$ outside of $B_r$.

\begin{lemma}\label{l.cz-one-step}
Under the hypotheses of \tref{meyers-CZ-sigma-global}.  Let $4B \subset B_r$ and suppose that
        \[ \{ \mathcal{M} |\grad u|^2 \leq 1 \} \cap \{\mathcal{M}|{h}|^2 \leq \delta^2 \} \cap B \neq \emptyset\]
    then for all $\mu >\mu_0(n,C_1)>1$
    \[|\{\mathcal{M}|\grad u|^2 > \mu^2\} \cap B | \leq C_3\bigg[(\delta^2+\sigma^2)\mu^{-2}+(\mu_0/\mu)^{p_0}\bigg] |B |,\]
    with $C_3 \geq 1$ depending on $\Lambda$, $C_1$, $C_2$, $d$, and $\frac{1}{p_0-2} \in [0,\infty)$ but independent of $\sigma>0$. The term $(\mu_0/\mu)^{p_0}$ is interpreted as $0$ in this formula when $p_0 = +\infty$.
\end{lemma}
\begin{proof}
    By assumption there exists $x_0 \in B \cap U$ with
    \[(\mathcal{M}|\grad u|^2)(x_0) \leq 1 \ \hbox{ and } \ (\mathcal{M}|{h}|^2)(x_0) \leq \delta^2.\]
    In particular
    \[\|\grad u\|^2_{\underline{L}^2(4B)} \leq (5/4)^n \ \hbox{ and } \ \|{h}\|^2_{\underline{L}^2(4B)} \leq (5/4)^n\delta^2.\]
    Also we note that, for all $x \in B$ any ball $B_{\rho}(x)$ intersecting the complement of $4B$ must have radius at least $3$ times the radius of $B$, and therefore must contain the ball $B_{\rho/3}(x_0)$, and so
    \begin{align*}
        (\mathcal{M}|\grad u|^2)(x)  &\leq (\mathcal{M}_{4B}|\grad u|^2)(x) + 3^n(\mathcal{M}|\grad u|^2)(x_0) \\ &\leq (\mathcal{M}_{4B}|\grad u|^2)(x)+3^n.
    \end{align*}
      Thus for any $\mu^2 \geq 2 \cdot 3^n$ we have
    \[|\{ x \in B: \mathcal{M}|\grad u|^2 > \mu^2\}| \leq |\{ x \in B: \mathcal{M}_{4B}|\grad u|^2 > \mu^2/2\}|.\]
    So it suffices to bound the localized maximal function $(\mathcal{M}_{4B}|\grad u|^2)(x)$.
    
    Let $u_{4B\cap U}$ be the $a$-harmonic replacement as defined in \dref{harmonic-replacement-bdry}.  By the $a$-harmonic replacement property \eref{harmonic-replace-Lp0-global}
    \begin{equation}\label{e.cz-one-step-replacement-err}
        \|\grad u - \grad u_{4B \cap U}\|_{\underline{L}^2(4B)} \leq C_2(\|{h}\|_{\underline{L}^2(4B)}+\sigma \|\grad u\|_{\underline{L}^2(4B)}) \leq C(d)C_2(\delta+\sigma) .
    \end{equation}
    By the $W^{1,p_0}$ estimate of $a$-harmonic functions \eref{a-harmonic-Lp0-global}
    \begin{equation}\label{e.cz-one-step-replacement-Lp0}
        \|\grad u_{4B \cap U}\|_{\underline{L}^{p_0}(2B)} \leq C_1\|\grad u\|_{\underline{L}^2(4B)} \leq C_1C(d)
    \end{equation}
    Since
    \[|\grad u|^2 \leq 2 |\grad u_{4B \cap U}|^2+2|\grad u_{4B \cap U} - \grad u|^2\]
    using subadditivity of the maximal function
    \[\mathcal{M}_{4B}(|\grad u|^2) \leq 2\mathcal{M}_{4B}(|\grad u_{4B \cap U}|^2) + 2\mathcal{M}_{4B}(|\grad u_{4B \cap U} - \grad u|^2).\]
    So, when $(\mathcal{M}_{4B}|\grad u|^2)(x) \geq \mu^2/2$ then either $\mathcal{M}_{4B}(|\grad u_{4B \cap U}|^2)(x) \geq \mu^2/8$ or $\mathcal{M}_{4B}(|\grad u_{4B \cap U} - \grad u|^2)(x)\geq \mu^2 /8$, so 
    \[|\{\mathcal{M}_{4B}|\grad u|^2 \geq \mu^2/2\} \cap B | \leq |\{\mathcal{M}|\grad u_{4B \cap U}|^2 \geq \mu^2/8\} \cap B |+|\{\mathcal{M}(|\grad u_{4B \cap U} - \grad u|^2) \geq \mu^2/8\} \cap B |.\]
    Applying the weak-type $(1,1)$ estimate \eref{weak-type} and \eref{cz-one-step-replacement-err} we find
    \begin{align*}
        |\{x \in B: \mathcal{M}_{4B}(|\grad u_{4B \cap U} - \grad u|^2) \geq \mu^2/8\}  | &\leq C(d)\mu^{-2}\||\grad u_{4B \cap U} - \grad u|^2\|_{L^1( B )} \\
        &\leq C(d)C_2^2(\delta^2+\sigma^2)\mu^{-2}|B|.
    \end{align*}
    By Chebyshev's inequality, the strong $L^{p_0/2}$ bounds \eref{strong-type}, and \eref{cz-one-step-replacement-Lp0} respectively
    \begin{align*}
        |\{x \in B: \mathcal{M}_{4B}|\grad u_{4B \cap U}|^2 \geq \mu^2/8\}  | &\leq C(d)(2\sqrt{2}/\mu)^{p_0}\|\mathcal{M}_{4B}|\grad u_{4B \cap U}|^2\|^{p_0/2}_{L^{p_0/2}(B)} \\
        &\leq C(d,\tfrac{1}{p_0-2})(2\sqrt{2}/\mu)^{p_0}\|\grad u_{4B \cap U}\|^{p_0}_{L^{p_0}(4B)} \\
        &\leq (C(d,\tfrac{1}{p_0-2})C_1/\mu)^{p_0}|B|.
    \end{align*}
      Or in the case that $p_0 = +\infty$ then $\{\mathcal{M}|\grad u_{4B \cap U}|^2 \geq \mu^2/8\} \cap B  = \emptyset$ when $\mu > C(d) C_1$.  Thus overall we conclude that
      \[|\{ x \in B: \mathcal{M}|\grad u|^2 > \mu^2\}| \leq \left[C(d)C_2^2(\delta^2+\sigma^2)\mu^{-2}+(C(d)C_1/\mu)^{p_0}\right]|B|\]
      as long as $\mu > \mu_0 := C(d)(1+C_1)$ where the dimension constant $C(d)$ arises from the two requirements we put on $\mu$ during the course of the proof.
\end{proof}
This one-step bound is iterated using the following Vitali-type covering lemma. 
See \cite{Wang}*{Theorem 3} for the proof.
\begin{lemma}[Vitali-type lemma]\label{l.vitali-type}
Let $A \subset D \subset B_1$ be measurable sets with $|A| \leq \ep |B_1|$, for some $0 < \ep < 1$.  Suppose further that whenever $B_\rho(x)$ is a ball centered at a point $x \in B_1$ with $0 < \rho < 2$ and $|A \cap B_\rho(x)| \geq \ep |B_\rho|$ then $B_\rho(x) \subset D$. Then $|A| \leq 20^n\ep|D|$.
\end{lemma}
\begin{lemma}\label{l.CZ-iteration}
Under the hypotheses of \tref{meyers-CZ-sigma-global}, let $\mu > \mu_1(C_1,C_2,\Lambda,d,\frac{1}{p_0-2}) >1$ and assume that
\begin{equation}\label{e.CZ-iteration-hypothesis}
    |\{\mathcal{M}|\grad u|^2 > 1\}\cap B_{r/9}| \leq \ep|B_{r/9}| \ \hbox{ where } \ep:= C_3((\delta^2+\sigma^2)\mu^{-2}+\mu^{-p_0}).
\end{equation}
Then 
\begin{equation}\label{e.one-step-vitali-bound}
    |\{\mathcal{M}|\grad u|^2 > \mu^{2}\}\cap B_{r/9}| \leq 20^d\ep\left[|\{ \mathcal{M} |\grad u|^2 > 1 \}\cap B_{r/9}| +|\{\mathcal{M}|{h}|^2 > \delta^2 \}\cap B_{r/9}|\right]
\end{equation}
and furthermore for all integer $k \geq 1$
\begin{align*}
|\{\mathcal{M}|\grad u|^2 > \mu^{2k}\}\cap B_{r/9}| &\leq (20^d\ep)^k |\{\mathcal{M}|\grad u|^2 >1 \}\cap B_{r/9}|\\
&\quad +\sum_{i=1}^k(20^n\ep)^i|\{\mathcal{M}|{h}|^2 >\delta^2\mu^{2(k-i)}\}\cap B_{r/9}|.
\end{align*}
\end{lemma}
\begin{proof} Let $ \mu_1(C_1,C_2,\Lambda,d,\tfrac{1}{p_0-2})>1$ sufficiently large so that $C_3((\delta^2+\sigma^2)\mu_1^{-2}+\mu_1^{-p_0}) < 1$ (noting that $\delta, \sigma \leq 1$).  

First we establish \eref{one-step-vitali-bound}. Apply \lref{vitali-type} with sets
\[A = \{\mathcal{M}|\grad u|^2 > \mu^{2}\} \cap B_{r/9} \ \hbox{ and } \ D = \left(\{ \mathcal{M} |\grad u|^2 > 1 \} \cup \{\mathcal{M}|{h}|^2 > \delta^2 \}\right) \cap B_{r/9}\]
and $\ep := C_3((\delta^2+\sigma^2)\mu^{-2}+\mu^{-p_0})$, which satisfy the required property of \lref{vitali-type} due to \lref{cz-one-step}. Note that any ball $B_\rho(x)$ centered at $x \in B_{r/9}$ with $0< \rho < 2r/9$ has $B_{4\rho}(x) \subset B_{8r/9+r/9}(0) = B_r(0)$. This implies
\[|\{\mathcal{M}|\grad u|^2 > \mu^{2}\} \cap B_{r/9}| \leq 20^n\ep\left[|\{ \mathcal{M} |\grad u|^2 > 1 \}\cap B_{r/9}| +|\{\mathcal{M}|{h}|^2 > \delta^2 \}\cap B_{r/9}|\right].\]

Now let $k \geq 1$ and apply the one-step bound \eref{one-step-vitali-bound} with $u_k := \mu^{-k}u$ which satisfies property \eref{harmonic-replace-Lp0-global} with ${h}_k = \mu^{-k}{h}$ and the same $\sigma$. Also note that
\[|\{\mathcal{M}|\grad u_k|^2 > 1\}\cap B_{r/9}| \subset |\{\mathcal{M}|\grad u|^2 > 1\}\cap B_{r/9}| \leq \ep |B_{r/9}|\]
so the hypothesis of the previous paragraph still holds. Thus we can indeed apply \eref{one-step-vitali-bound} to $u_k$ and find
\[|\{\mathcal{M}|\grad u|^2 > \mu^{2(k+1)}\}\cap B_{r/9}| \leq 20^n\ep\left[|\{ \mathcal{M} |\grad u|^2 > \mu^{2k} \}\cap B_{r/9}| +|\{\mathcal{M}|{h}|^2 > \delta^2\mu^{2k} \}\cap B_{r/9}|\right].\]
Iterating this inductive estimate gives the claimed bound.

\end{proof}
Finally we return to the proof of \tref{meyers-CZ-sigma-global}.
\begin{proof}[Proof of \tref{meyers-CZ-sigma-global}]

Call $C_4 = 20^dC_3$. Let $p<p_0$, $\mu_2>\mu_1>1$ sufficiently large, depending on $p_0 - p$ and $C_4$, so that
\[C_4 \mu_2^{p-p_0} \leq \frac{1}{4}.\]
Then choose $\delta>0$ sufficiently small -- based on $\mu_2$, $C_4$, and $p_0$ -- so that
\[2C_4 \delta^2\mu_2^{p-2} \leq \frac{1}{4}.\]
Then choose $0 < \sigma \leq \sigma_0:=\delta$ and combine the previous inequalities to obtain
\begin{equation}\label{e.multiplier-leq-1/2}
    C_4 \mu_2^p((\sigma^2+\delta^2)\mu_2^{-2}+\mu_2^{-p_0}) \leq \frac{1}{2}.
\end{equation}

Now we rescale $u \mapsto \lambda u$, for $0<\lambda \sim \|\grad u\|_{\underline{L}^2(B_r)}^{-1}$ to be specified momentarily, in order to satisfy the hypothesis \eref{CZ-iteration-hypothesis} of \lref{CZ-iteration}. Note that $\bar{u}:=\lambda u$ still satisfies \eref{harmonic-replace-Lp0-global} with $\bar{h}:= \lambda {h}$. By Chebyshev and the strong-type $(2,2)$ bound
\[|\{\mathcal{M} |\grad (\lambda u)|^2 > 1\} \cap B_{r/9}| \leq \lambda^2\|\mathcal{M}\grad u\|_{\underline{L}^2(B_{r/9})}^2|B_{r/9}| \leq \lambda^2\|\grad u\|_{\underline{L}^2(B_{r})}^2|B_{r/9}|\]
so choosing
\[\lambda =\|\grad u\|_{\underline{L}^2(B_r)}^{-1}C_3\mu_2^p(\delta^2\mu_2^{-2}+\mu_2^{-p_0})\]
guarantees that $\bar{u}$ satisfies \eref{CZ-iteration-hypothesis}.  Note that $\lambda = c\|\grad u\|_{\underline{L}^2(B_r)}^{-1}$ where $c$ depends only on the universal parameters $C_1$, $C_2$, $\Lambda$ and $d$.

By the strong-type $(p,p)$ bound and using the layer-cake decomposition \eref{layer-cake-finite-support}
\begin{align*}
    \|\grad \bar{u}\|_{L^p(B_{r/9})}^p & = \||\grad \bar{u}|^2\|_{L^{p/2}(B_{r/9})}^{p/2} \\
    & \leq C\|\mathcal{M}|\grad \bar{u}|^2\|_{L^{p/2}(B_{r/9})}^{p/2} \\
    & \leq C(\mu_2) \sum_{k=1}^\infty \mu_0^{kp}|\{x \in B_{r/9}: \mathcal{M}|\grad \bar{u}|^2 > \mu_2^{2k}\}|.
\end{align*}
Then using \lref{CZ-iteration} and \eref{multiplier-leq-1/2}
\begin{align*}
& \sum_{k=1}^\infty \mu_2^{kp}|\{x \in B_{r/9}: \mathcal{M}|\grad \bar{u}|^2 > \mu_2^{2k}\}| \\
&\leq \sum_{k=1}^\infty\frac{1}{2^k}|\{\mathcal{M}|\grad \bar{u}|^2 >1 \}|  \\
&\quad+\sum_{k=1}^\infty\sum_{i=1}^k \mu_0^{p(k-i)}\mu_0^{pi}C_4^{i}(\delta^2\mu_2^{-2}+\mu_2^{-p_0})^i|\{\mathcal{M}|\bar{h}|^2 >\delta^2\mu_2^{2(k-i)}\}| \\
&\leq |B_r| + \sum_{i=1}^\infty\sum_{k=i}^\infty \frac{1}{2^i}\mu_0^{p(k-i)}|\{\mathcal{M}|\bar{h}|^2 >\delta^2\mu_2^{2(k-i)}\}| \\
&\leq |B_r| + \sum_{i=1}^\infty \frac{1}{2^i}C(\delta,\mu_2)\|\mathcal{M}|\bar{h}|^2\|_{L^{p/2}(B_r)}^{p/2}\\
&\leq |B_r| + C(p,\delta,\mu_2)\||\bar{h}|^2\|_{L^{p/2}(B_r)}^{p/2} = |B_r| + C\|\bar{h}\|_{L^p(B_r)}^p.
\end{align*}
Thus
\[\|\grad \bar{u}\|_{\underline{L}^p(B_{r/9})}^p \leq C(1+\| \bar{h}\|_{\underline{L}^{p}(B_r)}^p)\]
and
\[\|\grad u\|_{\underline{L}^p(B_{r/9})} = \lambda^{-1}\|\grad \bar{u}\|_{\underline{L}^p(B_r)} \leq C(\|\grad u\|_{\underline{L}^2(B_r)}+\|{h}\|_{\underline{L}^{p}(B_r)}).\]
The bound in $B_{r/2}$ instead of $B_{r/9}$ follows in a standard way by covering $B_{r/2}$ by $C(d)$ many balls of radius $r/18$.
\end{proof}

\bibliographystyle{amsplain}
\bibliography{two-phase-articles.bib}

\end{document}